\documentclass[11pt,a4paper]{article}
\usepackage[utf8]{inputenc}
\usepackage{amsmath}
\usepackage{amsfonts}
\usepackage{amssymb}
\usepackage{amsthm}
\usepackage[authoryear]{natbib}
\usepackage{enumerate}
\usepackage{color}
\usepackage{graphicx}
\usepackage[colorlinks, allcolors=blue]{hyperref}

\numberwithin{equation}{section}
\theoremstyle{plain}
\newtheorem{theorem}{Theorem}[section]
\newtheorem{corollary}[theorem]{Corollary}
\newtheorem{lemma}[theorem]{Lemma}
\newtheorem{assump}[theorem]{Assumption}
\newtheoremstyle{definition}{}{}{}{}{\bfseries}{}{ }{}
\theoremstyle{definition}

\newtheorem{remark}[theorem]{Remark}

\DeclareMathOperator{\Cov}{Cov}
\DeclareMathOperator{\sign}{sign}

\DeclareMathOperator*{\argmax}{\arg\!\max}
\DeclareMathOperator{\Varh}{Var}
\DeclareMathOperator{\Exp}{Exp}

\newcommand{\convd}{\overset{\mathcal{D}}{\Longrightarrow}}
\newcommand{\convp}{\overset{\mathbb{P}}{\longrightarrow}}
\newcommand{\convn}{\underset{n \to \infty}{\longrightarrow}}
\newcommand{\convdn}{\underset{d,n \to \infty}{\longrightarrow}}
\newcommand{\maxhd}{\max_{h=1}^{d}}

\newcommand{\maxhSdc}{\max_{h \in \Sd^c}}
\newcommand{\minhSdc}{\min_{h \in \Sd^c}}
\newcommand{\minhd}{\min_{h=1}^{d}}
\newcommand{\cusums}{\mathbb{U}_{n,h}(s)}
\newcommand{\cusumssq}{\mathbb{U}_{n,h}^2(s)}

\newcommand{\cusum}[1]{\mathbb{U}_{n,h}(#1)}
\newcommand{\Mint}{\mathbb{M}_{n,h}^2}
\newcommand{\hatMint}{\hat{\mathbb{M}}_{n,h}^2}
\newcommand{\hatth}{\hat{t}_h}
\newcommand{\hatqh}{\hat{q}_h}
\newcommand{\hatTnh}{\hat{T}_{n,h}}
\newcommand{\hatTnhb}{\hat{T}_{n,h}'}
\newcommand{\hatTnhdelta}{\hat{T}_{n,h}^{(\Delta)}}
\newcommand{\Tnh}{T_{n,h}}
\newcommand{\bt}{\textbf{t}}
\newcommand{\hatsigmah}{\hat{\sigma}_h}
\newcommand{\hatsigma}{\hat{\sigma}}
\newcommand{\sigmah}{\sigma_h}

\newcommand{\mean}[1]{\mathbb{E}\left[ #1 \right]}
\newcommand{\E}[1]{\mathbb{E}\left[ #1 \right]}
\newcommand{\R}{\mathbb{R}}
\newcommand{\F}{\mathcal{F}}
\newcommand{\Z}{\mathbb{Z}}
\newcommand{\N}{\mathbb{N}}
\newcommand{\Prob}{\mathbb{P}}
\newcommand{\B}{\mathcal{B}}
\newcommand{\A}{\mathcal{A}}

\newcommand{\Pro}{\mathcal{P}}
\newcommand{\Pb}[1]{\mathbb{P}\left(#1\right)}
\newcommand{\PbZ}[1]{\mathbb{P}_{\vert \mathcal{Z}_n}\left(#1 \right)}
\newcommand{\PbZL}[1]{\mathbb{P}_{\vert \mathcal{Z}_n}^{\mathcal{L}}\left(#1 \right)}

\newcommand{\floor}[1]{\lfloor #1 \rfloor}

\newcommand{\Var}[1]{\Varh\left(#1 \right) }

\newcommand{\tauhatth}{\hat{\tau}_h}
\newcommand{\tauth}{\tau(t_h)}
\newcommand{\limn}{\lim_{n \to \infty}}
\newcommand{\limd}{\lim_{d \to \infty}}
\newcommand{\limdn}{\lim_{d,n \to \infty}}
\newcommand{\dmuh}{\Delta\mu_h}
\newcommand{\absdmuh}{\vert \Delta\mu_h \vert}
\newcommand{\HNull}{H_{0,\Delta}}
\newcommand{\HA}{H_{A,\Delta}}
\newcommand{\TS}{\mathcal{T}_{d,n}}
\newcommand{\BSTS}{\mathcal{B}_{d,n}}
\newcommand{\setd}{\{1,\dots,d\}} 
\newcommand{\gua}{g_{1-\alpha}}

\newcommand{\Ld}{\mathcal{L}_d}
\newcommand{\hatdmuh}{\Delta\widehat{\mu}_h}
\newcommand{\Rd}{\mathcal{R}_d}
\newcommand{\hatRd}{\widehat{\mathcal{R}}_d}
\newcommand{\Sd}{\mathcal{S}_d}
\newcommand{\Id}{\mathcal{I}_d}
\newcommand{\Bd}{\mathcal{B}_d}
\newcommand{\Ed}{\mathcal{E}_d}
\newcommand{\Lhminus}{\widehat{L}_h^-}
\newcommand{\Lhplus}{\widehat{L}_h^+}
\newcommand{\Zn}{\mathcal{Z}_n}
\newcommand{\bea}{\begin{eqnarray*}}
\newcommand{\eea}{\end{eqnarray*}}
\newcommand{\be}{\begin{eqnarray}}
\newcommand{\ee}{\end{eqnarray}}
\title{Relevant change points in high dimensional time series}

\begin{document}
\author{
{\small Holger Dette, Josua G\"osmann} \\
{\small Ruhr-Universit\"at Bochum} \\
{\small Fakult\"at f\"ur Mathematik}\\
{\small 44780 Bochum, Germany} \\
{\small e-mail: holger.dette@rub.de}\\
{\small  josua.goesmann@rub.de}\\
}

\maketitle

\begin{abstract} 
This paper investigates the problem of detecting relevant change points in the mean vector, say $\mu_t =(\mu_{1,t},\ldots  ,\mu_{d,t})^T$ of a high dimensional time series $(Z_t)_{t\in \Z}$.
While the recent literature on testing for change points in this context considers hypotheses for the equality of the means $\mu_h^{(1)}$ and $\mu_h^{(2)}$ before and after the change points in the different components, we are interested in a null hypothesis of the form 
$$
H_0:  |\mu^{(1)}_{h} - \mu^{(2)}_{h}  | \leq \Delta_h  ~~~\mbox{ for all } ~~h=1,\ldots ,d
$$
where $\Delta_1, \ldots , \Delta_d$ are given thresholds for which a smaller difference of the means in the $h$-th component is considered to be non-relevant.
This formulation of the testing problem is motivated by the fact that in many applications a modification of the statistical analysis might not be necessary, if the differences between the parameters before and after the change points in the individual components are  small.   This problem is of particular relevance in high dimensional change point analysis, where a small change in only one component can yield a rejection by the classical procedure although all components change only in a non-relevant way.

We propose a new test for this problem based on the maximum of squared and integrated CUSUM statistics and investigate its properties 
as the sample size $n$ and the dimension $d$ both converge to infinity. In particular, using Gaussian approximations for the maximum of a large number of dependent random variables, we show that on certain points of the boundary of the null hypothesis
a standardised version of the maximum converges weakly to a Gumbel distribution.
This result is used to construct a consistent asymptotic level $\alpha$ test and a multiplier bootstrap procedure is proposed, which improves the finite sample performance of the test. The finite sample properties of the test are investigated by means of a simulation study and we also illustrate the new approach investigating data from hydrology.
\end{abstract}


\section{Introduction}  \label{sec1}
\def\theequation{1.\arabic{equation}}
\setcounter{equation}{0}
In the context of high dimensional time series it is typically unrealistic to assume stationarity. 
A simple form of non-stationarity, which is motivated by financial time series, where large panels of asset returns routinely display break points, is to assume structural breaks at different times (the change points) in the individual components.
One goal of statistical inference is to correctly estimate these ``change points'' such that the original data can be partitioned into shorter stationary segments.
This field is called change point analysis in the statistical literature and since the seminal papers of \cite{page1954,page1955} numerous authors have worked on the problem of detecting structural breaks or change points in various statistical models [see \cite{auehor2013} for a recent review]. 
There exists in particular an enormous amount of literature on testing for and estimating the location of a change in the mean vector  $\mu_t =(\mu_{1,t}, \ldots , \mu_{d,t})^T =  \E{Z_t} $ of a multivariate time series $(Z_t)^n_{t=1}$ [see \cite{chu1996}, \cite{horkokste1999},  
\cite{kirmusomb2015} among others]. A common feature in these references 
consists in the fact that the dimension, say $d$, of the time series is fixed.
High dimensional change point problems, where the dimension $d$ increases with sample size have only been recently considered in the literature [see \cite{bai2010}, \cite{zhangetal2010}, \cite{horvhusk2012}  and \cite{enikharch2014}, \cite{jirak:2015}, \cite{chofryz2015} and \cite{wangsam2016} among others].
Some of this work uses information across the coordinates in order to detect smaller changes than could be observed in any individual component series.  

In the simplest case of one structural break in each component many authors attack the problem of detecting the change point by means of hypothesis testing. For example, \cite{jirak:2015} investigates the hypothesis of no structural break in a high-dimensional time series by testing the hypotheses
\begin{align}
\label{null}
H_0: \mu_{1,h} = \mu_{2,h} = 
\ldots =\mu_{n,h} \text{ for all }  h=1, \ldots ,d~, 
\end{align} 
where $\mu_{t,h}$ denotes the $h$-th component of the mean vector $\mu_{t} $ of the random variable $Z_t \ (t=1,\dots,n)$.
The alternative is then formulated (in the simplest case of one structural break) as
\begin{align}\label{alt}
\begin{split}
H_1: \mu^{(1)}_h 
&=  \mu_{1,h}  = \mu_{2,h}  = \dots = \mu_{k_h,h}  \\
&  \neq  \ \mu_{k_h+1,h}   = \mu_{k_h+2,h} = \dots = \mu_{n,h} ~= ~\mu^{(2)}_h \\
\end{split}
\end{align} 
for at least one  $h \in \{1,\ldots , d\}$, where $k_h \in \{ 1,\dots,n \}$ denotes the (unknown) location of the change point in the $h$-th component. 

While - even under sparsity assumptions - the detection of small changes in each component is a very challenging problem, a modification of the statistical analysis (such as prediction) might not be necessary if the actual size of change is small.
For example, in risk management situations, one is interested in fitting a suitable model for forecasting Value at Risk from 
data after the last change point [see e.g.\ \cite{wied:2013}], but  in practice, small changes in the parameter are perhaps not very interesting because they do not yield large changes in the Value at Risk. 
The forecasting quality might only improve slightly, but this benefit could be negatively overcompensated by transaction costs, in particular for high-dimensional portfolios.
Moreover, even, if the null hypothesis \eqref{null} is not rejected, it is difficult to quantify the statistical uncertainty for the subsequent statistical analysis (conducted under the assumption of stationarity), as there is no control about the type II error in this case.

The present work is motivated by these observations and proposes a test for the null hypothesis of \emph{no relevant change point} in a high dimensional context, that is
\begin{align}\label{relevant}
\HNull: | \mu^{(1)}_h   - \mu^{(2)}_h   | &\leq \Delta_h \text{ for all }  h=1,\dots,d \\ 
\label{althyp} \mbox{versus} \;\;\;  \HA: | \mu^{(1)}_h   - \mu^{(2)}_h  | &> \Delta_h  \text{ for at least one  }   h \in \{1,\ldots , d \}~,
\end{align}
where $\mu^{(1)}_h $ and $\mu^{(2)}_h$ are the parameters before and after the change point in the $h$-th component and $\Delta_1, \ldots , \Delta_d$ are given thresholds for which a smaller difference of the means in the $h$-th component is considered as non-relevant.

The problem of testing for a relevant difference between (one dimensional) means of two samples has been discussed by numerous authors mainly in the field of biostatistics (see \cite{wellek2010testing} for a recent review).
In particular testing relevant hypotheses avoids the consistency problem as mentioned in \cite{berkson1938}, that is: \textit{Any consistent test will detect any arbitrary small change in the parameters if the sample size is sufficiently large.}
In the finite dimensional case \cite{dettewied2016} investigated relevant hypotheses in change point analysis for general parameters using an $L^2$-Norm. The present paper differs from their approach in several perspectives. 
First, we consider the high-dimensional setting, where the dimension increases with sample size.
Second, relevant changes in different components are allowed to occur at different locations, and we identify corresponding locations and components by our approach.
Third we develop a bootstrap procedure in order to obtain a reasonable approximation of the nominal level in high dimensions.

The alternative approach requires the specification of the thresholds $\Delta_h >0 $, and this has to be carefully discussed and depends on the specific application. 
We also note that the hypotheses \eqref{relevant} contain the classical hypotheses \eqref{null}, which are obtained by simply choosing $\Delta_h = 0 $ for all $h=1, \ldots , d$.
Nevertheless we argue that from a practical point of view it might be very reasonable to think about this choice more carefully and to define the size of change in which one is really interested. 
In particular it is often known that  $| \mu^{(1)}_h   - \mu^{(2)}_h  | \neq 0 $ although one is testing ``classical hypotheses'' of the form \eqref{null} and \eqref{alt}.
Moreover, a decision of no relevant structural break at a controlled type I error can be easily achieved by interchanging the null hypothesis and alternative in \eqref{relevant}, i.e. considering the hypotheses
\begin{equation}\label{equiv}
\begin{split}
\widetilde{H}_{0,\Delta}: | \mu^{(1)}_h   - \mu^{(2)}_h  | &> \Delta_h \text{ for at least one }  h \in \{1,\ldots , d \} \\ 
\mbox{versus}\;\;\; \widetilde{H}_{A,\Delta} : | \mu^{(1)}_h   - \mu^{(2)}_h   | &\leq \Delta_h  \text{ for all } 
h=1\,\dots,d ~.
\end{split}
\end{equation}
An obvious idea to address this problem is multiple testing using the univariate tests in \cite{dettewied2016} with a Bonferroni correction.
However, in the high-dimensional context such an approach may accumulate too many errors in the approximation of the nominal level by the individual tests.
As a consequence the resulting Bonferroni test for the hypotheses in \eqref{equiv} offers a worse approximation of the nominal level compared to our approach (corresponding simulation results are available from the authors).

In this paper we use a different approach to test the hypotheses of a relevant structural break in any of the components of a high dimensional time series.
The basic ideas are explained in Section \ref{sec2} (without going into any technical details), where we propose to calculate for any component the integral of the squared CUSUM process and reject the null hypotheses whenever the maximum of these integrals (calculated with respect to all components) is large.
In order to obtain critical values for this test we derive in Section \ref{sec3} the asymptotic distribution of an appropriately standardized version of the maximum as the sample size and the dimension converge to infinity.
We also provide several auxiliary results, which are of own interest, and investigate the case where the maximum is only calculated over a subset of the components.
These results are then used in Section \ref{sec4} to prove that the proposed test yields to a valid statistical inference, i.e. it is a consistent and asymptotic level $\alpha$ test. 
It turns out that - in contrast to the classical change point problem - the analysis of the test for no relevant structural breaks is substantially harder as the null hypothesis does not correspond to a stationary process (non-relevant changes in the means are allowed).  
Section \ref{sec5} is devoted to the investigation of a multiplier block bootstrap procedure.
In particular we prove that the quantiles generated by this resampling method also yield to a consistent asymptotic level 
$\alpha$ test. 
The finite sample properties of the new test are investigated in Section \ref{sec6}, where we also illustrate our approach analysing a data example from hydrology.
Finally some of the technical details are deferred to the appendix.

\section{Relevant changes in high dimensions - basic principles}
\label{sec2}
\def\theequation{2.\arabic{equation}}
\setcounter{equation}{0}

In this Section we explain the basic idea of our approach to test for a relevant change in at least one component of the mean vector of a high dimensional time series.
For the sake of a transparent representation we try to avoid technical details at this stage and refer to the subsequent sections, where we present the basic assumptions and mathematical details establishing the validity of the proposed method. 

Throughout this paper we consider an array of real valued random variables $\{Z_{j,h} \}_{j\in \mathbb{Z}, h \in \mathbb{N}}$ such that 
\begin{align}\label{eq:model}
Z_{j,h} = \mu_{j,h} + X_{j,h}~,
\end{align}
where $\mu_{j,h} \in \R$ for all $j\in \mathbb{Z},h \in \mathbb{N}$ and $\{ X_{j,h} \}_{j\in \mathbb{Z}, h \in \mathbb{N}}$ denotes an array of centered and real valued random variables, which implies $\mu_{j,h} = \mean{ Z_{j,h}}$ for all ${j\in \mathbb{Z}, h \in \mathbb{N}}$. 
It follows from the  assumptions made in Section \ref{sec3} that for each fixed $d \in \mathbb{N}$ the time series
\begin{align} \label{statser}
\{ (X_{j,1},\dots,X_{j,d})^T\}_{j \in \mathbb{Z}}
\end{align}
is stationary.
Suppose that 
$$Z_1=(Z_{1,1},\dots,Z_{1,d})^T,\dots, Z_{n}=(Z_{n,1},\dots,Z_{n,d})^T \in \mathbb{R}^d$$ 
are $d$-dimensional observations from the array  $\{Z_{j,h}\}_{j \in \mathbb{Z}, h \in \mathbb{N}}$ and assume that for each component $h \in \{1,\ldots ,d\} $ there exists an unknown constant $t_h \in (0,1)$, such that
\begin{align}
\begin{split}
\mu_h^{(1)} &= \mu_{j,h} ~,~~j=1, \ldots , \floor{nt_h}~,\\
\mu_h^{(2)} &= \mu_{j,h} ~,~~j= \floor{nt_h}+1 , \ldots, n~,
\end{split}
\end{align}
where  $\floor{x}=\sup\{z \in \mathbb{Z}\,\vert\, z \leq x\}$ denotes the largest integer smaller or equal than $x$.
In this case the random variables $\{ Z_{j,h} \}_{h=1,\ldots ,d; j=1, \ldots ,n}$ also depend on $n$, i.e. they form a triangular array, but for the sake of readability, we will suppress this dependence in our notation.

We define $\dmuh=\mu_h^{(1)} -\mu_h^{(2)}$ as the unknown difference between the means in the  $h$-th component before and after the change point $t_h$.
Note, that in the case $\dmuh \not = 0$ the actual location $k_h=\floor{nt_h}$ of the change point depends on the sample size $n$, which is a common assumption in the 
literature on change point problems to perform asymptotic inference.
It simply ensures that the number of observations before and after the change point is growing proportional to $n$. 
For each $h $ with $\dmuh = 0$  the observable process $\{Z_{j,h}\}_{j \in \mathbb{Z}} $ is stationary and to avoid misunderstandings we set $t_h = 1/2$, whenever  $\dmuh = 0$.
The reader should notice that in this case the actual value is of no interest in the following discussion.  
With this notation the hypotheses in \eqref{relevant} can be rewritten as 
\begin{align}\label{hnull}
\HNull: \absdmuh &\leq \Delta_h \;\text{for all}\; h \in \setd \\
\text{versus }  ~~
\label{halt}
\HA: \absdmuh &> \Delta_h \;\text{for some}\; h \in \setd~.
\end{align}
To develop an appropriate test statistic for these hypotheses we rely on the widely used concept of CUSUM-statistics.
For each component $h \in \setd$ we consider the corresponding CUSUM-process for the $h$-th component defined by
\begin{align}\label{eq:cusum}
\begin{split}
\cusums
&= \frac{1}{n}\sum_{j=1}^{\floor{ns}} Z_{j,h} - \frac{\floor{ns}}{n^2}\sum_{j=1}^n Z_{j,h} \\
&= \frac{n-\floor{ns}}{n^2}\sum_{j=1}^{\floor{ns}} Z_{j,h} - \frac{\floor{ns}}{n^2}\sum_{j=\floor{ns}+1}^n Z_{j,h}~.
\end{split}
\end{align}
Under the assumptions stated in Section \ref{sec3} it is shown in Section \ref{sec71} of the appendix that 
\begin{align} \label{motiv}
\frac{3}{(t_h(1-t_h))^2}\mean{ \int_0^1 \cusumssq ds}
=\dmuh^2 + o(1)
\end{align}
as $n \to \infty$ and therefore our considerations will be based on the statistic
\begin{align}\label{eq:Mint}
\hatMint 
=\frac{3}{(\hatth(1-\hatth))^2} \int_0^1 \cusumssq ds~,
\end{align}
where $\hatth$ denotes an appropriate estimator for the unknown location $t_h$ of the structural break, that will be precisely defined in Section \ref{sec2s}.
The null hypothesis 
\begin{align}\label{rel1d}
H_{0, h}:  \absdmuh \leq \Delta_h 
\end{align}
of no relevant change in the $h$-th component will be rejected for large values of $\hatMint $, and in order to determine a critical value we introduce the normalization
\begin{align}\label{eq:tnhdelta}
\hatTnhdelta = \frac{\sqrt{n}}{\tauhatth \hatsigmah \Delta_h} \big( \ \hatMint - \Delta_h^2 \ \big)~,
\end{align}
where $\hatsigmah$ denotes an estimator for the unknown long-run variance (see Section \ref{sec2s} for a precise definition),
and the quantity $\tauhatth$ is a function of the estimate of the change point defined by 
\begin{align} \label{tauh}
\tauhatth=\tau(\hatth)
:= \frac{2\sqrt{1+2\hatth(1-\hatth)}}{\sqrt{5}\hatth(1-\hatth)}~,
\end{align}
which  arises due to the integral in equation \eqref{eq:Mint}.
For the case $\Delta_h = \dmuh$, it will be shown in Section \ref{sec732} of the appendix that for a fixed component $h \in \setd$
\begin{align}
\label{weak1}
\hatTnhdelta \convd \mathcal{N}(0,1)\;\;\;\;\;\;\;\text{as}\;n\to\infty~,
\end{align}
where the symbol $\convd$ represents weak convergence of a random variable. 
Moreover monotonicity arguments show that the test, which rejects the null hypothesis \eqref{rel1d} of a relevant change point in the mean of the $h$-th component, whenever $\hatTnhdelta $ exceeds the $(1-\alpha)$-quantile of the standard normal distribution, is a consistent asymptotic level $\alpha$ test. 

In order to construct a test for the hypotheses $\HNull$ versus $\HA$ in \eqref{hnull} and \eqref{halt} of a relevant change point in at least one of the components of a high dimensional time series we  propose a simultaneous test, which rejects the null hypothesis for large values of the statistic 
$$
\maxhd \hatTnhdelta~.
$$
Note that a similar approach has been investigated by \cite{jirak:2015}, who considered the ``classical'' change point problem in high dimension, that is
\begin{align}\label{class}
\begin{split}
 H_{\text{0,class}}  &: \absdmuh = 0 \;\text{for all}\; h \in \setd  \\ 
 \mbox{  versus }~~
H_{\text{A,class}} &: \absdmuh > 0 \;\text{for some}\; h \in \setd~.
\end{split}
\end{align}
In this case the weak convergence \eqref{weak1} does not hold (in fact $\hatTnhdelta = o_\Prob (1)$ under $H_{\text{0,class}} $) and a different statistic has to be considered.
 
As it is well known that the (adjusted) maximum of standard normal distributed random variables converges weakly to a Gumbel distribution if they exhibit an appropriate dependence structure [see for example \cite{berman:1964}], it is reasonable to consider the following 
\begin{align}\label{eq:TS}
\TS =  a_d \Big( \maxhd \hatTnhdelta - b_d \Big)
\end{align}
for the high-dimensional change point problem \eqref{hnull}, where $a_d$ and $b_d$ denote appropriate sequences, such that the left-hand side converges weakly at specific points of the ``boundary'' of the parameter space corresponding to the null hypothesis.  

\begin{figure}[h]
\centering
\includegraphics[width=9cm,height=6cm]{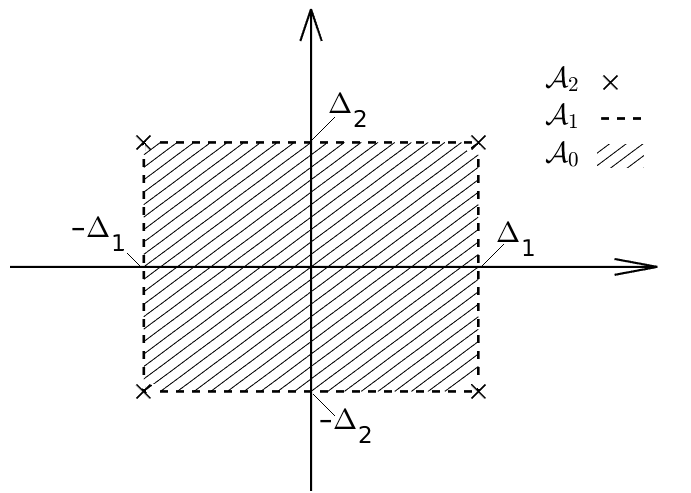}
\caption{\it Decomposition of the parameter space corresponding to the null hypothesis \eqref{hnull} into the sets $\A_0, \A_1$ and $\A_2$  for the case $d=2$. 
\label{fig:hypotheses}} 
\end{figure}

To make these arguments more precise, define the sequences
\begin{align}
\label{ad} a_d &= \sqrt{2\log d}~,\\
\label{bd} b_d &=  a_d - \log(4\pi \log d )/2a_d~, 
\end{align}
and note that the parameter spaces corresponding to the null hypothesis \eqref{hnull} and the alternative \eqref{halt} are given by 
$$
{\cal H}_0= \Big\{ (x,y) \in \R^d \times \R^d \big| ~ |x_h-y_h| \leq  \Delta_h ~~\forall \;  h \in \setd \Big\}
$$
and $ {\cal H}_A  = \R^{2d} \setminus {\cal H}_0$, respectively (here $x= (x_1,\ldots ,x_d)^T$, $y= (y_1,\ldots ,y_d)^T)$
denote $d$-dimensional vectors).
Define for $k=0, \ldots , d$ the ``$(d-k)$-dimensional boundary of the hypotheses \eqref{hnull} and \eqref{halt}'' by  
\begin{align}\label{kbound}
\begin{split}
\A_{k} = \Big\{ (x,y) \in &\R^d \times \R^d   \big| ~ |x_h-y_h| \leq  \Delta_h \,\forall\,  h \\ 
&\text{ with equality for precisely } k \text{ components} \Big\}
\subseteq {\cal H}_0 ~.
\end{split}
\end{align} 
For the case $d=2$, we illustrate this decomposition of the null hypothesis parameter space in Figure \ref{fig:hypotheses}.
In fact, a large part of this paper is devoted to prove the following statements (under appropriate assumptions - see Theorem \ref{thm:level}
in Section \ref{sec4}): 
\begin{itemize} 
\item[(1)]
If $ (\mu^{(1)} , \mu^{(2)} ) \in \A_d$, the weak convergence
\begin{align*}
\TS \convd G\;
\end{align*}
holds as both $n,d\to\infty$, where $G$ denotes a Gumbel distribution, i.e. 
\begin{align}\label{gumbeldef}
\Pb { G \leq x} = e^{-e^{-x}}~.
\end{align}
\item[(2)] If $ ( \mu^{(1)}, \mu^{(2)} ) \in \A_{k}$ and there exists a constant C, such that $| \mu_h^{(1)} - \mu_h^{(2)} | \leq C < \Delta_h $, whenever $| \mu_h^{(1)} - \mu_h^{(2)} | \neq \Delta_h$ and $\lim_{d \to \infty} k/d =c \in (0,1]$,  we have
$$\TS \convd G + \log c $$
as $n,d\to\infty$~.
\item[(3)] If $( \mu^{(1)}, \mu^{(2)}) \in \A_{k}$ and $\lim_{d \to \infty} k/d = 0 $ we have
$$\TS \convd  - \infty$$
as $n,d\to\infty$~.
\end{itemize} 
Let $\gua$ denote the $(1-\alpha)$ quantile of the Gumbel distribution, then it follows from these
considerations that the test which rejects $\HNull$ in favour of the alternative hypothesis $\HA$, whenever 
\begin{align}
\label{asymtest}
\TS > \gua ~,
\end{align}
is an asymptotic level $\alpha$ test.
More precisely, it follows (under appropriate assumptions stated below) that
\begin{equation} \label{heur} 
\mathbb{P}_{H_{0,\Delta}} \big ( \TS > \gua  \big)  \longrightarrow \begin{cases}  \alpha  &    \mbox{  if  }  ( \mu^{(1)}, \mu^{(2)}) \in \A_{d} \cr
 \alpha^\prime   &  \mbox{  if  }  ( \mu^{(1)}, \mu^{(2)}) \in \A_{k},\; \lim_{d \to \infty} k/d = c \in (0,1]  \cr
0  &  \mbox{  if  }  ( \mu^{(1)}, \mu^{(2)}) \in \A_{k},\; \lim_{d \to \infty} k/d = 0 \cr
\end{cases}
\end{equation}
as ${n,d \to \infty} $, where $  0 \leq  \alpha^\prime  \leq  \alpha$.
Additionally the test is consistent (see Theorem \ref{thm:consistent}).
In the following sections we will make these arguments more rigorous.
Moreover, in order to improve the finite sample approximation of the nominal level we also introduce a multiplier bootstrap procedure and prove its consistency in Section 5.

\section{Asymptotic properties}\label{sec3}
\def\theequation{3.\arabic{equation}}
\setcounter{equation}{0}

In this Section we provide the theoretical background for the test suggested in Section \ref{sec2}.
We begin introducing some notations and assumptions. After stating the main assumptions we provide in Section \ref{subsec:mainresults} the asymptotic theory for an analogue of the statistic $\TS$ defined in \eqref{eq:TS}, where the centering in \eqref{eq:tnhdelta} is performed by the ''true'' squared differences  $|\dmuh |^2$ and the estimates of the variances $\sigma_h$ and the locations of the change points $t_h$ have been replaced by their ''true'' counterparts.  
In Section \ref{sec2s} we introduce estimators for the locations of the structural breaks (which may occur at a different 
location in each component) and investigate their consistency properties. These are then used to define appropriate variance estimators (note that variances have to be estimated in the case of changes in the mean).
Finally, we consider in Section \ref{sec34} the asymptotic distribution of the maximum of the statistics \eqref{eq:tnhdelta} where again centering is performed by $|\dmuh |^2$ instead of $\Delta^2_h$.
These results will then be used in the subsequent Section \ref{sec4} to investigate the statistical properties of the test \eqref{asymtest}.
 
Throughout this paper we will use the following notation and symbols.
Let $x \wedge y$ and $x \vee y$ define the minimum and maximum of two real numbers $x$ and $y$, respectively. 
For an appropriately integrable random variable $Y$ and $q\geq 1$, let $\|Y\|_q = \mean{\vert Y \vert^q}^{1/q}$ denote the $\mathbb{L}^q$-norm.  
By the symbols $\convd$ and $\convp $ we denote weak convergence and  convergence in probability, respectively.
Moreover, we use the notation $x_n \lesssim y_n$, whenever the inequality  $x_n \leq C\cdot y_n$, holds for some constant $C>0$ which does not depend on the sample size and dimension and whose actual value is of no further interest. 
Due to its frequent appearance, $G$ will always represent a (standard) Gumbel distribution defined by \eqref{gumbeldef}.
In the high dimensional setup the dimension $d$ converges to infinity with $n \to \infty$~.

Recall the definition of model \eqref{eq:model} and assume throughout this paper that the time series $\{ X_{j,h} \}_{j\in \mathbb{Z}, h \in \mathbb{N}}$  
forms a physical system [see e.g. \cite{wu:2005}], that is
\begin{align}\label{eq:Xphysical}
X_{j,h} = g_h(\varepsilon_j,\varepsilon_{j-1},\dots)~,
\end{align}
where $\{ \varepsilon_j \}_{j \in \mathbb{Z}} $ is a sequence of i.i.d. random variables with values in some measure space $\mathbb{S}$ such that for each $h\in \N$ the function  $g_h : \mathbb{S}^{\mathbb{N}} \to \R $ is measurable. 
Note that it follows from \eqref{eq:Xphysical} that the time series defined in \eqref{statser} is stationary.
Let $\varepsilon_0'$ be an independent copy of $\varepsilon_0$ and define 
\begin{align}
X_{j,h}' = g_h(\varepsilon_j,\varepsilon_{j-1},\dots,\varepsilon_1, \varepsilon_0',\varepsilon_{-1},\dots)~.
\end{align}
The distance between $X_{j,h}$ and its counterpart $X_{j,h}'$ is used to quantify the (temporal) dependence of the physical system, 
and for this purpose we introduce the coefficients
\begin{align}\label{eq:physicalcoeff}
\vartheta_{j,h,p} := \Vert X_{j,h}- X_{j,h}' \Vert_p~, 
\end{align}
which measure the influence of $\varepsilon_0$ on the random variable $X_{j,h}$.
Let us  also define some additional parameters. For $h,i \in \mathbb{N}$
\begin{align*}
\phi_{h,i}(j) = \Cov(X_{0,h},X_{j,i}) = \Cov(Z_{0,h},Z_{j,i}) 
\end{align*}
denotes the covariance function of the $h$-th and $i$-th component at lag $j$.
Accordingly, the autocovariance function for the $h$-th component is given by $\phi_{h}(j) = \phi_{h,h}(j)=\Cov(X_{0,h},X_{j,h})$ and we obtain the well-known representations \begin{align} \label{sigmah}
\gamma_{h,i} = \sum_{j \in \mathbb{Z}} \phi_{h,i}(j)
\;\;\;
\text{and}
\;\;\;
\sigma_{h}^2 = \sum_{j \in \mathbb{Z}} \phi_h(j)
\end{align}
for the long-run covariance and long-run variance, respectively. 
If we have furthermore $\sigma_h, \sigma_i>0$, we can additionally define the long-run correlations by 
\begin{align}\label{longrun}
\rho_{h,i} &= \dfrac{\gamma_{h,i}}{\sigma_h\sigma_i}  
\end{align} 
and it will be crucial for our work, that these quantities become sufficiently small with an increasing temporal distance $\vert h - i \vert$.
This will be precisely formulated in the following section.

\subsection{Assumptions}\label{subsec:assump}

Operating in a high-dimensional framework usually needs stronger assumptions than those for the finite-dimensional case.
Mainly, we need uniform dependence and moment conditions among all components to exclude extreme cases and to ensure, that the unknown parameters can be estimated accurately.
In the high-dimensional setup considered in this paper the number of parameters $t_h$, $\sigma_h$ grows together with the dimension $d$, since even under the null hypothesis of no relevant change points each component exhibits its own variance and change point structure.
The precise assumptions made in this paper are the following.

\begin{assump}[temporal assumptions]
\label{a1}
Suppose that there exist constants $\delta \in (0,1)$, $\sigma_+ > 0$ and $\sigma_{-}>0$~,
such that for some $p>4$ the physical dependence coefficients
$\vartheta_{j,h,p}$ and long run variances $\sigma_{h}$ defined in
\eqref{eq:physicalcoeff} and \eqref{sigmah}, respectively, satisfy
\begin{enumerate}
\item[({T}1)] $\sup_{h \in \mathbb{N}} \vartheta_{j,h,p} \lesssim
\delta^{j}$~,
\end{enumerate}
\begin{enumerate}
\item[({T}2)] $\sigma_{-} 
\leq \inf_{h \in \mathbb{N}} \sigma_{h}
\leq \sup_{h \in \mathbb{N}} \sigma_{h}
\leq \sigma_{+}~.$
\end{enumerate}
\end{assump}

\begin{assump}[spatial assumptions]\label{a2}
The dimension $d$ increases with the sample size at a polynomial rate, i.e. we assume that for constants $C_1$ and $D$
\begin{enumerate}[(S1)]
\item $d = C_1n^D$~,
\end{enumerate}
where the exponent $D$ satisfies
\begin{enumerate}[({S}2)]
\item $0 < D < \min\{ p/4 - 1,~p/4 - B(p/2+1)-1\}$
\end{enumerate}
and $p>4$ refers to the $\mathbb{L}^p$-norm $\Vert \cdot \Vert_p$ used to measure the physical dependence in Assumption \ref{a1}.
Here $B \in (0,1/2)$ denotes a constant used to control the bandwidth of a variance estimator, that will be defined in Section \ref{sec2s}.
Further we assume for the long-run correlations in \eqref{longrun}
\begin{enumerate}[({S}3)]
\item $\sup_{\vert h - i \vert > 1}\vert \rho_{h,i}\vert \leq \rho_+ < 1$~,
\end{enumerate}
\begin{enumerate}[({S}4)]
\item $\vert \rho_{h,i} \vert \leq C_2 \log ( \vert h-i \vert + 2)^{-2 - \eta}$~,
\end{enumerate}
where $\rho_+>0$, $\eta>0$ and $C_2>0$ denote global constants.
\end{assump}

\begin{assump}[moment assumptions]\label{a3}
Suppose, that there exists a positive sequence $(M_d)_{d\in \N}$ and constants $C_3>1$ and $C_4>0$, such that
\begin{enumerate}
\item[(M{1})]  $\maxhd \mean{ \exp(\vert X_{1,h} \vert/ M_d)} \leq C_3$~,
\vspace{0.2cm}
\item[(M{2})] $M_d \leq C_4 n^m$ \;\;\text{with}\;\; $m<3/8$~.
\end{enumerate}
\end{assump}

\begin{assump}[location of the change points]\label{a4}\label{assump:change}
There exists a constant $\bt  \in (0,1/2) $, such that for all $h \in \mathbb{N}$ the unknown locations $\floor{nt_h}$ of the structural breaks satisfy
\begin{enumerate}[(C{1})]
\item $\bt \leq t_h \leq 1 - \bt$~.
\end{enumerate}
\end{assump}

\noindent
Let us give a brief explanation for the Assumptions \ref{a1} - \ref{a4}.
The temporal Assumptions (T1) and (T2) define the temporal dependence structure and bounds for the long-run variance.
Further (T1) implies the existence of the quantities $\sigma_h$, $\gamma_{h,i}$ and $\rho_{h,i}$ defined in \eqref{sigmah} and \eqref{longrun}.
Condition (S3) and (S4) refer to the spatial dependence and are only needed to derive the desired extreme value convergence.
Assumption (S1) gives a relation between the number of observations and its dimension, while (S2) is a slightly technical assumption.
In particular, the coefficients in \eqref{eq:physicalcoeff} satisfy the inequality $\vartheta_{j,h,p_1} \leq \vartheta_{j,h,p_2}$ if $1\leq p_1 \leq p_2$~, and therefore a larger value of $p$ in (T1) yields a less dependent time series.
This on the other hand is reflected in condition (S2), which is less restrictive for larger values of $p$. 
This condition enables reasonable estimators of the variance $\sigma_h$ and the locations $t_h$ of the structural breaks.
For a proof of the uniform consistency of the estimates for the latter quantity we must further rely on Assumption (C1), which makes the change points identifiable and is a common assumption in the literature.
Roughly speaking, it simply ensures to have enough data before and after the change point in each component. 
Assumptions (S1) and (S2) together with $n \to \infty$ directly imply $d \to \infty$.
It is also worth mentioning, that (S1) can be replaced by $d \lesssim n^D$, if one additionally supposes that $d \to \infty$. 

The moment Assumptions (M1) and (M2) are required for a Gaussian approximation and are satisfied, if $\{X_{j,h}\}_{j \in \mathbb{Z}, h \in \setd}$ is stationary with respect to the index $j$ and for each $h \in \mathbb{N}$ the random variable $X_{1,h}$ is sub-Gaussian with parameters $v_h, V_h>0$, i.e.  
\begin{align*}
\mean{\exp(\lambda X_{1,h})}
\leq V_h \exp(\lambda^2 v_h), ~~\text{ for all  } \lambda \in \mathbb{R}~,
\end{align*}
where the constants $(v_h)_{h=1,\ldots ,d}$ and $(V_h)_{h=1,\ldots ,d}$ satisfy  
\begin{align*}
\maxhd v_h < n^{3/4}
\;\;\;
\text{and}
\;\;\;
\max_{h \in \mathbb{N}} V_h \leq C~.
\end{align*} 
for some constant $C>0$.
Choosing $M_d = \maxhd \sqrt{v_h}$ we directly obtain condition (M2). Condition (M1) follows by a straightforward calculation, that is 
\begin{align*}
\maxhd \mean{\exp(\vert X_{1,h} \vert /M_d)} \leq 2\maxhd V_h e^{v_h/M_d^2} \leq 2Ce < \infty~.
\end{align*}

\subsection{Asymptotic properties - known variances and  locations}
\label{subsec:mainresults}
In this section we assume that the locations $t_h$ of the changes and the long-run variances  $\sigma_h$ are known.
Recalling the approximation \eqref{motiv} we define
\begin{align}\label{hmat}
\Mint = \frac{3}{(t_h(1-t_h))^2}\int_0^1 \cusumssq ds
\end{align}
and investigate the asymptotic properties of the maximum of the statistics
\begin{align}\label{t1}
\Tnh
&= \frac{\sqrt{n}}{\tau_h \sigmah \absdmuh} \left( \Mint - \dmuh^2 \right)~,
\end{align}
where $\tau_h=\tau(t_h)$ and the function $\tau $ is defined by \eqref{tauh}.
Note that $\Tnh$ is the analogue of the statistic $\hatTnhdelta$, where the thresholds $\Delta_h$, estimates $\hat t_h$ and $\hat \sigma_h$ have been replaced by the unknown quantities $\dmuh$, $t_h$ and $ \sigma_h$, respectively.
Our first result shows that an appropriate standardized version of the maximum of the statistics $\Tnh$ converges weakly to a Gumbel distribution.
The proof is complicated and we indicate the main steps in this section deferring the more technical arguments to the appendix.

\begin{theorem}\label{cor:tnhgumbel}
Assume that the Assumptions \ref{a1} - \ref{a4} are satisfied and that additionally there exist constants $C_{\ell}$, $C_{u}$ (independent of $h$ and $d$) such that the inequalities $C_{\ell} \leq \absdmuh \leq C_u$ hold for all $h=1,\ldots , d$ .
Then 
\begin{align*}
a_d \Big ( \maxhd T_{n,h} - b_d \Big) \convd G
\end{align*}
as $d,n \to \infty$, where the sequences $a_d$ and $b_d$ are defined in \eqref{ad} and \eqref{bd}, respectively.
\end{theorem}

\noindent
{\bf Proof of Theorem \ref{cor:tnhgumbel} (main steps).} 
Observing the definition \eqref{hmat} and \eqref{t1} we obtain by a straightforward calculation the representation 
\begin{align}\label{eq:tnhexpansion}
\begin{split}
\Tnh 
= \frac{3\sqrt{n}}{(t_h (1 - t_h))^2\tau_h\sigma_h\absdmuh} \left( \int_0^1\mathbb{U}_{n,h}^2(s)ds - \mu_h^2(s,t_h)ds \right)
= T_{n,h}^{(1)} + T_{n,h}^{(2)}~,
\end{split}
\end{align}
where 
\begin{align}\label{mu}
\mu_h(s,t_h)  = \left( t_h \wedge s - st_h \right)(\mu_{h,1}-\mu_{h,2})
\end{align}
and the statistics $T_{n,h}^{(1)}$ and $T_{n,h}^{(2)}$ are defined by
\begin{align}
\label{tn1} T_{n,h}^{(1)} &= \frac{3\sqrt{n}}{(t_h (1 - t_h))^2\tau_h\sigma_h\absdmuh} \int_0^1 (\mathbb{U}_{n,h}(s) - \mu_h(s,t_h))^2 ds~, \\
\label{tn2} T_{n,h}^{(2)} &= \frac{6\sqrt{n}}{(t_h (1 - t_h))^2\tau_h\sigma_h\absdmuh} \int_0^1 \mu_h(s,t_h)(\mathbb{U}_{n,h}(s) - \mu_h(s,t_h))ds~.
\end{align}
For the following discussion, we introduce the additional notation
\begin{align}
\label{tn2bar} \bar{T}_{n,h}^{(2)} &= \frac{6\sqrt{n}}{(t_h (1 - t_h))^2\tau_h\sigma_h\absdmuh} \int_0^1 \mu_h(s,t_h)\left( \mathbb{U}_{n,h}(s) - \mean{ \mathbb{U}_{n,h}(s)} \right) ds~.
\end{align}
Our first auxiliary result shows that the first term $T_{n,h}^{(1)}$ in the decomposition \eqref{eq:tnhexpansion} is asymptotically negligible.
It is proved in Section \ref{sec711} of the appendix. 

\begin{lemma}\label{lem:tnh1}
If the assumptions of Theorem \ref{cor:tnhgumbel} are satisfied, we have 
\begin{align*}
a_d \maxhd T_{n,h}^{(1)} \convp 0
\end{align*}
as $d,n \to \infty$, where the sequence $a_d$ is defined in \eqref{ad}.
\end{lemma}
\noindent
By Lemma \ref{lem:tnh1} it suffices now to deal with the term  $T_{n,h}^{(2)}$.
The next step in the proof of Theorem \ref{cor:tnhgumbel} consists in a (uniform) approximation of the distribution of the maximum of the statistics $\bar{T}_{n,h}^{(2)}$ by the distribution of the maximum of (dependent) Gaussian random variables. 
The proof of the following result is given in Section \ref{sec712} of the appendix, where we make use of new developments on Gaussian approximations for maxima of sums of random variables  [see \cite{chernozhukov:2013} and \cite{zhang:2016}]. 

\begin{lemma}\label{lem:gaussappro}
If the Assumptions \ref{a1} - \ref{a4} are satisfied, there exists a Gaussian distributed random vector $N=(N_1,\dots,N_d)^T$ with 
mean $\mean{N}=0$ and covariance matrix $(\Sigma_{h,i})_{h,i=1}^d$, such that
\begin{align*}
\sup_{x \in \mathbb{R}} \Big \vert \Prob  \Big( \maxhd \bar{T}_{n,h}^{(2)} \leq x 
\Big )  - \Prob \Big ( \maxhd N_h \leq x  \Big )\Big \vert
= o(1)
\end{align*}
for $d,n \to \infty$. Moreover, the entries of the matrix $\Sigma$ are bounded by
$ |\Sigma_{h,i}| \leq |\rho_{h,i}|~,$
where $\rho_{h,i}$ are the long-run correlations defined in \eqref{longrun}.
\end{lemma}

\noindent
By Lemma \ref{lem:gaussappro} it suffices to establish the desired limiting distribution for the maximum of a Gaussian distributed vector. 
Nowadays, this is a well-understood area of mathematics and we can rely on results of \cite{berman:1964}, who originally examined the behavior of the maximum of dependent Gaussian random variables.
A straightforward adaption of these arguments shows that the sequence of random variables $\{N_i\}_{i \in \mathbb{N}}$ constructed in Lemma \ref{lem:gaussappro} satisfies
\begin{align}\label{weakn}
a_d \Big( \maxhd N_h - b_d \Big) \convd G
\end{align}
as $d \to \infty$, where the sequences $a_d$ and $b_d$ are defined in \eqref{ad} and \eqref{bd}, respectively.\\
\noindent
Now, combining Lemma \ref{lem:gaussappro} and \eqref{weakn} directly leads to
\begin{align*}
a_d \Big( \maxhd \bar{T}_{n,h}^{(2)}- b_d \Big) \convd G~.
\end{align*}
By $\absdmuh \leq C_u$ it follows that $\maxhd \left( \mu_h(s,t_h) - \mean{\cusums} \right) = O(n^{-1})$, which yields
\begin{align}\label{eq:tnh2conv}
a_d \Big( \maxhd \Tnh^{(2)} - b_d \Big) \convd G~.
\end{align}
Due to $T_{n,h}^{(1)} \geq 0$, we obtain the inequalities
\begin{align*}
a_d \Big( \maxhd \Tnh^{(2)} - b_d \Big) 
&\leq a_d\Big( \maxhd  \Tnh - b_d \Big)\\
&\leq a_d \maxhd \Tnh^{(1)} + a_d \Big( \maxhd \Tnh^{(2)} - b_d \Big),
\end{align*}
which together with Lemma \ref{lem:tnh1} yields the assertion of Theorem \ref{cor:tnhgumbel}.
\hfill $\Box$

\bigskip

\noindent
In the next step we will replace the unknown quantities $t_h$ and $\sigmah$ in \eqref{t1} by suitable estimators, say $\hat t_h$ and $\hatsigmah$, and obtain the statistics
\begin{align}\label{that}
\hatTnh 
&= \frac{\sqrt{n}}{\tauhatth \hatsigmah \absdmuh} \Big ( \hatMint - \dmuh^2 \Big ),  ~~~h\in \setd ~.
\end{align}
We emphasize again that the statistics $\hatTnh $ do not coincide with the statistics $\hatTnhdelta$ in \eqref{eq:tnhdelta}, which are actually used in the test \eqref{asymtest} except in the case where $\dmuh^2 = \vert \mu_h^{(1)} - \mu_h^{(2)} \vert^2 = \Delta_h^2$ for all $h=1,\ldots,d$. 
Thus centering is still performed with respect to the unknown difference of the means before and after the change points.
In the following two subsections we give a precise definition of the two estimators and derive an analogue of Theorem \ref{cor:tnhgumbel} in the case of estimated change points and variances.

\subsection{Estimation of long-run variances and change point locations}  \label{sec2s}

Determining the relative locations $t_h$ of the structural breaks and constructing an estimator for the long-run variances 
$\sigma_h$ for all components $h \in \setd$ is a rather difficult task in a high dimensional setting. 
A further difficulty in the problem of testing for relevant structural breaks consists in the fact that
even under the null hypothesis there may appear structural  breaks in the mean and the corresponding process is not stationary. 
Therefore in contrast of testing the ``classical''  hypotheses in \eqref{class} the construction of a suitable variance estimator is not trivial. 
A standard long-run variance estimator in terms of $\sum_{i \leq \beta_n} \hat{\phi}_{h}(i)$ for an increasing sequence $\{\beta_n\}_{n \in \mathbb{N}}$ and appropriate estimators $\hat{\phi}_{h}(i)$ of the auto-covariance from the full sample may not be consistent due to possible changes of the mean.  

Following \cite{jirak:2015} we define for each component $h=1, \ldots , d$ the sets
\begin{align*}
\mathcal{D}_{h,1}:=\left\lbrace Z_{j,h} \; \vert \; 1\leq j \leq \floor{nt_h}  \right\rbrace
\;\;\;\text{and}\;\;\;
\mathcal{D}_{h,2}:=\left\lbrace Z_{j,h} \; \vert \; \floor{nt_h} <  j \leq n  \right\rbrace~,
\end{align*}
which are the observations before and after the (unknown) change point $\floor{nt_h}$, respectively.
Since these points are usually unknown, we need to estimate them and for this purpose we propose the common estimator 
\begin{align} \label{eq:hatth}
\hatth 
:= \argmax_{s \in (\bt, 1-\bt)}  \big | \cusum{s}  \big | 
= \argmax_{s \in (\bt, 1-\bt)} \Big | \sum_{j=1}^{\floor{ns}}Z_{j,h} - \frac{\floor{ns}}{n}\sum_{j=1}^{n}Z_{j,h}\Big |~.
\end{align}
The following Lemma shows that these estimators are uniformly consistent with respect to all components, where a change point exists.
The proof can be obtained by employing Theorem 3.1 in \cite{jirak:2015} and is given in Section \ref{proofsec23} of the appendix.
\begin{lemma}\label{lem:hatth}  
Let
\begin{align} \label{sd} 
\Sd = \left\lbrace 1 \leq h \leq d  \;\vert\; \absdmuh = 0 \right\rbrace
\end{align}
denote the set of all components $h \in \setd$, where the corresponding time series 
$\{ Z_{j,h} \}_{j\in \Z}$ is stationary, and define 
\begin{align} \label{mustar}
\mu^{\star}_d=\min_{h \in \Sd^c} \absdmuh~.
\end{align}
Suppose further that Assumptions \ref{a1}, \ref{a2} and \ref{a4} are satisfied.
\begin{enumerate}[(i)]
\item If for a sufficiently small constant $1> C >0$  the condition 
\begin{align}\label{eq:limsupmustar}
\limsup_{d,n\to \infty} \frac{\log n}{ (\mu_d^{\star})^2n} = o(n^{-C})
\end{align}
holds, then (for a possibly smaller constant)
\begin{align*}
\max_{h \in \Sd^c} \vert \hatth - t_h \vert =o_\Prob(n^{-C})~.
\end{align*}
\item If the condition $\mu^{\star}_d \geq C_{\ell}$ holds for some constant $C_{\ell}>0$, it follows that
\begin{align}\label{eq:explicitrate}
\max_{h \in \Sd^c} \vert \hatth - t_h \vert = o_\Prob\Big(n^{\tfrac{D}{p/2-2}-1}\Big)~.
\end{align}
\end{enumerate}
\end{lemma}
\begin{remark}
Lemma \ref{lem:hatth} provides a result regarding the uniform convergence of the estimators $(\hatth)_{h \in \setd}$
for the locations of the change points as the dimension $d$ converges to infinity.
Roughly speaking condition \eqref{eq:limsupmustar} guarantees that the decreasing sequence $\mu^{\star}_d$ does not converge too fast to $0$.
Otherwise it is not possible to identify the (relative) locations of all change points simultaneously.\\
In the case where the sequence $\mu^{\star}_d$ has a positive lower bound, part \textit{(ii)} of Lemma \ref{lem:hatth} gives an explicit rate of convergence in terms of the quantities $p$ and $D$.
Note, that larger values of  $p$ improve the rate and in the limit $p \to \infty$ with fixed $D$ one recovers the known optimal rate $o_p(n^{-1})$ for a finite dimension.
On the other hand an increasing parameter $D$, which controls the growth of the dimension ($d\sim n^D$), leads to a slower rate of convergence.
Additionally, we would like to point out, that statement \eqref{eq:explicitrate} combined with Assumption \ref{a2} gives
\begin{align}\label{eq:explicitrate2}
\max_{h \in \Sd^c} \vert \hatth - t_h \vert = o_\Prob\big(n^{-1/2}\big)~,
\end{align}
which we will use frequently in the sequel.
\end{remark}

\noindent
Due to the possible appearance of changes in the mean a standard approach of long-run variance estimation will usually fail. 
To circumvent this problem we follow \cite{jirak:2015a} who proposed two single variance estimators.
To be precise note that view of Lemma \ref{lem:hatth} it is reasonable to estimate the unknown sets $\mathcal{D}_{h,1}$ and $\mathcal{D}_{h,2}$ by 
\begin{align}
\begin{split}\label{dh12}
\widehat{\mathcal{D}}_{h,1} :&= \left\lbrace Z_{j,h} \; \vert \; 1\leq j \leq n\max\{S \hatth , \bt \}  \right\rbrace,\\
\widehat{\mathcal{D}}_{h,2} :&= \left\lbrace Z_{j,h} \; \vert \; n - n\max\{ S( 1 - \hatth ), \bt \} <  j \leq n  \right\rbrace,
\end{split}
\end{align}
respectively, where $S \in (0,1)$ denotes a user-specified separation constant.
For both sets, we can now compute the sample mean, that is
\begin{align*}
\bar{Z}_h^{(1)}=\dfrac{1}{\vert \widehat{\mathcal{D}}_{h,1} \vert }\sum_{Z_{j,h} \in \widehat{\mathcal{D}}_{h,1}} Z_{j,h}~,\;\;
\bar{Z}_h^{(2)}=\dfrac{1}{\vert \widehat{\mathcal{D}}_{h,2} \vert }\sum_{Z_{j,h} \in \widehat{\mathcal{D}}_{h,2}} Z_{j,h}~.
\end{align*}
We then obtain corresponding estimates of the autocovariances by
\begin{align*}
\hat{\phi}_{h}^{(1)}(i)
= \dfrac{1}{\vert\widehat{\mathcal{D}}_{h,1}\vert - \vert i \vert}\sum_{Z_{j,h},Z_{j+i,h} \in \widehat{\mathcal{D}}_{h,1}}(Z_{j,h} - \bar{Z}_h^{(1)})(Z_{j+i,h} - \bar{Z}_h^{(1)})~, \\
\hat{\phi}_{h}^{(2)}(i)
= \dfrac{1}{\vert\widehat{\mathcal{D}}_{h,2}\vert -\vert i \vert}\sum_{Z_{j,h},Z_{j+i,h} \in \widehat{\mathcal{D}}_{h,2}}(Z_{j,h} - \bar{Z}_h^{(2)})(Z_{j+i,h} - \bar{Z}_h^{(2)})~.
\end{align*} 
This yields two long-run variance estimators, namely
\begin{align*}
&\hatsigma_{h,1}^2
= \sum_{\vert i \vert \leq \beta_n} \hat{\phi}_h^{(1)}(i)\;\; \text{ (based on the set } \widehat{\mathcal{D}}_{h,1})~,\\
&\hatsigma_{h,2}^2
= \sum_{\vert i \vert \leq \beta_n} \hat{\phi}_h^{(2)}(i)\;\; \text{ (based on the set } \widehat{\mathcal{D}}_{h,2})~,
\end{align*}
where $\beta_n \sim n^B$ is the bandwidth and $B$ refers to the constant in Assumption \ref{a2}.
At this point we can use any convex combination of $\hat{\sigma}_{h,1}^2$ and $\hat{\sigma}_{h,2}^2$ to estimate the long run variances
$\sigma^2_h$, for example $\hatsigmah^2  =\frac{1}{2} ( \hat{\sigma}_{h,1}^2 + \hat{\sigma}_{h,2}^2) $.
To simplify the technical arguments in Sections \ref{proofsec23} and \ref{proofsec4} of the appendix we consider a truncated version, that is 
\begin{align}\label{eq:variance}
\hatsigmah^2 = \min \big \{ s_+^2 , \max  \big \{ s_-^2, (\tfrac{1}{2} ( \hatsigma_{h,1}^2 + \hatsigma_{h,2}^2) \big \} \big  \}~.
\end{align}
where $0<s_- $ and $ s_+ $ are a sufficiently small and large constant, respectively. 
The following statement yields the consistency of these estimators (uniformly with respect to the spatial component). 

\begin{lemma}\label{lem:hatvar}
Suppose that Assumptions \ref{a1} - \ref{a4} are satisfied and additionally that there exists a constant $C_\ell > 0$ such that the inequality $C_\ell \leq \absdmuh$ holds for all $h=1,\ldots , d$.
Then we have 
\begin{align*}
\maxhd \vert \hatsigmah - \sigma_h \vert = o_{\mathbb{P}} \left( n^{- \eta}  \right)
\end{align*}
for a sufficiently small constant $0 < \eta < 2\Big(p/4-B(p/2+1)-1-D\Big)/p$.
\end{lemma}

\noindent
From the proof of Lemma \ref{lem:hatvar} it follows that the constant $\eta$ can be chosen larger if the distance $p/4 - B(p/2+1) -1 - D$ in Assumption (S2) increases.

\subsection{Weak convergence} \label{sec34}
Equipped with Lemmas \ref{lem:hatth} and \ref{lem:hatvar} we are now able to state the main result of this section, which will be proved in Section \ref{sec713} of the appendix.
\begin{theorem}\label{thm:gumbel}
If Assumptions \ref{a1} - \ref{a4} are satisfied, then the statistics $\hatTnh$ defined in \eqref{that} satisfy
\begin{align}\label{stat312}
\TS = a_d \Big( \maxhd \hatTnh - b_d \Big) \convd G
\end{align}
as $d,n \to \infty$, where the sequences $a_d$ and $b_d$ are defined in \eqref{ad} and \eqref{bd}, respectively. 
\end{theorem}

\noindent
Recall again that the statistics $\hatTnh$ and $\hatTnhdelta$ in \eqref{eq:tnhdelta} do not coincide in general.
Thus Theorem \ref{thm:gumbel} does not show that the test \eqref{asymtest} is an asymptotic level $\alpha$ test because it does not cover all parameter configurations of the the null hypothesis \eqref{hnull}.
However, if the vector $ (\mu^{(1)}, \mu^{(2)})$ is an element of the set ${\cal A}_d$ defined in \eqref{kbound} we have $\hatTnh = \hatTnhdelta$ and it follows from this result that the probability of rejection converges to $\alpha$, that is
$$
\limdn
\Prob_{(\mu^{(1)}, \mu^{(2)} )\in {\cal A}_d}   \big ( \TS > \gua   \big )  = \alpha~.
$$ 
We conclude this section with a result, which can be used to control the type I error of the test \eqref{asymtest} for other values of the vector $(\mu^{(1)}, \mu^{(2)})$.

\begin{corollary}\label{cor:adjustedgumbel}
Let $\{{\cal M}_d\}_{d \in \mathbb{N}}$ be an increasing sequence of subsets of $\setd$ (as $d,n\to \infty$).
If the assumptions of Theorem \ref{cor:tnhgumbel} hold, then
\begin{align*}
a_d \Big ( \max_{h\in{\cal M}_d} \hatTnh - b_d \Big ) 
\convd
\left\{
\begin{array}{rl}
G          &\text{if}\;\; \limd | \mathcal{M}_d | /d = 1 ~,\\
G + \log c &\text{if}\;\; \limd |{\cal M}_d | /d = c~,~~c \in (0,1)~, \\
-\infty    &\text{if}\;\; \limd |{\cal M}_d|/d = 0~. \\
\end{array} 
\right.
\end{align*}
Irrespective of the sequence $ \{ {\cal M}_d \}_{d \in \mathbb{N}}$, the bound
\begin{align}\label{eq:smallerGumbel}
\limsup_{d,n \to \infty} \Prob \Big(  a_d \Big  ( \max_{h \in{\cal M}_d} \hatTnh - b_d \Big  ) > x \Big)  \leq \Pb{ G > x}
\end{align}
is valid for all $x \in \mathbb{R}$.
\end{corollary}

\section{Relevant changes in high dimensional time series} 
\label{sec4}
\def\theequation{4.\arabic{equation}}
\setcounter{equation}{0}

Recall the problem of testing the hypotheses of a relevant change defined in \eqref{hnull} and \eqref{halt}.
We propose to reject the null hypothesis of no relevant change in any component of the high dimensional mean vector, whenever the inequality \eqref{asymtest} holds, that is $\TS > \gua ~,$ where the test statistic $\TS$ is defined in $\eqref{eq:TS}$ and $\gua$ denotes the $(1-\alpha)$ quantile of the Gumbel distribution. 
The following two results make the discussion at the end of Section \ref{sec2} rigorous and show that the test introduced in \eqref{asymtest} defines in fact a consistent and asymptotic level $\alpha$ test.
The proofs can be found in the appendix.

\begin{theorem}\label{thm:level}
Suppose that the Assumptions \ref{a1} - \ref{a4} are satisfied, $\alpha \in (0,1-e^{-1}]$ and that there exist constants $\Delta_-$, $\Delta_+$ such that the thresholds $\Delta_h $ satisfy the inequalities
\begin{align}\label{ineq:Delta}
0<\Delta_-\leq \Delta_h \leq \Delta_+<\infty
\end{align}
for all $h=1,\ldots, d$. 
Then, under the null hypothesis  $H_{0,\Delta} $ of no relevant change, it follows
\begin{align} \label{limsup} 
\limsup_{d,n \to \infty} \Pb{ \TS > \gua } \leq \alpha~. 
\end{align}
Moreover, let $\Bd = \{ h \in \setd \;  | \;  \Delta_h = \absdmuh \}$ and assume further that 
\begin{align}\label{eq:thmlevelextra}
\Delta_h - \absdmuh \geq C_{\Delta} > 0\;\; \text{for all}\;\; h \in \Bd^c~,
\end{align}
for some constant $C_{\Delta} > 0$, then, under $H_{0,\Delta} $, we have
\begin{align}\label{eq:thmlevelcases}
\TS
\convd
\left\{
\begin{array}{rl}
G          &\text{if}\;\; \limd |\Bd|/d = 1~,\vspace{0.2cm} \\
G + \log c &\text{if}\;\; \limd |\Bd|/d = c \text{ for } c \in (0,1)~,\vspace{0.2cm} \\
-\infty    &\text{if}\;\; \limd |\Bd|/d = 0~.
\end{array} 
\right.
\end{align}
\end{theorem}

\begin{remark}
Condition \eqref{ineq:Delta} is actually not a very strong restriction since the thresholds $\Delta_h$ are defined by the user.
Nevertheless, the condition is crucial since we use the factor $1/\Delta_h$ as a normalisation in the statistics $\hatTnh^{(\Delta )}$ defined in \eqref{eq:tnhdelta}.
Note that under the null hypothesis \eqref{hnull} the inequality $\Delta_h \leq \Delta_+$ is equivalent to $\absdmuh \leq C_u$, which was one of the assumptions in Theorem \ref{thm:gumbel}.
Consequently the assertion of Theorem \ref{thm:level} follows from Theorem \ref{thm:gumbel} in the case where $\absdmuh = \Delta_h$ for all $h \in \setd$.
However, in the general case the proof of Theorem \ref{thm:level} is more complicated and deferred to Section \ref{sec721} of the appendix, where we also handle the case $\absdmuh < \Delta_h$.
\end{remark}
In the following result we investigate the consistency of the new test.
Interestingly it requires less assumptions than Theorem \ref{thm:level}.

\begin{theorem}\label{thm:consistent}
Suppose that the Assumptions \ref{a1} and \ref{a4} hold.
Then under the alternative hypothesis $\HA$ of at least one relevant change point we have
\begin{align*}
\TS \convp \infty~,
\end{align*}
as $d,n \to \infty$, which in particular gives
\begin{align*}
\limdn \Pb{ \TS > \gua } = 1~. 
\end{align*}
\end{theorem}

\noindent
If the test \eqref{asymtest} rejects the null hypothesis $\HNull$ in \eqref{hnull}
we conclude (at a controlled type I error) that there is at least one component with a relevant change in mean.
As there could exist relevant changes in several components, the next step in the statistical inference is the identification of the set
\begin{align}\label{eq:Rd}
\Rd = \left\lbrace 1 \leq h \leq d  \;\vert\; \vert \dmuh \vert > \Delta_h \right\rbrace~,
\end{align}
of all components with a relevant change.
Note that the hypotheses in \eqref{hnull} and \eqref{halt} are equivalent to $ \HNull: \Rd = \emptyset$ versus $
\HA: \Rd \not = \emptyset$ . 
In light of Theorem \ref{thm:level} and \ref{thm:consistent} a natural estimator for this set is therefore given by
\begin{align}\label{eq:hatRd}
\hatRd(\alpha) = \left\lbrace 1 \leq h \leq d \;\vert\; \hatTnhdelta > \gua/a_d + b_d \right\rbrace~.
\end{align}
The following theorem provides a consistency result of this estimate.
\begin{theorem} \label{lem:identification}
Suppose that the Assumptions \ref{a1} - \ref{a4} hold and assume additionally 
that there exist two constants $0<C<1/2$, $C_u>0$ such that 
\begin{align}\label{eq:mindistance}
n^C\min_{h \in \Rd} \left( \absdmuh - \Delta_h \right) = \infty
\;\;\;
\text{and}
\;\;\;
\max_{h \in \Rd} \absdmuh \leq C_u~.
\end{align}
Then, the set estimator defined in \eqref{eq:hatRd} satisfies for $\alpha \in (0, 1- e^{-1}]$ 
\begin{align}\label{eq:RdIncludedProb}
\limdn \Pb{ \Rd \subset \hatRd(\alpha) }  = 1~.
\end{align}
Moreover we have the following lower bound
\begin{align}\label{eq:RdEqualProb}
\liminf_{d,n \to \infty} \Pb{ \Rd = \hatRd(\alpha) } \geq 1 - \alpha~.
\end{align}
\end{theorem}
\noindent

\section{Bootstrap}\label{sec5}
\def\theequation{5.\arabic{equation}}
\setcounter{equation}{0}
The testing procedure introduced in the previous sections is based on the weak convergence of the maximum of appropriately standardized statistics to a Gumbel distribution, and it is well known that the speed of convergence in limit theorems of this type is rather slow.
As a consequence the approximation of the nominal level of the test \eqref{asymtest} for finite sample sizes may not be accurate.
A common way to improve the performance of the test, is to obtain the critical values from an appropriate bootstrap procedure. 
    
In the context of testing for relevant change points the construction of an appropriate resampling procedure is not obvious as  - in contrast to classical change point problems - the parameter space under the null hypothesis is rather large.
In particular it will be necessary to simulate the distribution of the statistic $\TS$ in case  $\absdmuh = \Delta_h$ for all $h \in \setd$ such that one can replace the quantile of the Gumbel distribution by the corresponding quantile of the bootstrap statistics. 
A further problem is to mimic the dependence of the underlying time series, which we will address employing a Gaussian multiplier bootstrap,
where observations are block-wise multiplied with independent Gaussian random variables [see \cite{kunsch1989} or \cite{lahiri:2003}].

To handle the problem of potential change points under the null hypothesis \eqref{hnull} of no relevant changes, observations
from blocks in a neighborhood of estimated change points will not be used in the estimate.
Furthermore, components without a change point will be ignored when the bootstrap statistic is constructed.
We begin describing this idea in more detail and show in Theorem \ref{thm:BSgumbel}, that the bootstrap statistic converges weakly to a Gumbel distribution as well.
In the sequel we assume without loss of generalization that $n=KL$ and will split the sample into $L$ blocks of length $K$, and additionally
\begin{align}\label{eq:LK}
L\sim n^{\ell}
\;\;\;
\text{and}
\;\;\;
K \sim n^{1-\ell}
\;\;\;
\text{for}
\;\;\;
\ell \in (0,1)~.
\end{align}
We obtain the following quantities, which control the number of blocks before and after the estimated change point.
\begin{align}\label{eq:blocklimits}
\begin{split}
\widehat{L}_h^{-}&=\sup \{ \ell \in \mathbb{N} \;  | \; \ell K + K/2 \leq \hatth n\}~,\\
\widehat{L}_h^{+}&=\inf \{ \ell \in \mathbb{N} \;  | \; \ell K - K/2 \geq \hatth n\}~,
\end{split} 
\end{align}
where $\hatth$ denotes the estimator of the location $t_h$ of the change in the $h$-th component defined in Section \ref{sec34}.
The corresponding sample means are given by
\begin{align}\label{eq:separatemeans}
\bar{Z}_h^{-} = \frac{1}{K\widehat{L}_h^{-}}\sum_{j=1}^{K\widehat{L}_h^{-}}Z_{j,h}
\;\;\;
\text{and}
\;\;\;
\bar{Z}_h^{+} = \frac{1}{K(L-\widehat{L}_h^{+})}\sum_{j=K\widehat{L}_h^{+}+1}^{KL}Z_{j,h}~,
\end{align} 
which can be used to define an estimator for the unknown amount of change $\dmuh = \mu_{h,1} - \mu_{h,2}$ in the $h$-th component, that is
\begin{align}\label{eq:hatdmuh}
\hatdmuh = \bar{Z}_h^{-}-\bar{Z}_h^{+}~.
\end{align}
Moreover, these estimators can also be used to define the ``mean corrected'' sample 
\begin{align}\label{eq:transform}
\widehat{Z}_{j,h}&=\left\{
\begin{array}{rl}
  Z_{j,h} - \bar{Z}_h^{-} \;\; &\text{for}\;\; j \leq K\widehat{L}_h^{-}~,\\
  0 \;\; &\text{for}\;\; K\widehat{L}_h^{-} < j \leq K\widehat{L}_h^{+}~,\\
  Z_{j,h} - \bar{Z}_h^{+} \;\; &\text{for}\;\; j > K\widehat{L}_{h}^{+}~. 			
\end{array} 
\right.
\end{align}
Finally, we define blocking variables (that are sums with respect to the different blocks) as 
\begin{align}\label{eq:Vlh}
\widehat{V}_{\ell,h}(k) &= \sum_{j=(\ell-1)K+1}^{\ell K}\widehat{Z}_{j,h}I\{j\leq k\}
\end{align}
and introduce the notation  
$$
\widehat{V}_{\ell,h}=\widehat{V}_{\ell,h}(n) = \sum_{j=(\ell-1)K+1}^{\ell K}\widehat{Z}_{j,h}~.
$$
From the representation \eqref{eq:transform} we directly obtain, that blocks near to the estimator $\hatth$ are ignored, i.e. we have 
\begin{align*}
\widehat{V}_{\ell,h}=0 \;\;\;\text{if}\; \ell \in \{\widehat{L}_h^- +1,\dots, \widehat{L}_h^{+}\}~.
\end{align*}
Further denote by
$$
\Zn = \sigma(Z_{j,h}\; \vert \; 1 \leq j \leq n, 1\leq h \leq d)
$$ 
the $\sigma$-field generated by the sample $(Z_{j,h})_{1 \leq j \leq n, 1\leq h \leq d}$ 
and by $\{\xi_{\ell}\}_{\ell \in \mathbb{N}}$ a sequence of i.i.d. standard Gaussian random variables, which is independent of $\Zn$. 
Now we consider a multiplier version of the CUSUM-process from the $L$ blocks, that is
\begin{align}
\label{eq:BSCUSUM}
\mathbb{U}_{n,h}^{(L)}(s)= \frac{1}{n}\sum_{\ell=1}^{L}\xi_{\ell}\widehat{V}_{\ell,h}(\floor{ns})
-\frac{\floor{ns}}{n^2}\sum_{\ell=1}^{L}\xi_{\ell}\widehat{V}_{\ell,h}(n)~,
\end{align}
and introduce the bootstrap integral statistics (for each component)
\begin{align}\label{eq:BSIntegral}
\begin{split}
B_{n,h} &=  \frac{6\sqrt{n}}{\widehat{s}_h\tau(\hatth)(\hatth(1-\hatth))^2} \int_0^1 \mathbb{U}_{n,h}^{(L)}(s)k(s,\hatth)ds \cdot
{I}\big \lbrace \vert \hatdmuh \vert > n^{-1/4} \big \rbrace \\
&+ {I}\big \lbrace \vert \hatdmuh \vert \leq n^{-1/4} \big \rbrace \cdot b_d~, 
\end{split}
\end{align}
where $k(x,y)=x\wedge y - xy$ denotes the covariance kernel of the standard Brownian bridge and
$$
\hat{s}_h^2 = \frac{1}{L}\sum_{\ell=1}^{L}\xi_{\ell}^2\cdot \hatsigmah
$$
is the variance estimate from the bootstrap sample.
In an analogous manner to the previous sections, we define a normalized maximum of the bootstrap statistics $B_{n,h}$  by
\begin{align} \label{bootstat}
\BSTS = a_d\Big ( \maxhd B_{n,h} - b_d\Big )~,
\end{align}
where the normalizing sequences $a_d$ and $b_d$ are given by \eqref{ad} and \eqref{bd}, respectively.

\begin{remark}
Let us give a brief heuristic explanation, why \eqref{eq:BSIntegral} and \eqref{bootstat} define an appropriate bootstrap statistic.
Our basic aim is to mimic the distribution of the test statistics $\TS$ on the ``$0$-dimensional boundary'' of the null hypothesis $\HNull$, i.e. in case of $\absdmuh = \Delta_h$ for all $h=1,\dots,d$.
Note, that we have $\dfrac{\mu_h(s,t_h)}{\Delta_h} = \sign (\dmuh) k(s,t_h)$ and it was outlined in Section \ref{sec3}, that in this setting the representation 
\begin{align*}
\Tnh &= 
\frac{3\sqrt{n}}{(t_h (1 - t_h))^2\tau_h\sigma_h\absdmuh} \int_0^1 (\mathbb{U}_{n,h}(s) - \mu_h(s,t_h))^2 ds \\
&\;\;\;+ \frac{6\sqrt{n}}{(t_h (1 - t_h))^2\tau_h\sigma_h\absdmuh} \int_0^1 \mu_h(s,t_h)(\mathbb{U}_{n,h}(s) - \mu_h(s,t_h))ds
\end{align*}
holds, where by Lemma \ref{lem:tnh1} the first summand on the right-hand side is of order $o_\Prob(1)$.
The component-wise bootstrap statistic $\B_{n,h}$ is supposed to imitate the second summand in this decomposition.
However, this approach is only sensible for all components $h$, that actually contain a change and for this reason we introduce the indicator function $I\big \lbrace \vert \hatdmuh \vert > n^{-1/4} \big \rbrace$.
To be more precise, we will show in the appendix that
\begin{align*}
\BSTS = a_d\Big ( \maxhd B_{n,h} - b_d\Big ) \approx a_d \Big ( \max_{h \in \Sd^c} B_{n,h} - b_d\Big )
\end{align*}
as $d,n \to \infty$, where the set $\Sd$ is defined in $\eqref{sd}$.
The statistic $\BSTS$ will then be used to generate bootstrap quantiles for the statistic $\TS $.
In order to prove that this is a valid approach we require the following additional assumptions.
\end{remark}

\begin{assump}[assumptions for the bootstrap]\label{assump:bootstrap}
For the constants $p,D$ in Assumption \ref{a2} and the exponent $\ell$ in \eqref{eq:LK} assume that
\begin{enumerate}[(B1)]
\item $D< \min\{(1-\ell)(p/2-2),\ell(p/4-1)\}$ with $\ell>3/4$ and $p>8$~,
\item $\lim_{n \to \infty} \frac{\log n}{ K (\mu^{\ast}_d)^2}=0$~,
\end{enumerate}
where $\mu^{\ast}_d $ is defined in \eqref{mustar}.
\end{assump}

\noindent
Assumption (B1) is a rather technical condition relating the dimension $d \sim n^D$, the number of blocks $L=n^{\ell}$ and the constant $p$, which was initially introduced in Assumption \ref{a1}.
Assumption (B2) is only a restriction for the monotone decreasing sequence $\mu^{\ast}_d = \min_{h\in \Sd^c}\absdmuh$, that is not allowed to decrease arbitrarily fast.

We are now ready to state the main results of this section.
Our first lemma shows, that we are able to identify the set of stable components $\Sd$ correctly. 

\begin{lemma}\label{lem:BSstableset}
If Assumptions \ref{a1} - \ref{a4} and Assumption \ref{assump:bootstrap} are satisfied, then 
\begin{align*}
a_d \Big ( \max_{h \in \Sd} B_{n,h} - b_d \Big ) \convp 0~,
\end{align*}
where the set $\Sd$ is the set of all components with no change point defined in \eqref{sd} and the sequences $a_d$ and $b_d$ are defined in \eqref{ad} and \eqref{bd}, respectively.
\end{lemma}

\begin{theorem}\label{thm:BSunstableset}
If  Assumptions \ref{a1} - \ref{a4} and Assumption \ref{assump:bootstrap} are satisfied, then
\begin{align*}
a_d \Big ( \max_{h \in \Sd^c} B_{n,h} - b_d \Big )
\convd
\left\{
\begin{array}{rl}
  G &\text{if}\;\; \limd | \Sd^c |/d  = 1,\vspace{0.2cm} \\
  G + \log c &\text{if}\;\; \limd | \Sd^c |/d  = c \text{ for } c \in (0,1)~,\vspace{0.2cm} \\
  -\infty &\text{if}\;\; \limd | \Sd^c |/d  = 0 \\
\end{array} 
\right.
\end{align*}
conditional on $\Zn$ in probability, where the sequences $a_d$ and $b_d$ are defined in \eqref{ad} and \eqref{bd}, respectively.
\end{theorem}
\noindent
Finally, the representation
\begin{align}
\BSTS =  \max\left\{ a_d \left( \max_{h \in \Sd}  B_{n,h} - b_d \right), a_d \left( \max_{h \in \Sd^c} B_{n,h} -b_d \right) \right\}~, 
\end{align}
Lemma \ref{lem:BSstableset} and Theorem \ref{thm:BSunstableset} directly yield the following main result of this section.

\begin{theorem}\label{thm:BSgumbel}
If Assumptions \ref{a1} - \ref{a4} and Assumption \ref{assump:bootstrap} are satisfied, then  
\begin{align*}
\BSTS
\convd
\left\{
\begin{array}{rl}
  G &\text{if}\;\; \limd | \Sd^c |/d  = 1~, \vspace{0.2cm} \\	
  \max \{ G + \log c, 0\}  & \text{if}\;\; \limd | \Sd^c |/d  = c \text{ for } c \in (0,1)~,\vspace{0.2cm} \\
  0 &\text{if}\;\;\limd | \Sd^c |/d  = 0  \\
\end{array} 
\right.
\end{align*}
conditional on $\Zn$ in probability.
\end{theorem}

\begin{remark}\label{rembias}
In the following let $g_{1-\alpha}^* $ denote the $(1-\alpha)$-quantile of the distribution of the bootstrap statistic $\BSTS$ and
define the bootstrap test by rejecting the null hypothesis \eqref{hnull} in favour of  \eqref{halt}, whenever 
\begin{align}\label{eq:BSrejectionrule}
\TS > g_{1-\alpha}^*~,
\end{align} 
where the statistics $\TS $ is defined in \eqref{eq:TS}.
If the alternative hypothesis of at least one relevant change point holds, Theorem \ref{thm:consistent} shows $\TS \convp \infty$, which due to Theorem \ref{thm:BSgumbel} directly yields consistence of the bootstrap test. 
Under the null hypothesis, we consider different cases.
Recall the definition of the sets
\begin{align*}
\Bd &= \{ h \in \setd \;  | \;  \absdmuh = \Delta_h \}~,\\
\Sd &= \{ h \in \setd \;  | \;  \absdmuh = 0 \}~,
\end{align*}
where we always have $\Bd \subset S_d^c \subset \setd$.
A combination of Theorem \ref{thm:level} and Theorem \ref{thm:BSgumbel} now shows, that the rule in \eqref{eq:BSrejectionrule} gives an asymptotic level $\alpha$ test, whenever the limit $\limd | \Sd^c |/d$ exists.
\end{remark}

\section{Finite sample properties}
\label{sec6} 
\def\theequation{6.\arabic{equation}}
\setcounter{equation}{0}

In this section we examine the finite sample properties of the asymptotic and bootstrap test by means of a small
simulation study and  illustrate its application in an example.

\subsection{Simulation study}
The results of the previous section demonstrate that a test which rejects the null hypothesis in \eqref{relevant} for large values of the statistic $\TS $ defined in \eqref{eq:TS} is consistent and has asymptotic level $\alpha$, provided that the critical values are either chosen by the asymptotic theory or estimated by the bootstrap procedure introduced in Section \ref{sec5}.
It turns out that a bias correction, which does not change the asymptotic properties, yields substantial improvements of the finite
sample properties of the asymptotic and bootstrap test.
To be precise, recall the decomposition in \eqref{eq:tnhexpansion}, that is
\begin{align*}
\Tnh = T_{n,h}^{(1)} + T_{n,h}^{(2)}~,
\end{align*}
where $T_{n,h}^{(1)}$ and $ T_{n,h}^{(2)}$ are defined in  \eqref{tn1} and  \eqref{tn2}, respectively. 
It is shown in Section \ref{sec3}, that the maximum of the terms  $T_{n,h}^{(1)}$ is of order $o_{\mathbb{P}}(1)$, while the maximum of the terms
$T_{n,h}^{(2)}$ (appropriately standardized) converges weakly to a Gumbel distribution.
However, $T_{n,h}^{(1)}$ is always nonnegative, which may lead to a non negligible bias
in applications, in particular when the sample size is small relative to the dimension. 
\\
To solve this problem for the asymptotic test \eqref{asymtest}, note that we have
\begin{align*}
\mean{T_{n,h}^{(1)}  } 
&= \frac{3\sigma_h }{\sqrt{n}(t_h (1 - t_h))^2\tau_h  \absdmuh  }\int_0^1k(s,s)ds(1+o(1))\\ 
&= \frac{\sigma_h (1+o(1))}{2\sqrt{n}(t_h (1 - t_h))^2\tau_h  \absdmuh  }
\end{align*}
and therefore we subtract the term
\begin{align*}
\frac{\hatsigmah}{2\sqrt{n}(\hatth (1 - \hatth))^2\tauhatth  \Delta_h  }
\end{align*}
from each statistic $\hatTnhdelta$ defined in  \eqref{eq:tnhdelta}.
Similarly, a bias correction is also suggested for the bootstrap test in Section \ref{sec5}.
Recall that the Bootstrap statistic is already constructed to mimic the distribution of the statistic  $T_{n,h}^{(2)}$.
Consequently we use 
\begin{align*}
\dfrac{3\sqrt{n}}{\widehat{s}_h \tauhatth(\hatth(1-\hatth))^2\Delta_h}\int_0^1\left(\mathbb{U}_{n,h}^{(L)}(s)\right)^2 ds
\end{align*}
to mimic the distribution of $T_{n,h}^{(1)}$ and we add it (for each $h$) to the statistic $B_{n,h}$ defined in \eqref{eq:BSIntegral}, while the statistics $\hatTnhdelta$ in \eqref{eq:tnhdelta} are left unchanged. 

To guarantee a stable long-run variance estimation, we replaced the standard long-run variance estimator used in the theory of Section \ref{sec2s}, by an estimator using the Bartlett kernel [see \cite{newey1987}], that is (for component $h$)
$$\hat{\phi}_{h}(0) +  \sum_{1\leq \vert j \vert \leq \beta_n} k\left( \dfrac{j}{\beta_n} \right)\hat{\phi}_{h}(j)\;\; \text{with}\;\;
k(x)=\left\{
\begin{array}{rl}
	1 - \vert x \vert &\text{if}\;\;\vert x \vert \leq 1 \\
  0 &\text{if}\;\; \vert x \vert > 1 ~,\\
\end{array} 
\right. $$
where the bandwidth is chosen as $\beta_n = n^B$ and $B$ is the constant included in Assumption (S2).
In our examples we have fixed the bandwidth to $\beta_n=8$.
The only exception here are the simulations contained in Table \ref{table1} for the stronger dependent model (III) defined below, for which we use a bandwidth of $\beta_n=60$ in the case $n=2000$ and $\beta_n=80$ in the cases $n=5000$ or $n=10000$. 
In order to have a more conservative test we use the maximum of the two variance estimates based on the sets $\widehat{\mathcal{D}}_{h,1}$ and $\widehat{\mathcal{D}}_{h,2}$ defined in \eqref{dh12} as the final variance estimation.

We will focus on two scenarios with independent innovations in model \eqref{eq:model}, i.e.
\begin{itemize}
\item[]  (I) $X_{j,h} \sim\mathcal{N}(0,1)$ i.i.d. 
\item[] (II) $X_{j,h} \sim \Exp(1)$ i.i.d.
\end{itemize}
and on two models of dependent data, an ARMA-process and a MA-process, defined by 
\begin{itemize}
\item[] (III) $X_{j,h} = 0.2 X_{j-1,h} -0.3X_{j-2,h} -0.4Y_{j,h} + 0.8 Y_{j-1,h}$~,\\
\item[]  where ~~ $Y_{k,h}=\varepsilon_{k}  + \sum_{i=1}^{19}i^{-3} \varepsilon_{k-i,h}$
and $\varepsilon_{k,h} \sim \mathcal{N}(0,1)$ i.i.d.\\
\item[]  (IV) $X_{j,h} = \varepsilon_{j,h} + \frac{1}{10}\sum_{k=1}^{29} k^{-3}\varepsilon_{j-k,h}$, where  $ \varepsilon_{k,h} \sim \mathcal{N}(0,1)$ i.i.d.
\end{itemize}
All innovations are constructed such that  Var$(X_{j,h}) \approx 1$.
Throughout this section we assume that different components are independent and we are interested in testing the relevant hypotheses
\begin{align} \label{h0ex}
& \HNull: \absdmuh \leq 1  \;\text{for all}\; h \in \setd \\
\text{versus }  ~~
&  
\HA: \absdmuh > 1 \;\text{for at least one }\; h \in \setd~,
\end{align}
that is $\Delta_h =1 $ for all $ h\in \setd$.
Power and level of the tests are simulated for $d$-dimensional vectors with means 
$$\mathbb{E} [Z_{j,h}] =  \begin{cases} 0 & \mbox{ if }  j \leq \lfloor nt_h \rfloor \\ 
\mu & \mbox{ if }~  j > \lfloor nt_h \rfloor 
\end{cases}  ~~, ~h=1,\ldots , d,
$$   
where different values of the parameter $\mu$ are considered and the time of change is $t_h = 1/2$.
All results presented in this section are based on $1000$ simulation runs and the used significance level is always $\alpha = 0.05$.
Further, the constant $S$ involved in \eqref{dh12} was fixed to $0.9$.\\
In our first example we investigate the finite sample properties of the asymptotic test \eqref{asymtest} which uses the quantiles of the Gumbel distribution.
In Table \ref{table1} we display the simulated type I error of this test at the ``$0$-dimensional boundary'' of the null hypothesis (that is $\Delta_h=\dmuh=1$ for all $h \in \setd$) for different values of $n$ and $d$. 
It is well known that the approximation of the distribution of the maximum of normally distributed random variables by a Gumbel distribution is not very accurate for small samples and therefore we consider relatively large sample sizes and dimensions in order to illustrate the properties of the asymptotic test.
The results reflect the asymptotic properties in Section \ref{sec3}.
For the independent cases (I) and (II) the approximation of the nominal level is more precise as for the dependent scenarios (III) and (IV), where the test is more conservative.
We also mention that the rejection probabilities are increasing with $ \dmuh $ as predicted in Section \ref{sec2} and \ref{sec4} (these results are not presented for the sake of brevity).
\begin{table}[h]
    \centering
    \begin{tabular}{ccccccc}
        \hline 
        \vspace{2mm} 
        & \multicolumn{2}{c}{\textit{n=2000}}  & \multicolumn{2}{c}{\textit{n=5000}} & \multicolumn{2}{c}{\textit{n=10000}} \\
        \textit{model} & \multicolumn{1}{c}{\textit{d=500}} &
\multicolumn{1}{c}{\textit{d=1000}} & \textit{d=500} &
\multicolumn{1}{c}{\textit{d=1000}} & \multicolumn{1}{c}{\textit{d=500}}
& \textit{d=1000} \\
        \hline  
        (I) 	& 8.8\% 	& 11.1 \% &  6\% 	& 7.5\%	& 6.4\%  & 7.4\% \\
        (II) 	& 5.4\% 	& 5.9  \% &  5\% 	& 5.5\% 	& 4\% 	& 5\% \\
        (III) 	& 4.9\%     & 4.5  \% &  1.9\% 	& 2.5\% 	& 2.5\% 	& 4\% \\
        (IV) 	& 9.7\%     & 13 \% &  6.4\% 	& 5.2\% 	& 6.2\% 	& 5.9\% \\
        \hline
    \end{tabular}
    \smallskip 
\caption{\it Empirical rejection probabilities of the asymptotic test \eqref{asymtest} under a specific point at the ``$0$-dimensional boundary'' of null hypothesis, that is  $\Delta_h = \dmuh = 1$ for all $h \in \setd$. \label{table1}}
\end{table}

\begin{table}[h]
\centering
\begin{tabular}{ccccccc}
\hline 
$\vert \Delta \mu_h \vert$ & & 0.9 & 0.95 & 1.0 & 1.05 & 1.1 \\ 
model & $K$ &  & & & & \\
\hline 
(I) & 1 & 2.2\% & 4.2\% & 8.2\%  & 14.7\% & 26.9\% \\ 
(I) & 4 & 5.2\% & 9.1\% & 17.3\%  & 29.6\% & 47.8\% \\ 
\hline 
(II) & 1 & 0.6\% & 0.9\% & 2.5\%  & 7.8\% & 17.3\% \\  
(II) & 4 & 0.9\% & 3.2\% & 9.2\%  & 20.2\% & 36.6\% \\ 
\hline 
(III) & 5 & 0\% & 0\% & 0.6\%  & 3.2\% & 13.7\% \\ 
(III) & 10 & 0.1\% & 1.3\% & 4.6\%  & 21.6\% & 59.7\% \\ 
\hline 
(IV) & 2 & 3.3\% & 5.9\% & 10.9\%  & 21.5\% & 35.8\% \\ 
(IV) & 4 & 5\% & 10.3\% & 18.4\%  & 32.5\% & 48.7\% \\ 
\hline 
\end{tabular}
\smallskip

\caption{\it Empirical rejection probabilities of the bootstrap test  \eqref{eq:BSrejectionrule} for $n = d = 100$  
at specific points in the alternative and null hypothesis
($\Delta_h  = 1 $  for all $h \in \setd$). Different block length $K$ in the multiplier bootstrap are considered.
\label{table2}}
\end{table}

Next we analyse the properties of the bootstrap test \eqref{eq:BSrejectionrule}, which was developed in Section \ref{sec5}. Here we focus on relatively small sample sizes, that is $n=100$, $n=200$ and a moderate  dimension compared to the sample sizes, that is $d=50$, $d=100$.
It is well known that the multiplier bootstrap is sensitive with respect to the choice of the block length and this dependence is also observed for the bootstrap test proposed here.
Exemplarily we show in Table \ref{table2} the simulated rejection probabilities
for the different models (I) - (IV), different values of $K$ and $\dmuh=\mu$. Here the values $ \absdmuh \leq 1$ correspond to the null hypothesis.
For $ \absdmuh=1$ (for all $h \in \setd$) the results of Section \ref{sec5} predict that at this point the level of the test should be close to $\alpha=0.05$.
Note that for the case of independent innovations (model (I) and (II)) the choice $K=1$ (which means that no blocks are used) leads to the most reasonable results, given by rejection probabilities on the ``$0$-dimensional boundary'' ${\cal A}_d$ of the null hypothesis of $8.2\%$ and $2.5\%$, respectively, while larger values of $K$ such as $K=4$ yield a too large type I error. 
On the other hand in the dependent models (III) and (IV), the block length needs to be carefully adapted to the time series structure.
For model (III) $K=10$ seems to be optimal, while for model (IV) choosing $K=2$ gives acceptable results.
The larger block length for model (III) might be required due to its autoregressive structure.
Moreover, inspecting the results in rows 6 and 8 of Table \ref{table2} shows that too small values of $K$ lead to a loss of power, while - similar to the first two models - too large values can cause an uncontrolled type I error.
\begin{figure}[t]
\begin{center}
\includegraphics[width=6cm]{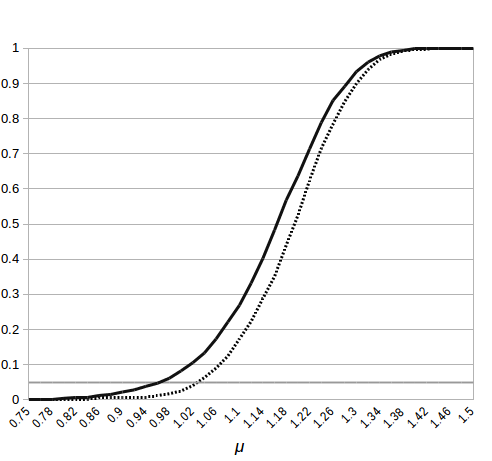}
\includegraphics[width=6cm]{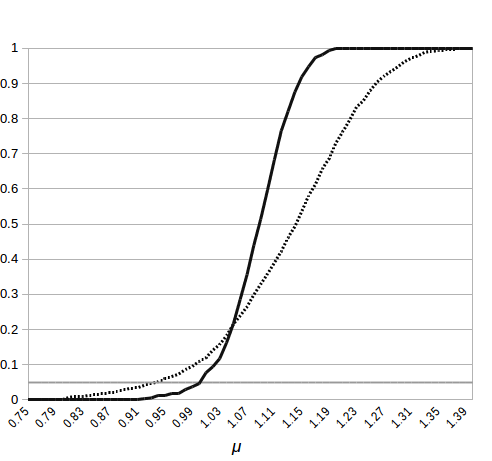}
\caption{ \it Simulated power of the Bootstrap test \eqref{eq:BSrejectionrule} for the hypothesis \eqref{h0ex}.
The sample size is  $n=100$ and the dimension $d=100$. 
All differences are given by $\dmuh = \mu $, where the choice $\mu =1$ corresponds to a point of the ``$0$-dimensional boundary''  ${\cal A}_d$ of the hull hypothesis.
Left panel: Solid line: Model (I), dashed line: Model (II).
Right panel: Solid line Model (III), dashed line: Model (IV).
\label{powerplots}} 
\end{center}
\end{figure}
Next we display in Figure \ref{powerplots} the simulated power of the bootstrap test for all four models under consideration (where we use the optimal $K$ from Table \ref{table2}). 
Note that the rejection probabilities are increasing with $ \dmuh $ as  predicted by the asymptotic arguments of the previous sections. 
In the left panel we show the results for the independent scenarios (I) and (II), which are rather similar.
On the other hand an inspection of the right panel shows larger difference in the dependent case. 
The different dependency structures in model (III) and (IV) yield substantial differences in power of the bootstrap test \eqref{eq:BSrejectionrule}. 
We conclude the discussion of the bootstrap test investigating the sample size $n=200$.
The corresponding results are presented in Table \ref{table4} for the dimensions $d=50$ and $d=100$.
\begin{table}[h]
\centering
\begin{tabular}{ccccccc }
$\vert \Delta \mu_h \vert$ & & 0.9 & 0.95 & 1.0 & 1.05 & 1.1 \\
model & $K$ &  & & & & \\
\hline
 $d=50$ \\
(I)   &  1 & 0.5\% & 2.3\% & 5.1\% &  14.7\% & 34.6\% \\
(II)  &  1 & 0.3\% &   1\% & 4.1\% &    11\% & 30.4\% \\
(III) & 20 &   0\% & 0.6\% & 7.6\% &  46.6\% & 94.8\% \\ 
(IV)  &  2 &   0\% &   0\% & 9.5\% & 47.56\% &  100\% \\  
\hline
$d=100$\\
(I) & 1 &  0.9\% & 2.3\% & 6.9\%  & 18\% & 40.7\% \\
(II) & 1 &  0\% & 1.2\% & 3.7\%  & 10.6\% & 26.6\%  \\
(III) & 20 & 0\% & 0.3\% & 6.4\%  & 51\% & 95.3\% \\
(IV) & 2 & 0\% & 0\% & 13.1\%  & 100\% & 100\%  \\
\end{tabular}
\caption{\it Empirical rejection probabilities of the bootstrap test \eqref{eq:BSrejectionrule} at specific points in the alternative and null hypothesis ($\Delta_h  = 1 $ for all $h \in \setd$).
The level is $5\%$, and the sample size is $n=200$.\label{table4}}
\end{table}

We also consider an example, where the dimension is substantially larger than the sample size.
In Table \ref{table5} we display the simulated size of the bootstrap test \eqref{eq:BSrejectionrule} at the ``$0$-dimensional boundary'' of null hypothesis, that is  $\Delta_h = \dmuh = 1$ for all $h \in \setd$.
We consider again the models (I) - (IV),  the sample size is $n=100$ and the dimension is given by $d=500$, $1000$ and $10000$.
We observe a reasonable approximation of the nominal level in models (I) - (III), while the nominal level is overestimated in model (IV).
\begin{table}[h]
\centering
\begin{tabular}{ccccc}
model  & $K$ & $d=500$ & $d=1000$ & $d=10000$  \\
\hline
(I)   &  1 & 5.7\% &  5.8\% & 5.8\%\\
(II)  &  1 & 2.4\% &  3.0\% & 2.7\%\\
(III) & 10 & 4.4\% &  4.4\% & 4.1\%\\
(IV)  &  2 & 9.4\% & 12.4\% & 13.5\%
\end{tabular}
\caption{\it Type I error of the Bootstrap test \eqref{eq:BSrejectionrule} at ``$0$-dimensional boundary'' of null hypothesis, that is $\Delta_h = \dmuh = 1$ for all $h \in \setd$.  
The level is $5\%$, the  sample size is $n=100$ and the dimension is given by $d=500$, $1000$ and $10000$ .\label{table5}}
\end{table}

We conclude with a small study of the robustness aspects of the new procedure, which is motivated by a comment of a referee on an earlier version of this paper.
Note that the asymptotic test \eqref{asymtest} is constructed for the case, where there is at most one change point in each component of the series.
In the final example of this section we will therefore investigate the performance of the test in the situation, where there
are multiple small changes in the components of the series accumulating to a magnitude larger than the given thresholds $\Delta_h$.
To be precise we have conducted the asymptotic test \eqref{asymtest} at level $\alpha=5\%$ for the hypotheses \eqref{h0ex} (with  $\Delta_h=1$ for all $h =1 , \ldots d$).
The empirical rejection probabilities are displayed in Figure \ref{figrobust} for the model (I).
The left panel of the figure corresponds the  case  $n=500$, $d=500$, where there are  $10$  ``small'' changes 
(of the same size) in each component of the mean function at equidistant locations summing up to a total change amount of $\Delta\mu_h$.
The right panel of the figure shows corresponding results for the case $n=100$, $d=100$ and $5$ changes (of the same size) in each component at equidistant locations.
We observe that the test is still able to detect deviations in the mean function if the total amount of change is sufficiently large.
\begin{figure}[t]
\begin{center}
\includegraphics[width=6cm]{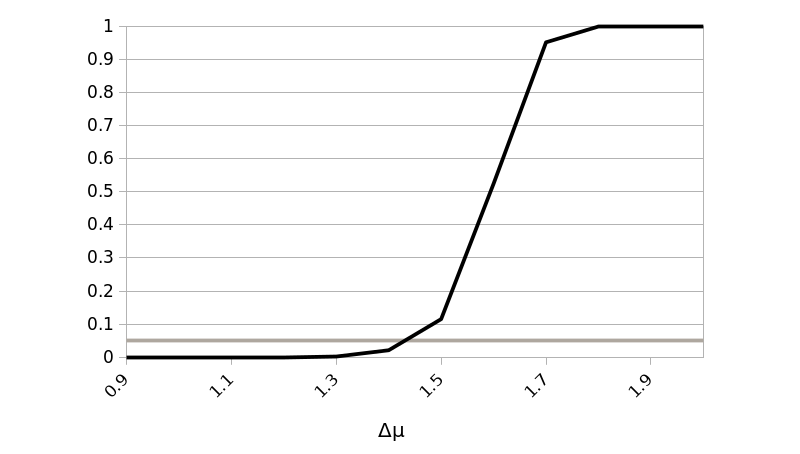}
\includegraphics[width=6cm]{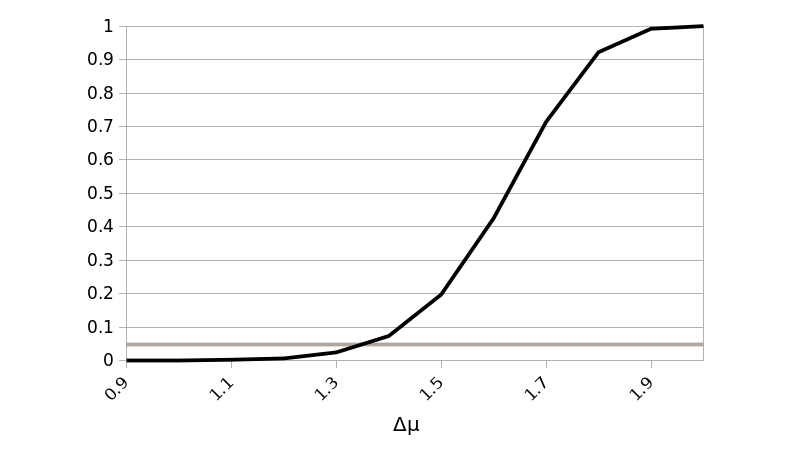}
\caption{ \it  
Simulated power of the asymptotic test \eqref{asymtest} for the hypothesis \eqref{h0ex} in a situation with multiple change points.
Left panel: sample size $n=100$, dimension $d=100$ and $10$ ``small'' changes in each component, where the  
total amount of change is equal to $\Delta\mu_h$. Right panel: $n=500$, $d=500$ and $5$ change points.} 
\label{figrobust}
\end{center}
\end{figure}

\subsection{Data example}\label{secrealdata}
In this section we illustrate in a short example, how the new test can be used in applications. Our dataset is taken from hydrology and consists of average daily flows ($m^3/$sec) of the river Chemnitz at Goeritzhain (Germany) in the period 1909-2014.
This data set has been recently analysed by  \cite{sharipov:2016} using a statistical model from functional data analysis.
Following these authors we subdivide the data into $n=105$ years with  $d={365}$ days per year.
To avoid confusion, the reader should note that the German hydrological year starts on the $1$st of November.

\begin{figure}[p]
\includegraphics[width=13cm,height=4.5cm]{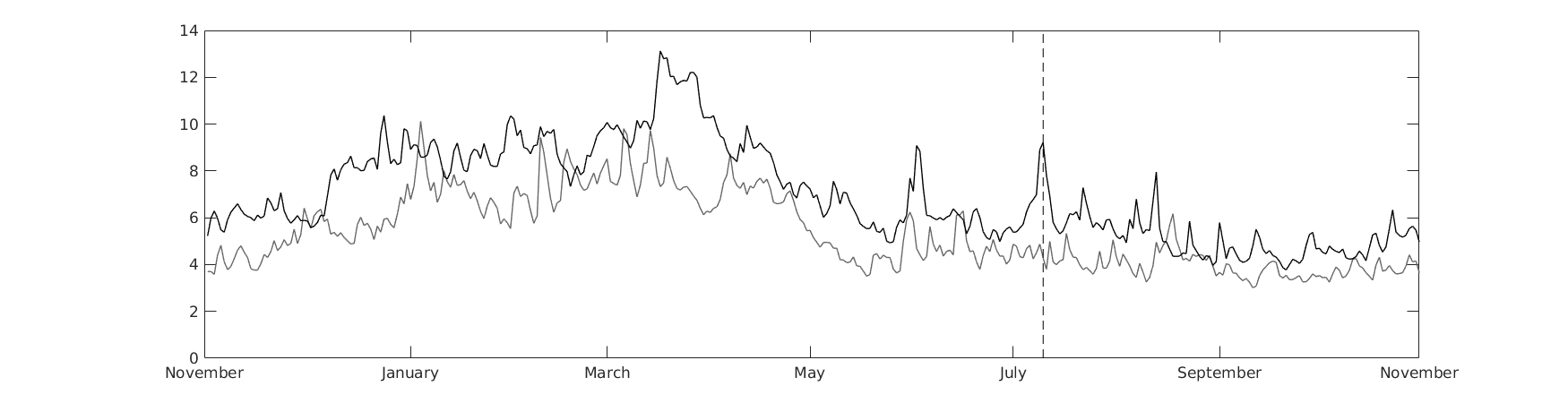}\\
\includegraphics[width=13cm,height=4.5cm]{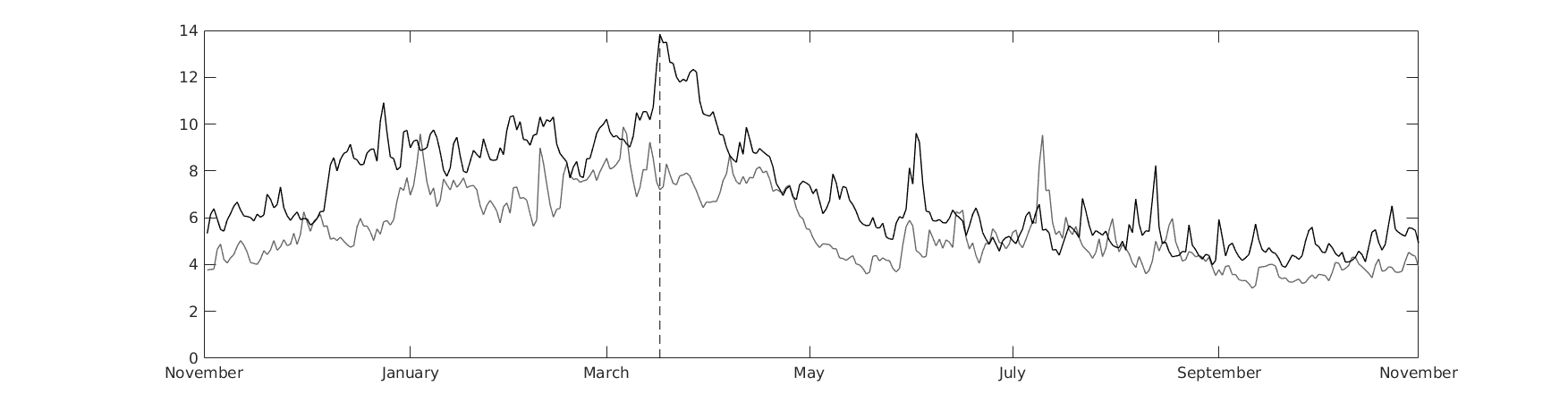} \\ 
\includegraphics[width=13cm,height=4.5cm]{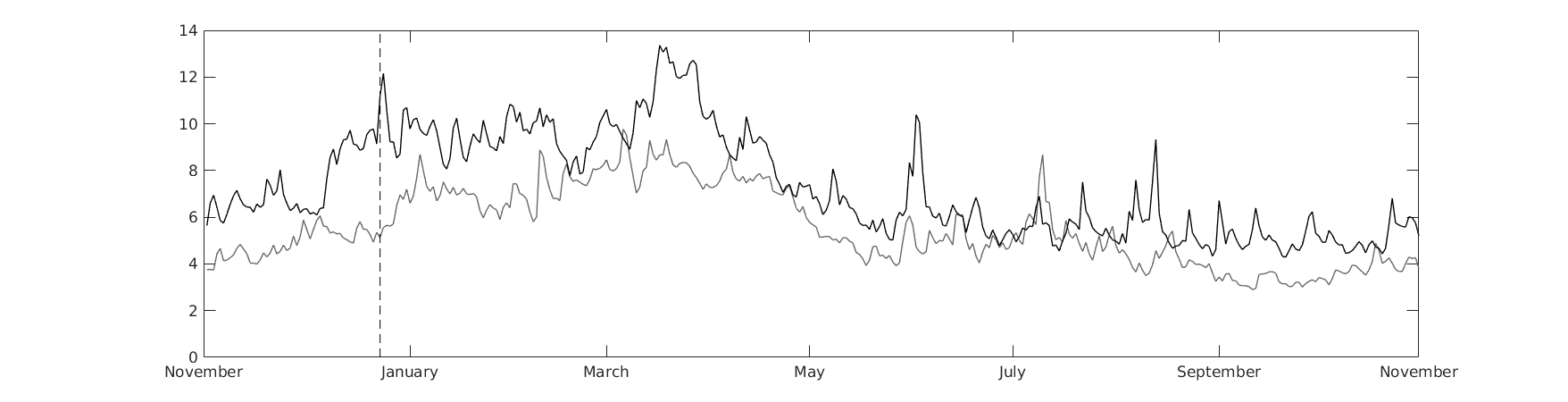}\\
\includegraphics[width=13cm,height=4.5cm]{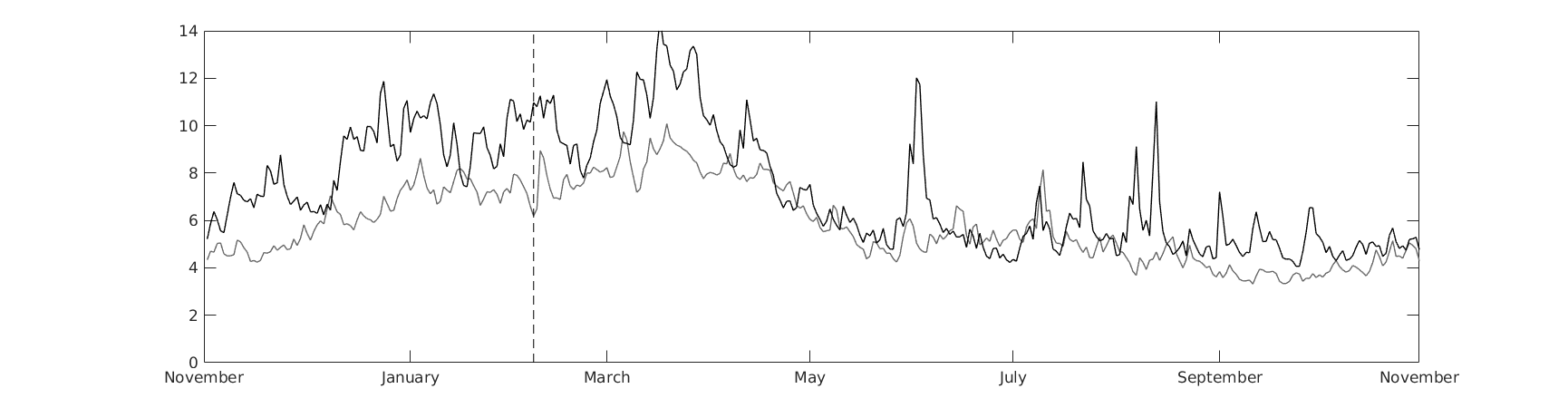}
\caption{\it Average annual flow curves for the river Chemnitz at Goeritzhain for the periods before  (gray) and  after (black) the estimated year of change.
The four figures correspond to different days of the year, where a change point has been localised: 
$10$th of July (first row), $18$ of March(second row) $23$th of December (third row)
and $7$th of February (fourth row).
\label{averageflow}} 
\end{figure}

Equipped with our new methodology, we are now able to test if there is a relevant change in the mean of at least one component. To specify the term 'relevant', we exemplarily set the thresholds for all components to
\begin{align*}
\Delta_1=\Delta_2=\dots=\Delta_{365}=0.63,
\end{align*}
which is close to $10\%$ of the overall mean of the data under consideration.
For a significance level of $5\%$ the Bootstrap test defined in Section \ref{sec5} rejects the null hypothesis of no relevant change for the given thresholds. 
Moreover, we can also identify the components, where the individual test statistic leads to a rejection, that is $a_d(\hatTnhdelta - b_d) > g^*_{1-\alpha} $,
where $g^*_{1-\alpha} $ is the $({1-\alpha} )$ quantile of the bootstrap distribution used in \eqref{eq:BSrejectionrule}.
For the data under  consideration we found four components with a relevant change, given by $h=53$, $ 99$, $137$  and $ 252 $ with corresponding estimators  
$n\hat{t}_{53}=56$, $n\hat{t}_{99}=70$,
$n\hat{t}_{137}=47$ and $n\hat{t}_{252}=41$, respectively. This corresponds to the $23$th of December $1965$,
the $7$th of February $1979$, the $18$th of March $1956 $ and to the $10$th of July $1950$, respectively.
In  Figure \ref{averageflow} we display for these cases the time series before and after the year. For example the panel in the first row shows the average annual flow curves before and after $1950$ and the other three years are interpreted separately.
In all cases we observe a large difference between both curves close the estimated component (marked by the vertical dashed line).
We finally note that the approach of \cite{sharipov:2016} identifies only one change point, namely the year $1965$.
In contrast our analysis indicates that there might be additional change points in the years $1950$, $1956$, $1965$ and $1979$ corresponding to different parts of the hydrological year. 

\section*{Acknowledgments}
This work has been supported in part by the Collaborative Research Center ``Statistical modeling of nonlinear
dynamic processes'' (SFB 823, Teilprojekt  A1, C1) and the Research Training Group 'High-dimensional phenomena in probability - fluctuations and discontinuity' (RTG 2131) of the German Research Foundation (DFG).
The authors would like to thank Martina Stein, who typed parts of this paper with considerable technical expertise and Moritz Jirak for extremely helpful discussions. Moreover we are grateful to Andreas Schumann and Svenja Fischer from the Institute of Hydrology in Bochum, who provided hydrological data for Section \ref{secrealdata}.
Finally, we would like to thank the referees for their constructive comments on an earlier version of the paper.

\setlength{\bibsep}{3pt}

\appendix
\newpage 

\noindent Throughout the proofs we will use that Assumption (C1) directly implies the existence of two constants $\tau_-, \tau_+$, such that the function $\tau$ defined in \eqref{tauh} satisfies
\begin{align*}
0< \tau_- \leq \tau(t_h) \leq \tau_+ < \infty
\end{align*}
holds for all $h \in \mathbb{N}$.

\section{Proofs of the results in Sections 2 and 3}\label{proofsec23}
\subsection{Proof of assertion (\ref{motiv})} \label{sec71}
Straightforward calculations yield
\begin{align}\label{cusumsparam}
\begin{split}
\mean{ \cusums } &= \mu_h(s,t_h) + O(n^{-1})\dmuh\\
n\Var{ \cusums } &= \sigma_h^2(s - s^2) + o(1),
\end{split}
\end{align}
uniformly with respect to $s \in [0,1]$, where $\mu_h(s,t_h)$ is defined in \eqref{mu}.
An application of Fubini's theorem now gives
\begin{align*}
\mathbb{E}  \Big [ {\int_0^1 \cusumssq ds} \Big] 
&= \int_0^1 \mean{\cusumssq} ds\\
&= \int_0^1  \big \{  \mean{( \cusums - \mu_h(s,t_h) )^2}\\  
&\;\;\;+ 2 \mu_h(s,t_h)  \mean{ \cusums  } -  \mu_h^2(s,t_h) \big\} ds + o(1)\\
&= \dfrac{\sigma_h^2}{6n} + \int_0^1 \mu_h^2(s,t_h) ds + o(1)
= \dmuh \dfrac{(t_h(1-t_h))^2}{3} + o(1).
\end{align*}

\subsection{Details in the proof of Theorem  \ref{cor:tnhgumbel}}

\subsubsection{Proof of Lemma \ref{lem:tnh1}} \label{sec711}
Observing the definition \eqref{tauh} and Assumptions (T2) and (C1) it easily follows that there exist constants $c$ and $C$, such that the inequalities 
\begin{align} \label{lem:auxiliary1}
c < (t_h(1-t_h))^2\sigma_h < C
\end{align}
hold for all $h \in \setd$. With these inequalities  we obtain
\begin{align*}
0
&\leq a_d \maxhd T_{n,h}^{(1)}\\
&=a_d \maxhd \frac{3\sqrt{n}}{(t_h (1 - t_h))^2\tau_h\sigma_h\absdmuh} \int_0^1 (\mathbb{U}_{n,h}(s) - \mu_h(s,t_h))^2 ds\\
&\lesssim a_d \maxhd \sqrt{n}\int_0^1 \left( \cusums - \mu_h(s,t_h) \right)^2 ds\\
&\leq \frac{a_d}{\sqrt{n}} \maxhd \sup_{s\in [0,1]} (\sqrt{n}|\cusums - \mu_h(s,t_h)|)^2.
\end{align*}
Using \eqref{cusumsparam} and $\maxhd \vert \dmuh \vert \leq C_u$ this is further bounded by
\begin{align*}
2\frac{a_d}{\sqrt{n}} \Big(\maxhd \sup_{s\in [0,1]}\sqrt{n}|\cusums - \mean{\cusums}| \Big)^2 + o(1)
\end{align*}
Introducing the notation $e_d=2\sqrt{2\log 2d}$ it follows that the last term can be bounded by the random variable
\begin{align*}
4\frac{a_d}{\sqrt{n}} \Big(\maxhd \sup_{s\in [0,1]}\sqrt{n}|\cusums - \mean{\cusums}| - \frac{e_d}{4}\Big)^2 + o(1)
\end{align*}
and the claim follows from Theorem 2.5 in \citep{jirak:2015} noting that $\{ \cusums - \mean{\cusums} \}_{s\in [0,1]}$ corresponds to a CUSUM-process under the classical null hypothesis of no change point, that is $\mu_h^{(1)} = \mu_h^{(2)}$ for all $h \in \setd$ (note that there is a typo in the original paper, which has been corrected in the arXiv version).

\subsubsection{Proof of Lemma \ref{lem:gaussappro}}
\label{sec712}

A straightforward calculation yields
\begin{align*}
n\Cov(\cusum{s_1},\mathbb{U}_{n,i}(s_2))=k(s_1,s_2)\gamma_{h,i} + r_{n,h,i}(s_1,s_2),
\end{align*}
where $k(s_1,s_2) = s_1 \wedge s_2 -s_1s_2$ denotes the covariance kernel of a Brownian bridge and the remainder term satisfies
\begin{align*}
\limn \sqrt{n} \sup_{s_1,s_2\in [0,1]} \sup_{h,i \in \mathbb{N}} \vert r_{n,h,i}(s_1,s_2) \vert = 0.
\end{align*}
An application of Fubini's theorem shows that 
\begin{align}\label{eq:covtnh2}
\Cov \left( \bar{T}_{n,h}^{(2)}, \bar{T}_{n,i}^{(2)} \right) = \frac{\sign(\dmuh)\sign(\Delta\mu_i)\gamma_{h,i}\widetilde{\tau}(t_h,t_i)}{\sigma_h\sigma_i\tau(t_h)\tau(t_i)} + r_{n,h,i},
\end{align}
where the function $\widetilde{\tau}$ is given by 
\begin{align}\label{eq:widetildetau}
\widetilde{\tau}(t,t')=\frac{36}{(t(1-t)t'(1-t'))^2}
\int_0^1\int_0^1k(s_1,t)k(s_2,t')k(s_1,s_2)ds_1ds_2
\end{align}
and the remainder term $r_{n,h,i}$ satisfies 
\begin{align*}
\limn \sqrt{n}\sup_{h,i \in \mathbb{N}} \vert r_{n,h,i} \vert = 0.
\end{align*}
Moreover, for the function $\tau$ defined in \eqref{tauh} we obtain the representation  
\begin{align*}
\tau(t)= \frac{6}{(t(1-t))^2} \sqrt{\int_0^1\int_0^1k(s_1,t)k(s_2,t)k(s_1,s_2)ds_1ds_2},
\end{align*}
and it follows that $\tau(t_h)\tau(t_h)=\widetilde{\tau}(t_h,t_h)$.
Therefore we obtain as a special case the estimate
\begin{align}\label{eq:vartnh2}
\Var{ \bar{T}_{n,h}^{(2)} } = 1 + r_{n,h,h}~.
\end{align}
Furthermore the representation 
\begin{align*}
\cusums -\mean{\cusums}
= \dfrac{1}{n}\sum_{j=1}^{\floor{ns}}X_{j,h} - \dfrac{\floor{ns}}{n^2}\sum_{j=1}^n X_{j,h}
=: \mathbb{U}_{n,h}'(s)
\end{align*}
yields
\begin{align*}
\bar{T}_{n,h}^{(2)} &= \frac{6\sqrt{n}}{(t_h (1 - t_h))^2\tau_h\sigma_h\absdmuh} \int_0^1 \mu_h(s,t_h) \mathbb{U}_{n,h}'(s) ds~.
\end{align*} 
and the proof can now be performed in two steps:
\begin{enumerate}[{{Step }1{:}}]
\item For two constants $c,C>0$ it holds that
\begin{align}\label{eq:gaussapprostep1}
\sup_{x \in \mathbb{R}}\bigg\vert\Pb{\maxhd \bar{T}_{n,h}^{(2)} \leq x } - \Pb{\maxhd \widetilde{N}_h \leq x }\bigg\vert \leq Cn^{-c}~,
\end{align}
where $\widetilde{N}$ is a $d$-dimensional centered Gaussian distributed random variable with same covariance structure as $(\bar{T}_{n,1}^{(2)},\dots,\bar{T}_{n,d}^{(2)})^T$.
\item There exist two constants $c,C>0$ such that
\begin{align}
\sup_{x \in \mathbb{R}}\bigg\vert\Pb{\maxhd \widetilde{N}_h \leq x } - \Pb{\maxhd N_h  \leq x }\bigg\vert \leq Cn^{-c}~,
\end{align}
where $N$ is a centered $d$-dimensional Gaussian random variable with covariance matrix $\Sigma =(\Sigma_{h,i})_{i,j=1,\ldots, d}$ satisfying
\begin{align} \label{claim} 
\left\vert \Sigma_{h,i} \right\vert \leq \vert \rho_{h,i} \vert
\end{align}
for all $h,i \in \setd$.
\end{enumerate}
\textbf{Step 1:} At the end of this proof we derive the following representation
\begin{align}\label{eq:cnhj}
\bar{T}_{n,h}^{(2)} = \frac{1}{\sqrt{n}}\sum_{j=1}^n c_{n,j,h}X_{j,h}~,
\end{align}
where the coefficients $c_{n,j,h}$ are uniformly bounded, that is 
\begin{align*}
\sup_{n,j,h \in \mathbb{N}} \vert c_{n,j,h} \vert \leq c_0 < \infty
\end{align*}
Next we apply the Gaussian approximation in Corollary 2.2 of \cite{zhang:2016} to the random variables $Y_{n,j,h} = c_{n,j,h}X_{j,h}$. 
For this purpose we check the assumptions of this result.
By Assumption (M1) we obtain a sequence $(M_d')_{d\in \N} $ by $M_d'= c_0 \cdot M_d$, which still satisfies
\begin{align*}
\max_{j=1}^n \maxhd \mean{ \exp (\vert Y_{n,j,h} \vert /M_d')}
\leq \max_{j=1}^n \maxhd \mean{ \exp (\vert X_{j,h} \vert /M_d)}
\leq C_0~.
\end{align*}
Moreover, Assumption (M2) yields that $M_d' \lesssim n^{m}$ with $m<3/8$.
This means that for sufficiently small $b$ we have $m < (3-17b)/8$ and by Assumption (S1) it follows that $d \lesssim \exp(n^b)$.
Using the identity \eqref{eq:Xphysical} the triangular array $ \{ Y_{n,j,h} \, \vert \, 1 \leq j \leq n, 1 \leq h \leq d \}_{n \in \mathbb{N}}$ exhibits the following structure
\begin{align*}
Y_{n,j,h}
= c_{n,j,h}\cdot g_h(\varepsilon_j,\varepsilon_{j-1},\dots)
:= \tilde{g}_{n,j,h}(\varepsilon_j,\varepsilon_{j-1},\dots).
\end{align*}
Define the coefficients
\begin{align*}
\vartheta^{Y}_{n,j,h,p} = \max_{i=1}^n\|\tilde{g}_{n,i,h}(\varepsilon_i,\dots) - \tilde{g}_{n,i,h}(\varepsilon_i,,\dots,\varepsilon_{i-j+1},\varepsilon_{i-j}',\varepsilon_{i-j-1},\dots)  \|_p,
\end{align*}
where $\varepsilon_{i-j}'$ is an independent copy of $\varepsilon_{i-j}$, and observe the inequality 
\begin{align*}
\sum_{j=u}^{\infty} \maxhd \vartheta_{n,j,h,p}^Y
\leq \sum_{j=u}^{\infty} \maxhd \sup_{n \in \mathbb{N}} \max_{i=1}^n \vert c_{n,i,h} \vert \vartheta_{j,h,p}
\lesssim c_0 \sum_{j=u}^{\infty} \delta^j \lesssim \delta^u,
\end{align*}
which holds uniformly with respect to  $n$.
By \eqref{eq:vartnh2} there exist constants $c_1$ and $c_2$, such that
\begin{align*}
0<c_1 <  \minhd \Var{ \bar{T}_{n,h}^{(2)}} \leq  \maxhd \Var{ \bar{T}_{n,h}^{(2)}} <  c_2
\end{align*}
if $n$ is sufficiently large.
Since all requirements are met, Corollary 2.2 of \cite{zhang:2016} implies the existence of a Gaussian random variable $\widetilde{N}$ having the same covariance matrix as the vector $(\bar{T}_{n,1}^{(2)},\dots,\bar{T}_{n,d}^{(2)})^T$ and satisfying inequality \eqref{eq:gaussapprostep1}.
\\\\
\textbf{Step 2:}
We choose the random variable $N$ to be $d$-dimensional centered Gaussian with covariance matrix given by
\begin{align*}
\Sigma_{h,i}=\frac{\sign(\dmuh)\sign(\Delta\mu_i)\gamma_{h,i}\widetilde{\tau}(t_h,t_i)}{\sigma_h\sigma_i\tau(t_h)\tau(t_i)}=\sign(\dmuh)\sign(\Delta\mu_i)\rho_{h,i}\frac{\widetilde{\tau}(t_h,t_i)}{\tau(t_h)\tau(t_i)},
\end{align*}
where the function $\widetilde{\tau}$ is defined in equation \eqref{eq:widetildetau}.
Next denote with $\widetilde{\Sigma}$ the covariance matrix of the vector $\widetilde{N}$ from Step 1.
By \eqref{eq:covtnh2} we have
\begin{align*}
\theta_d &
:= \max_{h,i=1}^d \vert \Sigma_{h,i} - \widetilde{\Sigma}_{h,i}\vert \\
&=  \max_{h,i=1}^d \left\vert  \frac{\sign(\dmuh)\sign(\Delta\mu_i)\gamma_{h,i}\widetilde{\tau}(t_h,t_i)}{\sigma_h\sigma_i\tau(t_h)\tau(t_i)}-\Cov(T_{n,h}^{(2)},T_{n,i}^{(2)})\right\vert \lesssim n^{-1/2}
\end{align*}
and an application of the Gaussian comparison inequality, in Lemma 3.1 of \cite{chernozhukov:2013} gives
\begin{align*}
\sup_{x \in \mathbb{R}}\Big \vert \Prob \Big ( { \maxhd N_h \leq x} \Big) - \Prob \Big( { \maxhd \widetilde{N}_h \leq x}  \Big) \Big \vert
\lesssim \theta_d^{1/3} \max\{1, \log(d/\theta_d)\}^{2/3}\lesssim n^{-C}.
\end{align*}
The proof of Step 2 is now completed observing the bound $ | \widetilde{\tau}(t_h,t_i)  |  \leq  | \tau(t_i) | |\tau(t_h)|$, which is a consequence of the (generalised) Cauchy-Schwarz inequality.

\medskip

\noindent
{\bf Proof of the representation (\ref{eq:cnhj}).} Recall the definition $k(s,t)=  s \wedge t  - st$, then 
\begin{align*}
\int_0^1 k(s,t_h) \cusums &ds
= \frac{1}{n}\sum_{i=0}^{n-1} \int_{i/n}^{(i+1)/n} k(s,t_h) \Big (\sum_{j=1}^{\floor{ns}} X_{j,h} - \frac{\floor{ns}}{n}  \sum_{j=1}^{n}X_{j,h} \Big ) ds \\
&= \frac{1}{n}\sum_{i=0}^{n-1} \int_{i/n}^{(i+1)/n} \sum_{j=1}^n k(s,t_h)\left(I\{j\leq i\} - i/n \right) X_{j,h} ds \\
&= \frac{1}{n}\sum_{j=1}^n\left( \sum_{i=0}^{n-1} \int_{i/n}^{(i+1)/n}k(s,t_h)\left(I\{j\leq i\} - i/n \right)ds \right)X_{j,h}.
\end{align*}
Now observe the representation for $\bar{T}_{n,h}^{(2)}$ in \eqref{tn2bar}, where\\ $\mu_h(s,t) =  2 \Delta \mu_h  k(s,t)$.
Define 
\begin{align*}
c_{n,j,h} := \frac{6\dmuh}{(t_h(1-t_h))^2\tau(t_h)\sigma_h\absdmuh}\sum_{i=0}^{n-1} \int_{i/n}^{(i+1)/n}k(s,t)\left(I\{j\leq i\} - i/n \right)ds~, 
\end{align*}
then the representation \eqref{eq:cnhj} holds, and the proof is completed observing the inequalities
\begin{align*}
|c_{n,j,h}| \leq  \frac{6}{\tau_{-}\sigma_-\bt^4} \sum_{i=0}^{n-1}\int_{i/n}^{(i+1)/n} 2 ds
\leq c_0 := \frac{12}{\tau_{-}\sigma_-\bt^4}~.
\end{align*}

\subsubsection{Proof of Lemma \ref{lem:hatth}}
For the proof of Lemma \ref{lem:hatth} we will require Theorem 3.1 from \cite{jirak:2015}, which we state here for the sake of readability and completeness.
Note that our assumptions imply those, used in the reference.
\begin{theorem}\label{thm:jirakhatth}
If Assumptions \ref{a1} - \ref{a4} hold, then
\begin{align*}
\Pb{\maxhSdc |\hatth - t_h| \geq x} \lesssim |\Sd^c| (x n \log n )^{-p/2+2}~,
\end{align*}
provided that 
$$
x \geq C_a \dfrac{\log(n)}{(\mu^{\star}_d)^2n},
$$
where $C_a>0$ denotes a constant only depending on $\bt$ and the sequence $\big\{\sup_{h \in \mathbb{N}}\vartheta_{j,h,p}\big\}_{j \in \mathbb{N}}$~.
\end{theorem}
\noindent Now we can proceed to the actual proof of Lemma \ref{lem:hatth}.\\\\
Let us start with \textit{(i)}: Define $\tilde{C}:= \min\big\{\tfrac{-D}{p/2-2}+1,~C\big\}$.  
Now fix $\varepsilon>0$ and observe that for $n$ sufficiently large, we obtain by Assumption \eqref{eq:limsupmustar}
\begin{align*}
n^{-\tilde{C}}\varepsilon \geq C_a \dfrac{\log(n)}{(\mu^{\star}_d)^2n}~,
\end{align*}
where $C_a$ is the constant involved in Theorem \ref{thm:jirakhatth}.
Thus an application of Theorem \ref{thm:jirakhatth} gives
\begin{align*}
\Pb{ \maxhSdc |\hatth - t_h| \geq \varepsilon \cdot n^{-\tilde{C}} }
&\lesssim |\Sd^c| \Big(n^{-\tilde{C}}n \log(n)\Big)^{-p/2+2}\\
&\leq |\Sd^c|n^{-D} \Big(\log(n)\Big)^{-p/2+2} = o(1)~.
\end{align*}
To prove \textit{(ii)}, one can repeat the proof of \textit{(i)} with $\tilde{C}=\tfrac{-D}{p/2-2}+1$~.

\subsubsection{Proof of Lemma \ref{lem:hatvar}}
For a proof of this result, we have to introduce some additional notation.
For the random variables $\{\varepsilon_j\}_{j \in \mathbb{Z}}$, that drive the physical system defined in Section \ref{sec3}, let
\begin{align*}
\F_j = \sigma(\varepsilon_j,\varepsilon_{j-1},\dots)
\end{align*}
define the canonical filtration.
Further define projections by
\begin{align*}
\Pro_j( \cdot) = \E{\;\cdot\; | \F_j } - \E{ \;\cdot\; | \F_{j-1} }~.
\end{align*}
Since the sequence $\{\varepsilon_j\}_{j \in \mathbb{Z}}$ is i.i.d. it follows  that $\F_{-\infty}  = \cap_{j \in \mathbb{Z}} \F_{j}$ is $\mathbb{P}$-trivial.
So the backwards martingale convergence theorem yields for the centered random variables $X_{j,h}$ defined in \eqref{eq:model} that 
\begin{align*}
X_{n,h} = \sum_{j=0}^{\infty}  \Pro_{n-j}X_{n,h}~.
\end{align*}
Note, that Jensen's inequality implies that $\| \Pro_j(X_0) \|_p \leq \vartheta_{j,h,p}$, which we will frequently apply in the sequel.
In a first step, we will consider estimators of the autocovariances based on the non-observable, centered random variables $X_j$.
For this purpose introduce
\begin{align}\label{eq:autocovarstar}
\hat{\phi}^*_h(i,k) = \dfrac{1}{k-i} \sum_{j=i+1}^{k} \big(X_{j,h} - \bar{X}_h(k)\big)\big(X_{j-i,h} - \bar{X}_h(k)\big)
\end{align}
as lag $i$ autocovariance estimator in (spatial) component $h$ based on the sample $X_{1,h},X_{2,h},\dots,X_{k,h}$ and let $\bar{X}_h(k)$ denote its sample mean.
We get the following uniform consistency result.
\begin{lemma}\label{lem:uniformautocov}
Grant Assumptions \ref{a1} - \ref{a3} and let $c \in (0,1)$ be a fixed constant.
It holds that
\begin{align*}
\sqrt{n} \max_{i=1}^{\beta_n}\maxhd \max_{k=c\cdot n}^n \Vert \hat{\phi}_h(i,k) - \phi_h(i)) \Vert_{p/2}
= \mathcal{O}(1)~. 
\end{align*}
\end{lemma}
\begin{proof}
Throughout the proof assume that $n$ is sufficiently large, such that
\begin{align*}
k-i \geq c \cdot n - \beta_n = c \cdot n - n^B > 0~.   
\end{align*}
We will use the upper bound
\begin{align}\label{eq:ufac1}
\begin{split}
\|\hat{\phi}_h(i,k) - &\phi_h(i)\|_{p/2}  \\
&\leq \Big\Vert \dfrac{1}{k-i}\sum_{j=i+1}^{k}X_{j,h}X_{j-i,h} - \mean{X_{i,h}X_{0,h}}\Big\Vert_{p/2}
+ \Big\Vert (\bar{X_h}(k))^2 \Big\Vert_{p/2} \\
&\qquad+ \Big\Vert \dfrac{\bar{X_h}(k)}{k-i}\sum_{j=i+1}^{k}X_{j-i,h} \Big\Vert_{p/2}
+ \Big\Vert \dfrac{\bar{X_h}(k)}{k-i}\sum_{j=i+1}^{k}X_{j,h}  \Big\Vert_{p/2} ~.
\end{split}
\end{align}
Let us treat the first summand of the right-hand side at first.
We can proceed similar to the proof of Theorem 1 in \citet{wu:2007} and assume without loss of generality that $\mean{X_{i,h}X_{0,h}}=0$.
Due to the discussion above, that leads to
\begin{align}\label{eq:ufac2}
S_{k,h} := \sum_{\ell=i+1}^{k} X_{\ell,h}X_{\ell-i,h} = \sum_{\ell=i+1}^{k} \sum_{j=0}^{\infty} \Pro_{\ell-j}X_{\ell,h}X_{\ell-i,h} =  \sum_{j=0}^{\infty}\sum_{\ell=i+1}^{k} \Pro_{\ell-j}X_{\ell,h}X_{\ell-i,h}~.
\end{align}
By Burkholder and triangle inequality we can bound the inner sum
\begin{align}\label{eq:ufac3}
\Big\Vert \sum_{\ell=i+1}^{k} \Pro_{\ell-j}X_{\ell,h}X_{\ell-i,h} \Big\Vert_{p/2}
\leq \sqrt{n} \cdot C_p \Vert \Pro_{0}X_{j,h}X_{j-i,h} \Vert_{p/2}~.
\end{align}
Applying the ideas of the proof of Lemma 1 in \cite{wu:2009}, we have
\begin{align}\label{eq:ufac4}
\Vert \Pro_{0}X_{j,h}X_{j-i,h} \Vert_{p/2}
\lesssim \delta^{j-i}I_{\{j-i\geq 0\}}+\delta^jI_{\{j\geq 0\}}
\end{align}
uniformly with respect to  $h$, where we also used Assumption (T1).
Combining the representation derived in \eqref{eq:ufac2} with \eqref{eq:ufac3} and \eqref{eq:ufac4}, we obtain
\begin{align*}
\| S_{k,h} \|_{p/2} 
\leq \sum_{j=0}^{\infty} \Big\Vert \sum_{\ell=i+1}^{k} \Pro_{\ell-j}X_{\ell,h}X_{\ell-i,h} \Big\Vert_{p/2}
\leq \sqrt{n} C_p \sum_{j=0}^{\infty}  \| \Pro_{0}X_{j,h}X_{j-i,h} \|_{p/2}\\
\leq \sqrt{n} C_p \sum_{j=0}^{\infty} (\delta^{j-i}I_{\{j-i\geq 0\}}+\delta^jI_{\{j\geq 0\}})
\lesssim \sqrt{n}
\end{align*} 
uniformly with respect to $h$. The treatment of the first summand of \eqref{eq:ufac1} is now  finished by using the bound
\begin{align*}
\max_{i=1}^{\beta_n}\maxhd \max_{k=c\cdot n}^n \dfrac{\sqrt{n}}{k - i } \big\Vert S_{k,h} \big\Vert_{p/2}
&\leq \maxhd\max_{k=c\cdot n}^n \dfrac{\sqrt{n}}{cn - n^B } \big\Vert S_{k,h} \big\Vert_{p/2}\\
&\lesssim \dfrac{n}{cn -n^B}=\mathcal{O}(1)~.
\end{align*}
It remains to consider the last three summands of \eqref{eq:ufac1}.
Since the arguments are similar we will only consider  the second one.
Again we have the representation
\begin{align*}
\bar{X_h}(k)
= \dfrac{1}{k}\sum_{\ell=1}^{k} \sum_{j=0}^{\infty}\Pro_{\ell-j}X_{\ell,h}
= \dfrac{1}{k} \sum_{j=0}^{\infty}\sum_{\ell=1}^{k}\Pro_{\ell-j}X_{\ell,h}
\end{align*}
and can apply Burkholder inequality
\begin{align*}
\Big\Vert \sum_{\ell=1}^{k}\Pro_{\ell-j}X_{\ell,h} \Big\Vert_p
\leq C_p \sqrt{n} \|\Pro_{0}X_{j,h} \|_{p}~.
\end{align*}
Due to Assumption (T1) and $\| \Pro_j(X_{0,h}) \|_p \leq \vartheta_{j,h,p}~$, it follows that
\begin{align*}
\maxhd \max_{k=cn}^n \| \bar{X}_h(k) \|_p
\leq \maxhd \max_{k=cn}^n \dfrac{1}{k} \sum_{j=0}^{\infty} \Big\Vert \sum_{\ell=1}^{k}\Pro_{\ell-j}X_{\ell,h} \Big\Vert_p\\
\leq  \maxhd \dfrac{1}{c\sqrt{n}}C_p \sum_{j=0}^{\infty}\|\Pro_{0}X_{j,h} \|_{p}
\lesssim \dfrac{1}{\sqrt{n}}~,
\end{align*}
Finally Cauchy-Schwarz inequality gives that
\begin{align*}
\sqrt{n}\maxhd \max_{k=cn}^n \| \bar{X}_h^2(k) \|_{p/2}
= \mathcal{O}(1)~,
\end{align*}
which completes the proof.
\end{proof}
\medskip

Based on the autocovariance estimators $\hat{\phi}^*_h(i,k)$ defined in \eqref{eq:autocovarstar} let
\begin{align*}
( \sigma_h^*(k) )^2
= \sum_{ |i| \leq \beta_n} \hat{\phi}^*_h(i,k)~,
\end{align*}
where the bandwidth is $\beta_n = n^B$ for $B$ in Assumption (S2).
The next Lemma is a uniform version of Lemma E.6 of \cite{jirak:2015a}. 
\begin{lemma}\label{lem:sigmastar}
Grant Assumptions \ref{a1} - \ref{a3} and let $c \in (0,1)$ be a fixed constant.
If $n$ is large enough, there exist $\eta>0$ , sufficiently small, such that
\begin{align*}
\Pb{ \maxhd \max_{k=c \cdot n}^n \vert (\hatsigmah^*(k))^2 - \sigmah^2 \vert \geq n^{-\eta} } \lesssim n^{-C}
\end{align*}
for a constant $C>0$.
\end{lemma}
\begin{proof}
We have the following decomposition
\begin{align}
\begin{split}\label{bandwidthdecomp}
\Prob \big( \maxhd \max_{k=c \cdot n}^n  \vert \left(\hatsigmah^{\ast}(k)\right)^2 - &\sigmah^2 \vert \geq n^{-\eta} \big) 
\leq \mathbb{P} \Big(  \maxhd \sum_{i=\beta_n+1}^{\infty}\vert \phi_h(i) \vert \geq \frac{n^{-\eta}}{2} \Big) \\
&+  \mathbb{P} \Big(  \maxhd \max_{k=c\cdot n}^n  \sum_{i=0}^{\beta_n} \vert \hat{\phi}_h(i,k) - \phi_h(i) \vert \geq \dfrac{n^{-\eta}}{2} \Big ) ~.
\end{split}
\end{align}
An adaption of Lemma E.4 from \cite{jirak:2015} yields
\begin{align*}
\sum_{i=\beta_n+1}^{\infty}\vert \phi_h(i) \vert \lesssim \delta^{\beta_n}~,
\end{align*}
uniformly with respect to $h$ and so for sufficiently large $n$ we obtain
\begin{align*}
 \mathbb{P} \Big( \maxhd \sum_{i=\beta_n+1}^{\infty}\vert \phi_h(i) \vert \geq \frac{n^{-\eta}}{2} \Big ) =0~.
\end{align*}
The second summand of the right-hand side of \eqref{bandwidthdecomp} can be bounded by Markov's inequality. i.e.
\begin{align*}
& \mathbb{P} \Big( \maxhd \max_{k=c \cdot n}^n  \sum_{i=0}^{\beta_n} \vert \hat{\phi}_h(i,k) - \phi_h(i) \vert \geq \dfrac{n^{-\eta}}{2} 
\Big ) \\
&\leq \sum_{h=1}^{d}\sum_{i=0}^{\beta_n}\sum_{k=c \cdot n}^n 
\Pb{ \vert \hat{\phi}_h(i,k) - \phi_h(i) \vert \geq n^{-\eta}/(2\beta_n+2) }\\
&\lesssim (n^{-\eta})^{-p/2}\sum_{h=1}^{d} \sum_{i=0}^{\beta_n} \sum_{k=c \cdot n}^n(\beta_n)^{p/2} \|\hat{\phi}_h(i,k) -\phi_h(i)\|_{p/2}^{p/2}~.
\end{align*}
Applying Lemma \ref{lem:uniformautocov} and using $d \lesssim n^D$, $\beta_n\sim n^B$ the last expression is bounded by
\begin{align*}
n^{\eta p/2}d \beta_n n \beta_n^{p/2}n^{-p/4} \lesssim n^{\eta p/2 +D +1 +B(p/2+1)-p/4}~.
\end{align*}
Due to Assumption (S2), we have $D+1+B(p/2+1)-p/4<0$. 
Therefore choosing $\eta>0$ sufficiently small, such that
\begin{align*}
0 < \eta < 2\Big(p/4-B(p/2+1)-1-D\Big)/p
\end{align*}
yields the claim.
\end{proof}
\noindent We can now proceed to the actual proof of Lemma \ref{lem:hatvar}
\begin{proof}
Denote with $\mathcal{S}$ the set
\begin{align*}
\mathcal{S} = 
\Big \{ S < \minhd \min\Big( \dfrac{1-t_h}{1-\hatth},~\dfrac{t_h}{\hatth} \Big ) \Big \}~,
\end{align*}
where $S$ refers to the constant used in the data separation in equation \eqref{dh12}~.
An application of Lemma \ref{lem:hatth} yields
\begin{align*}
\Pb{\mathcal{S^C}} 
&\leq \mathbb{P} \Big(  S \geq \minhd \dfrac{(1-t_h)}{1 - \hatth} \Big )  +  \mathbb{P} \Big( S \geq \minhd \dfrac{t_h}{\hatth}  \Big ) \\
&=  \mathbb{P} \Big( 1 - \minhd \dfrac{(1-t_h)}{1 - \hatth} \geq 1 - S  \Big )  +  \mathbb{P} \Big( 1 - \minhd \dfrac{t_h}{\hatth} \geq 1 - S \Big ) \\
&=  \mathbb{P} \Big(  \maxhd 1 - \dfrac{(1-t_h)}{1 - \hatth} \geq 1 - S \Big )  +  
 \mathbb{P} \Big( \maxhd 1 - \dfrac{t_h}{\hatth} \geq 1 - S \Big ) \\
&\leq  \mathbb{P} \Big(  \maxhd |\hatth - t_h| \geq (1 - S)\bt  \Big )
+  \mathbb{P} \Big( \maxhd |\hatth - t_h| \geq (1 - S)\bt  \Big )
= o(1)~.
\end{align*}
Using the upper bound
\begin{align*}
\vert \hatsigmah - \sigmah \vert
\leq \dfrac{\vert \hatsigmah^2 - \sigmah^2 \vert}{\hatsigmah + \sigmah}
\leq \dfrac{\vert \hatsigmah^2 - \sigmah^2 \vert}{\sigma_-}
\end{align*}
we have the following decomposition
\begin{align}\label{eq:hatvar12}
\begin{split}
\mathbb{P} \Big(  \maxhd \left| \hatsigmah - \sigmah \right| \geq n^{-\eta} \Big ) 
&\leq \mathbb{P} \Big(  \maxhd \left| \hatsigmah^2 - \sigmah^2 \right| \geq n^{-\eta}\sigma_-  \Big ) \\
\leq \mathbb{P} \Big( \maxhd \left| \hatsigma_{h,1}^2 - \sigmah^2 \right| \geq n^{-\eta}\sigma_-  \Big )
&+  \mathbb{P} \Big(  \maxhd \left| \hatsigma_{h,2}^2 - \sigmah^2 \right| \geq n^{-\eta}\sigma_- \Big ) ~,
\end{split}
\end{align}
where we used that $\hatsigmah^2  = \tfrac{1}{2}\big(\hatsigma_{h,1}^2 + \hatsigma_{h,2}^2 \big)$.
Since the arguments are the same, we will only treat the first summand of the second line of 
\eqref{eq:hatvar12}.
We can conclude that
\begin{align}\label{Sbenutzt}
\mathbb{P} \Big(  \maxhd \left| \hatsigma_{h,1}^2 - \sigmah^2 \right| \geq n^{-\eta}\sigma_-  \Big ) 
\leq  \mathbb{P} \Big( \maxhd \left| \hatsigma_{h,1}^2 - \sigmah^2 \right| \geq n^{-\eta}\sigma_- \cap \mathcal{S} \Big ) + o(1)~.
\end{align}
On the set $\mathcal{S}$ the inclusion $\widehat{\mathcal{D}}_{h,1} \subset \mathcal{D}_{h,1}$ holds and since there are now structural breaks within  
$\mathcal{D}_{h,1}$, the following identity holds (on the set $\mathcal{S}$) due to definition \eqref{eq:model}
\begin{align*}
\hatsigma_{h,1}
=\hatsigmah^{\ast}\left(n\max\{S \hatth , \bt \}\right)~.
\end{align*}
Further note that we have by definition
\begin{align*}
\bt n \leq \minhd n\max\{S \hatth , \bt \} = \minhd \vert \widehat{\mathcal{D}}_{h,1} \vert \leq n~.
\end{align*}
Finally, we obtain
\begin{align*}
\Prob \Big ( \maxhd \left| \hatsigma_{h,1}^2 - \sigmah^2 \right| &\geq n^{-\eta}\sigma_- \cap \mathcal{S} \Big)\\
&= \mathbb{P} \Big( \maxhd \left| \hatsigmah^{\ast}\left(n\max\{S \hatth , \bt \}\right)^2 - \sigmah^2 \right| \geq n^{-\eta}\sigma_- \cap \mathcal{S}
\Big ) \\
&\leq  \mathbb{P} \Big( \maxhd \left| \hatsigmah^{\ast}\left(n\max\{S \hatth , \bt \}\right)^2 - \sigmah^2 \right| \geq n^{-\eta}\sigma_- \Big ) \\
&\leq \mathbb{P} \Big( \maxhd \max_{k=\bt n}^n \vert \sigma_h^{\ast}(k)^2 - \sigmah^2 \vert \geq n^{-\eta}\sigma_- \Big )~.
\end{align*}
Employing Lemma \ref{lem:sigmastar} completes the proof.
\end{proof}

\subsubsection{Proof of Theorem \ref{thm:gumbel}} \label{sec713}
As $\absdmuh > C_u$ for all $h\in \setd $ we have $\Sd = \emptyset$ and $\Sd^c = \setd$, and Lemma \ref{lem:hatth} implies 
\begin{align}\label{eq:3hatth0}
\maxhd \vert \hatth - t_h \vert = o_p(n^{-1/2})
\end{align}
Observing that $\hatth \in [ \bt, 1- \bt ]$ it follows that 
\begin{align}\label{lem:3hatth(i)}
\max\limits_{h=1}^{d}\Big \vert \frac{t_h^2(1-t_h)^2}{\hatth^2(1-\hatth)^2} - 1 \Big  \vert
= o_p(n^{-1/2})
\end{align}
Moreover, one easily verifies that the function $t \to \tau(t)$ defined in \eqref{tauh} is Lipschitz-continuous on the interval $[\bt, 1-\bt]$ and 
therefore  we obtain from \eqref{eq:3hatth0} that
\begin{align}\label{lem:3hatth(ii)}
\max\limits_{h=1}^{d} \vert \tau(t_h) - \tau(\hatth) \vert = o_{\mathbb{P}}(n^{-1/2})~.
\end{align}
Finally, we note that for a sufficiently small constant $C>0$ the estimate
\begin{align}\label{lem:3hatth(iii)}
\max\limits_{h=1}^{d} \vert \hatsigmah \tau(\hatth) - \sigma_h\tau(\hatth) \vert = o_{\mathbb{P}}(n^{-C})
\end{align} 
holds, which is a direct consequence Lemma \ref{lem:hatvar} and assertion \eqref{lem:3hatth(ii)}. 
\\

\noindent
After these preparations we return to the proof of Theorem \ref{thm:gumbel}.
We recall the definition \eqref{eq:Mint} and introduce the notation 
\begin{align*}
\hatTnhb = \frac{\sqrt{n}}{\tau(t_h)\sigma_h\absdmuh} \left( \hatMint - \dmuh^2 \right)
\end{align*}
We will first show the weak convergence 
\begin{align}\label{stat312a}
a_d \big( \maxhd \hatTnhb - b_d \big ) \convd G
\end{align}
For a proof of \eqref{stat312a} we recall the definition of the statistic $ \Tnh$ in  \eqref{t1} and obtain from Theorem \ref{cor:tnhgumbel}
\begin{align*}
a_d \big ( \maxhd \Tnh - b_d \big ) \convd G~.
\end{align*}
With the representation for $\hatMint$ and $\Mint$ in \eqref{eq:Mint} and \eqref{hmat}, respectively, and the notation $\hatqh = (t_h(1-t_h))^2/(\hatth(1-\hatth))^2$ it now follows that 
\begin{align*}
&a_d \maxhd \big ( \hatTnhb - b_d \big ) \\
&= a_d \maxhd \Big ( \hatqh \frac{\sqrt{n}}{\tau(t_h)\sigma_h\absdmuh} \left( \Mint - \dmuh^2 \right) - b_d +  \frac{\sqrt{n} (\hatqh - 1) \dmuh^2}{\tau(t_h)\sigma_h\absdmuh} \Big ) =  B_{n,d}\notag 
\end{align*}
where 
\begin{align*}
B_{n,d} := a_d \maxhd \Big ( \hatqh \Tnh - b_d + (\hatqh-1) \dmuh^2 \frac{\sqrt{n}}{\tau(t_h)\sigma_h\absdmuh} \Big )~.
\end{align*}
It is easy to see that this term can be bounded by   
\begin{align*}
a_d \big ( \maxhd \hat{q}_h \Tnh - b_d \big ) - 
R_{n,d}  &\leq  B_{n,d} 
\leq a_d \big ( \maxhd \hat{q}_h \Tnh - b_d \big ) + R_{n,d}~,
\end{align*}
where the remainder satisfies 
\begin{align*}
R_{n,d}  = a_d \maxhd \left| \hat{q}_h - 1 \right| \maxhd \frac{\sqrt{n}}{\tau(t_h)\sigma_h\absdmuh} \dmuh^2
\convp 0~, 
\end{align*}
which follows observing  the inequalities  \eqref{lem:auxiliary1}, \eqref{lem:3hatth(i)} and the condition  $C_{\ell} \leq \absdmuh \leq C_{u}$.
Thus \eqref{stat312a} follows, if we can establish 
\begin{align}\label{conv:qTnh}
a_d \big ( \maxhd \hat{q}_h \Tnh - b_d \big ) \convd G~.
\end{align}
For a proof of this result we fix $x \in \mathbb{R}$ and define $u_d(x)= x/a_d + b_d$. By
\begin{align}\label{ersterschaetzers1}
\begin{split}
\Prob \Big( { 0 \leq \maxhd \hatqh \maxhd \Tnh \leq u_d(x)} \Big)
&\leq \Prob \Big({ 0 \leq \maxhd \hatqh \Tnh \leq u_d(x)} \Big)\\
&\leq \Prob \Big( { 0 \leq \minhd \hatqh \maxhd \Tnh \leq u_d(x)} \Big)~.
\end{split}
\end{align}
and
\begin{align*}
\Prob \Big({ \minhd \hatqh \maxhd \Tnh \geq 0}\Big) = 
\Prob \Big({ \maxhd \hatqh \Tnh \geq 0} \Big)=
\Prob \Big({ \maxhd \hatqh \maxhd \Tnh \geq 0}\Big)~,
\end{align*}
we obtain
\begin{align}
\begin{split}
\label{ersterschaetzers2}
\Prob \Big({ \maxhd \hatqh \maxhd \Tnh \leq u_d(x)}\Big)
&\leq \Prob \Big({ \maxhd \hatqh \Tnh \leq u_d(x)}\Big)\\
&\leq \Prob \Big({ \minhd \hatqh \maxhd \Tnh \leq u_d(x)}\Big)~.
\end{split}
\end{align}
From \eqref{lem:3hatth(i)} and $d=C_1n^D$ it follows that 
\begin{align*}
a_db_d  \big  ( \minhd \hatqh -1 \big  ) \convp 0
\;\;\; 
\text{and}
\;\;\;
a_db_d \big ( \maxhd \hatqh -1 \big ) \convp  0~, 
\end{align*}
which due to Slutsky's theorem directly implies
\begin{align*}
&\Prob \Big({ \minhd \hatqh \maxhd \Tnh \leq u_d(x)} \Big) \convn e^{-e^{-x}}~,\\
&\Prob \Big( { \maxhd \hatqh \maxhd \Tnh \leq u_d(x)}\Big) \convn e^{-e^{-x}}~.
\end{align*}
Thus we have established \eqref{conv:qTnh} and proved \eqref{stat312a}.
\\\\
To complete the proof of Theorem \ref{thm:gumbel} note that the assertion \eqref{stat312} is equivalent to
\begin{align}\label{stat312b}
\Prob \Big({ \maxhd \hatTnh \leq u_d(x)} \Big) \convn e^{-e^{-x}}~,
\end{align}
where we use again $u_d(x) = x/a_d + b_d$.
To prove this statement define
\begin{align*}
Q^{-}_{d} := \minhd \frac{\hatsigmah\tauhatth}{\sigmah \tauth}~,~
Q^{+}_{d} := \maxhd \frac{\hatsigmah\tauhatth}{\sigmah \tauth}~,
\end{align*}
and consider the set
$
\mathcal{Q}_d := \left\lbrace \left| Q^{+}_{d} - 1 \right| \vee \left| Q^{-}_{d} -1 \right| \leq \delta_d \right\rbrace 
$, 
where the involved sequence is given by $\delta_d = (\log d)^{-2}$.
The inequalities \eqref{lem:auxiliary1} and the estimate \eqref{lem:3hatth(iii)} yield
\begin{align*}
\Pb{ \left| Q^{+}_{d} - 1 \right| > \delta_d } 
&= \Prob \Big({ \Big | \maxhd \frac{\hatsigmah \tauhatth }{\sigmah \tauth  } - 1 \Big | > \delta_d }   \Big)\\ 
&\leq \Prob \Big({ \Big | \maxhd \hatsigmah \tauhatth -  \sigmah \tauth  \Big | > \delta_d c } \Big) = o(1)~.
\end{align*}
By a similar argument for the term  $\left| Q^{-}_{d} -1 \right|$ we obtain $\Pb{\mathcal{Q}_d^c}\to 0$. If $\maxhd \hatTnh \geq 0$ holds, we can conclude that 
\begin{align*}
\frac{1}{Q^+_{d}}\maxhd \hatTnhb 
&= \frac{1}{Q^+_{d}} \maxhd \frac{\sqrt{n}}{\tau(t_h)\sigma_h\absdmuh} \left( \hatMint - \dmuh^2 \right) \\
&\leq  \maxhd \frac{\sqrt{n}}{\tau(\hatth)\hatsigmah\absdmuh} \left(\hatMint - \dmuh^2 \right)
= \hatTnh
= \frac{1}{Q^{-}_d} \maxhd \hatTnhb~.
\end{align*}
Therefore the following inequalities hold
\begin{align}
\label{eq:withzero}
\begin{split}
\Prob \Big({ 0 \leq \maxhd \hatTnhb \leq u_d(x)Q^{-}_{d} } \Big) 
&\leq \Prob \Big({ 0 \leq \maxhd \hatTnh \leq u_d(x) } \Big)\\
&\leq \Prob \Big({ 0 \leq \maxhd \hatTnhb \leq u_d(x)Q^{+}_{d}}\Big)~.
\end{split}
\end{align}
Observing the identity
\begin{align*}
\Prob \Big( { \maxhd \hatTnhb \geq 0}\Big ) = \Prob \Big({\exists h: \hatMint \geq \dmuh^2 } \Big)  = \Prob \Big({\maxhd \hatTnh \geq 0}  \Big) 
\end{align*}
we can derive from \eqref{eq:withzero}
\begin{align}\label{eq:withoutzero}
\begin{split}
\Prob \Big({ \maxhd \hatTnhb \leq u_d(x)Q^{-}_{d} }\Big) 
&\leq \Prob \Big({ \maxhd \hatTnh \leq u_d(x) } \Big)\\ 
&\leq \Prob \Big({ \maxhd \hatTnhb \leq u_d(x)Q^{+}_{d}} \Big)~. 
\end{split}
\end{align}
Hence, we directly obtain 
\begin{align*}
\Prob \Big({ \maxhd \hatTnhb \leq u_d(x)Q^{+}_{d}}\Big)  
&\leq \Prob({\mathcal{Q}_d^c}) + \Prob \Big({ \maxhd \hatTnhb \leq u_d(x)Q^{+}_{d} \cap \mathcal{Q}_d}\Big) \\
&\leq o(1) + \Prob \Big({ \maxhd \hatTnhb \leq u_d(x) + u_d(x)\delta_d}\Big)~.
\end{align*}
Now we fix $\varepsilon > 0$ and note that the inequality $u_d(x - \varepsilon) < u_d(x) + u_d(x)\delta_d < u_d(x+\varepsilon)$ holds if $n$ (or equivalently $d$) is sufficiently large.
The weak convergence \eqref{stat312a} then yields
\begin{align*}
\limsup_{d,n \to \infty}\; &\Prob \Big({ \maxhd \hatTnhb \leq u_d(x)\mathcal{Q}_{d}^+ } \Big)\\
&\leq \limsup_{d,n \to \infty} \Prob \Big({ \maxhd \hatTnhb \leq u_d(x + \varepsilon)}\Big) = e^{-e^{-(x+\varepsilon})}~.
\end{align*}
Using Bonferroni's inequality we can proceed similarly for the lower bound of \eqref{eq:withoutzero}, i.e. 
\begin{align*}
\liminf_{d,n \to \infty}  \Prob \Big( \maxhd \hatTnhb \leq &u_d(x)Q_{d}^- \Big) 
\geq \liminf_{d,n \to \infty} \Prob \Big({ \maxhd \hatTnhb \leq u_d(x)Q_{d}^- \cap \mathcal{Q}_d} \Big) \\
&\geq \liminf_{d,n \to \infty} \Prob \Big({ \maxhd \hatTnhb \leq u_d(x) - u_d(x)\delta_d} \Big) - \Pb{\mathcal{Q}_d^c}\\
&\geq \liminf_{d,n \to \infty} \Prob \Big({ \maxhd \hatTnhb \leq u_d(x-\varepsilon)}\Big) 
= e^{-e^{-(x-\varepsilon})}~.
\end{align*}
The assertion \eqref{stat312b} then follows by $\varepsilon \to 0$, which completes the proof of Theorem \ref{thm:gumbel}. 

\subsubsection{Proof of Corollary \ref{cor:adjustedgumbel}}
At first we consider the case where $m_d:=\vert{\cal M}_d \vert = c \cdot d + o(d) $ for some constant $c\in (0,1]$
and note that in this case 
\begin{align*}
a_d \Big ( \max_{h\in{\cal M}_d} \hatTnh - b_d \Big )
= \frac{a_d}{a_{m_d}}\bigg( a_{m_d}\max_{h \in{\cal M}_d} \hatTnh -a_{m_d}b_{m_d} \bigg) + a_db_{m_d}-a_db_d~.
\end{align*}
Theorem \ref{thm:gumbel} yields
$ a_{m_d}\max_{h \in{\cal M}_d} \hatTnh -a_{m_d}b_{m_d} \convd G $ and furthermore we have 
\begin{align*}
\frac{a_d}{a_{m_d}}
\underset{d \to \infty}{\longrightarrow} 1
\;\;\;
\text{and}
\;\;\;
a_{d}b_{m_d} - a_db_d 
\underset{d \to \infty}{\longrightarrow} \log c~.
\end{align*}
A short calculation therefore leads to
$ a_d \big ( \max_{h\in{\cal M}_d} \hatTnh - b_d \big ) \convd G + \log c$~.
The case $m_d = o(d) $ can be treated similarly. 
Finally, statement \eqref{eq:smallerGumbel} is a consequence of the inequality
\begin{align*}
\max_{h \in{\cal M}_d} \hatTnh \leq \maxhd \hatTnh~.
\end{align*}

\section{Proofs of the results in Section 4}\label{proofsec4}

\subsection{Proof of Theorem \ref{thm:level}} \label{sec721}
By Assumption (C1) and the definition of $\hatth$ in \eqref{eq:hatth}, there exists a global constant $C(\bt) > 1$ such that
\begin{align}\label{eq:Cbt}
\dfrac{(t_h(1-t_h))^2}{(\hatth(1-\hatth))^2} \leq C(\bt)~.
\end{align}
Recall the definition of the set $\Sd$ in \eqref{sd}, choose a constant $\Delta_-^2>\zeta>0$ and consider the following decomposition of the set $\setd \setminus \Sd$
\begin{align}\label{eq:thmleveldecomp}
\begin{split}
\Id &:= \{ h \in \setd \;  | \;  (\Delta_h-\zeta)/\sqrt{C(\bt)} \geq \vert \dmuh \vert > 0 \}~,\\
\Ed &:= \{ h \in \setd \;  | \;  \Delta_h \geq \vert \dmuh \vert > (\Delta_h -\zeta )/\sqrt{C(\bt}) \}~.
\end{split}
\end{align}
Using the representation
\begin{align*}
\TS =
\max \Big \lbrace a_d  \Big ( \max_{h\in \Sd} \hatTnhdelta - b_d  \Big ), a_d  \Big ( \max_{h\in \Id } \hatTnhdelta - b_d  \Big ), a_d  \Big ( \max_{h\in \Ed} \hatTnhdelta - b_d  \Big ) \Big \rbrace~.
\end{align*}
the first assertion \eqref{limsup} follows from the following three statements
\begin{align}
&a_d  \Big ( \max_{h\in \Sd} \hatTnhdelta - b_d  \Big  ) \convp - \infty,\label{eq:s1niveau}\\
&a_d  \Big  ( \max_{h\in \Id} \hatTnhdelta - b_d  \Big  ) \convp - \infty,\label{eq:s2niveau}\\
\limdn &\Prob  \Big ({ a_d \ \Big ( \max_{h\in \Ed} \hatTnhdelta - b_d \ \Big ) \geq \gua } 
 \Big ) \leq \alpha.\label{eq:s3niveau}
\end{align}
\textbf{Proof of (\ref{eq:s1niveau}):} 
Observing the definition of $ \hatTnhdelta$ in \eqref{eq:tnhdelta} we obtain the inequality 
\begin{align}\label{eq:s11niveau}
\begin{split}
a_d \Big  ( \max_{h \in \Sd}  &\hatTnhdelta - b_d  \Big  ) \\
&\leq a_d \max_{h \in \Sd} \frac{\sqrt{n}}{\tau(\hatth)\hatsigmah\Delta_h} \hatMint - a_d\min_{h \in \Sd} \frac{\sqrt{n}}{\tau(\hatth)\hatsigmah\Delta_h} \Delta_h^2 - a_db_d~.
\end{split}
\end{align}
The first summand of this expression is further bounded by
\begin{align}\label{step12niveau}
a_d \max_{h \in \Sd } \frac{\sqrt{n}}{\tau(\hatth)\hatsigmah\Delta_h} \hatMint 
\lesssim a_d \max_{h \in \Sd } \sqrt{n} \cdot \Mint~,
\end{align}
and arguing as in the proof of Lemma \ref{lem:tnh1} yields $a_d \max_{h \in \Sd } \sqrt{n} \cdot \Mint \convp 0~.$
For the second summand of \eqref{eq:s11niveau} we can use $\frac{\sqrt{n}}{\tau(\hatth)\hatsigmah\Delta_h} \Delta_h^2 > 0 $ and $a_db_d \sim \log d$ to obtain
\begin{align*}
a_d\min_{h \in \Sd }\frac{\sqrt{n}}{\tau(\hatth)\hatsigmah\Delta_h}  \Delta_h^2 + a_db_d \convdn \infty,
\end{align*}
which yields \eqref{eq:s1niveau}.\\

\noindent
\textbf{Proof of (\ref{eq:s2niveau}):} By definition of the set $\Id$, we get $\Delta_h^2 \geq C(\bt)\dmuh^2 + \zeta $,
which leads to
\begin{align}\label{eq:s21niveau}
\begin{split}
a_d \Big ( \max_{h \in \Id } \hatTnhdelta - b_d \Big )
\leq  a_d  &\Big( \max_{h \in \Id} \frac{\sqrt{n}}{\tau(\hatth)\hatsigmah\Delta_h} \Big  ( \hatMint - C(\bt)\dmuh^2 \Big  ) - b_d \Big  )\\ 
&- a_d\min_{h \in \Id}\frac{\sqrt{n}}{\tau(\hatth)\hatsigmah\Delta_h} \zeta~.
\end{split}
\end{align}
For the second summand of the last expression it holds that
\begin{align}\label{eq:s22niveau}
a_d\sqrt{n} \min_{h \in \Id} \frac{\zeta}{\tau(\hatth)\hatsigmah\Delta_h}
\gtrsim \frac{a_d\sqrt{n}}{\tau_+s_+\Delta_+}
\convn \infty~,
\end{align}
From inequality \eqref{eq:Cbt}, we obtain $\hatMint \leq C(\bt)\Mint$, which gives the following bound for the first summand of \eqref{eq:s21niveau}
\begin{align}\label{eq:s23niveau}
\begin{split}
a_d \Big ( \max_{h \in \Id} \frac{\sqrt{n}}{\tau(\hatth)\hatsigmah\Delta_h}&\Big  ( \hatMint - C(\bt)\dmuh^2 \Big ) - b_d \Big  )\\
&\lesssim a_d \Big  ( \max_{h \in \Id} \frac{\sqrt{n}}{\tau(\hatth)\hatsigmah\Delta_h} \Big  ( \Mint - \dmuh^2 \Big ) - b_d \Big  )~.
\end{split}
\end{align}
Similar to \eqref{eq:tnhexpansion} we can use the following decomposition
\begin{align*}
\frac{\sqrt{n}}{\tau(\hatth)\hatsigmah\Delta_h}\left( \Mint - \dmuh^2 \right)= S_{n,h}^{(1)} + S_{n,h}^{(2)}~,
\end{align*}
where the quantities $S_{n,h}^{(1)}$ and $  S_{n,h}^{(2)}$ are given by 
\begin{align*}
S_{n,h}^{(1)} &= \frac{3\sqrt{n}}{t_h^2(1-t_h)^2\tau(\hatth)\hatsigmah\Delta_h}\int_0^1 \left(\mathbb{U}_{n,h}(s) - \mu_h(s,t_h) \right)^2 ds,\\
S_{n,h}^{(2)} &= \frac{6\sqrt{n}}{t_h^2(1-t_h)^2\tau(\hatth)\hatsigmah\Delta_h}\int_0^1 \mu_h(s,t_h)(\mathbb{U}_{n,h}(s) - \mu_h(s,t_h))ds~,
\end{align*}
respectively.
Now we have an upper bound for \eqref{eq:s23niveau} given by
\begin{align*}
a_d \max_{h \in \Id} S_{n,h}^{(1)} + a_d \Big ( \max_{h \in \Id} S_{n,h}^{(2)} - b_d \Big)~. 
\end{align*}
Similar as in the proof of (\ref{eq:s1niveau}) one easily shows that $a_d \max\limits_{h \in \Id } S_{n,h}^{(1)} = o_{\mathbb{P}}(1)$.
In the case that $S_{n,h}^{(2)} \geq 0$ we have
\begin{align*}
S_{n,h}^{(2)}
\lesssim  S_{n,h}^{(3)}:= 
\frac{6\sqrt{n}}{t_h^2(1-t_h)^2\tau(t_h)\sigmah\absdmuh}\int_0^1 \mu_h(s,t_h)(\mathbb{U}_{n,h}(s) - \mu_h(s,t_h))ds~,
\end{align*}
which gives 
\begin{align*}
a_d \Big ( \max_{h \in \Id} S_{n,h}^{(2)} - b_d \Big  )
&\lesssim a_d \Big ( \max_{h \in \Id } \max \left\lbrace S_{n,h}^{(3)},0 \right\rbrace - b_d \Big )\\
&\leq \max \Big  \lbrace a_d \Big  ( \max_{h \in \Id} S_{n,h}^{(3)} - b_d\Big  ) , - a_db_d \Big  \rbrace.
\end{align*}
Applying Lemma \ref{lem:gaussappro} yields
\begin{align*}
\limsup_{d,n \to \infty} \Prob \Big ( { a_d \Big ( \max_{h \in \Id} S_{n,h}^{(3)} - b_d \Big ) > x } \Big )  \leq \Pb{ G > x}
\end{align*}
for all $x \in \mathbb{R}$ and  consequently the right hand side of   \eqref{eq:s23niveau} is of order $O_{\mathbb{P}}(1)$.
Now (\ref{eq:s2niveau}) follows from \eqref{eq:s21niveau} and  \eqref{eq:s22niveau}.\\

\noindent
\textbf{Proof of (\ref{eq:s3niveau}):} Observing that $d_h := \Delta_h^2 - \dmuh^2 \geq 0$ we obtain 
\begin{align}\label{eq:s31niveau}
a_d \Big ( \max_{h \in \Ed} \hatTnhdelta - b_d  \Big )
&\leq a_d \Big  ( \max_{h \in \Ed} \frac{\sqrt{n}}{\tau(\hatth)\hatsigmah\Delta_h}\Big  ( \hatMint - \dmuh^2 \Big  ) - b_d\Big  )~, 
\end{align}
As $\alpha \in \left( 0, 1 - e^{-1} \right]$ the  quantile of the Gumbel distribution satisfies 
$ \gua = -\log ( \log ( \tfrac{1}{1-\alpha} ) ) \geq 0~,$
and we can proceed as follows
\begin{align*}
\Prob \Big (  a_d  \Big  ( \max_{h \in \Ed } &\hatTnhdelta - b_d  \Big ) > \gua   \Big )\\
&\leq \Prob \Big (  \max_{h \in \Ed }\frac{\absdmuh} {\Delta_h} \cdot a_d \Big ( \max_{h \in \Ed} \hatTnh - b_d \Big ) > \gua \Big)\\
&\leq \Prob \Big ( {  a_d \Big  ( \max_{h \in \Ed} \hatTnh - b_d \Big  ) > \gua } \Big ).
\end{align*}
An application of Corollary \ref{cor:adjustedgumbel} now yields
\begin{align*}
\limsup_{d,n \to \infty} \Prob \Big ( { a_d \Big ( \max_{h \in \Ed} \hatTnh - b_d \Big  ) > \gua} \Big ) \leq \Pb{ G > \gua} = \alpha~,
\end{align*}
which gives assertion \eqref{eq:s3niveau} and completes the proof of assertion \eqref{limsup}.\\

\noindent It remains to show assertion \eqref{eq:thmlevelcases} under the additional assumption of \eqref{eq:thmlevelextra}. Note that under the latter assumption, we can further decompose the set $\Ed$ into $\Ed = (\Ed \setminus \Bd) \cup \Bd$ and observe that \eqref{eq:thmlevelextra} yields
\begin{align*}
\Ed \setminus \Bd = \{ h \in \setd \;  | \;  \Delta_h - C_{\Delta} \geq \absdmuh > (\Delta_h -\zeta )/\sqrt{C(\bt)} \}.
\end{align*}
Again, we can examine both sets separately.
For $\Ed \setminus \Bd$ we have
\begin{align}\label{EdohneBd}
a_d \left( \max_{h \in \Ed \setminus \Bd} \hatTnhdelta - b_d \right) \leq a_d \left( \max_{h \in \Ed \setminus \Bd} \hatTnh - b_d \right) -a_d\min_{h\in \Ed \setminus \Bd} \dfrac{\sqrt{n}}{\tau(\hatth)\hatsigmah\Delta_h}~.
\end{align}
By definition of $\Ed \setminus \Bd$ we obtain that the second summand on the right-hand side of \eqref{EdohneBd} tends (in probability) to $-\infty$. Due to the lower bound $\min_{h \in \Ed \setminus \Bd} \absdmuh > (\Delta_{-} - \zeta)/\sqrt{C(\bt)}$, which holds uniformly in $d$, we can apply Corollary \ref{cor:adjustedgumbel} to the first summand of the right-hand side of \eqref{EdohneBd}, which then gives
\begin{align*}
a_d \left( \max_{h \in \Ed \setminus \Bd} \hatTnhdelta - b_d \right) \convp -\infty~.
\end{align*}  
On the set $\Bd = \Ed$ we can directly apply Corollary \ref{cor:adjustedgumbel}, so that we obtain \eqref{eq:thmlevelcases}. 

\subsection{Proof of Theorem \ref{thm:consistent}}\label{sec732}
It follows from Theorem 3 of \cite{wu:2005} that
\begin{align}\label{eq:partialsumconv}
\Big \lbrace \frac{1}{\sqrt{n}} \sum_{j=1}^{\floor{ns}} \big(Z_{j,h} - \mean{Z_{j,h}}\big) \Big \rbrace_{s \in [0,1]} 
\convd \left\lbrace 
\sigma_h W_s \right\rbrace_{ s \in [0,1]}~,
\end{align}
where $ \left\lbrace W_s \right\rbrace_{ s \in [0,1]}$ denotes the (standard) Brownian motion on the interval $[0,1]$. 
The definition of $\mu_h(s,t_h)$ in \eqref{mu}, $\mean{\cusums}= \mu_h(s,t_h)(1+o(1))$  (uniformly with respect to $s\in [0,1]$) and the continuous mapping theorem yield
\begin{align}\label{eq:cusumconv}
\left\lbrace \sqrt{n} \big(\cusums -\mu_h(s,t_h)\big) \right\rbrace_{s \in [0,1]} \convd  \left\lbrace \sigma_h B_s \right\rbrace_{s \in [0,1]}~,
\end{align}
where $ \left\lbrace  B_s \right\rbrace_{s \in [0,1]}$ denotes a (standard) Brownian bridge.
Observing \eqref{eq:tnhexpansion} we get
\begin{align*}
\sqrt{n}\left(\Mint - \dmuh^2 \right)
&= \frac{3\sqrt{n}}{(t_h(1-t_h))^2} \int_0^1 (\cusums  - \mu_h(s,t_h))^2 ds \\
&+ \frac{6\sqrt{n}}{(t_h(1-t_h))^2}\int_0^1 \mu_h(s,t_h)(\cusums  - \mu_h(s,t_h))ds~. 
\end{align*}
Statement \eqref{eq:cusumconv} and the continuous mapping theorem imply
\begin{align*}
\sqrt{n}\left(\Mint - \dmuh^2 \right) \convd \frac{6}{(t_h(1-t_h))^2}\int_0^1 \mu_h(s,t_h)\sigma_hB(s)ds~.
\end{align*}
It is well known, that the expression on the right-hand side follows a centered normal distribution and a straightforward calculation shows that its variance is given by $\dmuh^2\tau^2(t_h)\sigma_h^2$. Replacing $\sigma_h$ and $t_h$
by  the estimators $\hatsigmah$ and $\hatth$ 
we obtain from  Lemmas \ref{lem:hatth} and \ref{lem:hatvar} the weak convergence
\begin{align} \label{lem:onedimconv}
\frac{\sqrt{n}}{\tau(\hatth)\hatsigmah} \Big ( \hatMint - \dmuh^2 \Big ) \convd \mathcal{N}(0, \dmuh^2)~,
\end{align}
for each (fixed) $h \in \mathbb{N}$ provided that $\absdmuh >0$.
\noindent 
After these preparations we are ready to prove the consistency of the test \eqref{asymtest}. 
If the alternative hypothesis $\HA$ is valid, we can fix $k \in \setd$, such that $d_{k}:= \Delta \mu_{k}^2 - \Delta_{k}^2 >0$. 
From the definition of the test statistic $\TS$ in \eqref{eq:TS} we obtain 
\begin{align*}
\TS 
&\geq a_d \left( \hat{T}_{n,k}^{(\Delta)} - b_d \right) 
= a_d \Big ( \frac{\sqrt{n}}{\tau(\hat{t}_k) \sigma_k \Delta_k}\left( \hat{\mathbb{M}}^2_k - \Delta \mu_k^2 \right) 
+ \frac{\sqrt{n}}{\tau(\hat{t}_k) \sigma_k \Delta_k}d_k - b_d \Big )~,
\end{align*}
which gives 
\begin{align}\label{eq:consistent1}
\Pb{ \TS > \gua } 
&\geq 
\Prob \Big( { \frac{\sqrt{n}}{\tau(\hat{t}_k) \sigma_k \Delta_k} \left( \hat{\mathbb{M}}^2_{k} - \Delta \mu_{k}^2 \right) > \frac{\gua}{a_d} + b_d       
- \frac{\sqrt{n}}{\tau(\hat{t}_k) \sigma_k \Delta_k} d_k } \Big)~.
\end{align} 
Using $b_d \sim \log d $ and $\frac{\sqrt{n}}{\tau(\hat{t}_k) \sigma_k \Delta_k} \gtrsim \sqrt{n}$ leads to
\begin{align}\label{testkons1}
\frac{\gua}{a_d} + b_d - \frac{\sqrt{n}}{\tau(\hat{t}_k) \sigma_k \Delta_k} d_k  \convp -\infty~. 
\end{align}
On the other hand we obtain from \eqref{lem:onedimconv}
\begin{align*}
\frac{\sqrt{n}}{\tau(\hat{t}_k) \sigma_k \Delta_k}\left( \hat{\mathbb{M}}^2_k - \Delta \mu_k^2 \right) \convd \frac{\vert \Delta \mu_k \vert}{\Delta_k}\cdot \mathcal{N}\left(0, 1\right)~,
\end{align*}
and the assertion of Theorem \ref{thm:consistent} follows.

\subsection{Proof of Theorem  \ref{lem:identification}}
Due to $\gua/a_d + b_d > 0$ we deduce 
\begin{align*}
\Prob ({ \Rd \subset \hatRd(\alpha) } )
&= \Prob \Big ( { \min_{h \in \Rd} \hatTnhdelta > \gua/a_d + b_d} \Big) \\
&\geq  \Prob \Big ({ \min_{h \in \Rd} \frac{\sqrt{n}}{\tau(\hat{t}_h) \hatsigmah \absdmuh} \big (\hatMint - \Delta_h^2 \big )  > \gua/a_d + b_d} \Big)~.
\end{align*}
Using the notation  $d_h = \dmuh^2 - \Delta_h^2$ we get
\begin{align*}
\Prob ( \Rd &\subset \hatRd(\alpha)  ) 
\geq \Prob\Big ( { \min_{h \in \Rd} \hatTnh > \gua/a_d + b_d - \min_{h \in \Rd} d_h\frac{\sqrt{n}}{\tau(\hat{t}_h) \hatsigmah \absdmuh}} \Big )  \\
&= \Prob\Big  ({ a_d \min_{h \in \Rd} \big ( \hatTnh  + b_d \big ) \geq \gua +2a_db_d -a_d\min_{h \in \Rd} d_h \frac{\sqrt{n}}{\tau(\hat{t}_h) \hatsigmah \absdmuh} }
\Big ) 
\end{align*}
By assumption \eqref{eq:mindistance} we have
\begin{align*}
n^C \min_{h \in \Rd} d_h 
&\geq n^C \min_{h \in \Rd} \left( \absdmuh  - \Delta_h \right)\Delta_{-} \convn \infty~, 
\end{align*}
which implies (as $1   \lesssim \tau(\hat{t}_h) \hatsigmah \absdmuh \lesssim 1 $)
\begin{align*}
a_db_d - a_d\min_{h \in \Rd} d_h \frac{\sqrt{n}}{\tau(\hat{t}_h) \hatsigmah \absdmuh} \convdn -\infty~.
\end{align*}
By arguments similar to those in the proofs of Section \ref{sec3} one can show that for all $x \in \mathbb{R}$
$$ \liminf_{d,n \to \infty} \Prob \Big( { a_d \min_{h \in \Rd} \big( \hatTnh  + b_d \big ) \geq x} \Big)  \geq \Pb{ - G \geq x},$$
which yields assertion \eqref{eq:RdIncludedProb}.
\noindent
For a proof of \eqref{eq:RdEqualProb} we apply Bonferroni's inequality, which gives
 \begin{align*}
\Pb{ \hatRd(\alpha) = \Rd} 
&= \Pb{ \hatRd(\alpha) \subset \Rd , \Rd \subset \hatRd(\alpha) } \\
&\geq 1 - \Pb{ \hatRd(\alpha) \nsubseteq \Rd} - \Pb{ \Rd \nsubseteq \hatRd(\alpha) }.
\end{align*}
By the arguments in the previous paragraph we have $\Pb{ \Rd \nsubseteq \hatRd(\alpha) }=o(1)$ and Theorem \ref{thm:level} gives
\begin{align*}
\limsup_{d,n \to \infty} \Pb{ \hatRd(\alpha) \nsubseteq \Rd} = 
\limsup_{d,n \to \infty} \Pb{ \max_{ h \in \Rd^c} \hatTnhdelta > \gua /a_d + b_d  } \leq \alpha~,
\end{align*}
which finishes the proof.

\section{Proofs of the results in Section 5}
To establish the bootstrap results, we recall the definition of the set $\Sd$ in \eqref{sd} and we introduce the set
\begin{align}\label{eq:Ld}
\Ld = \Big\{ \forall\, h \in \Sd^c:\; n t_h \in (K\Lhminus, K\Lhplus) \Big \}~,
\end{align}
which represents the event, that the locations of all change points are identified correctly. We need the following basic properties.  

\begin{lemma}\label{lem:basicbootstrap}
If the assumptions of Section \ref{subsec:assump} and Assumption \ref{assump:bootstrap} hold, then
\begin{enumerate}[(i)]
\item $n^C\max\limits_{h=1}^d \left\vert \hatdmuh - \dmuh  \right\vert \convp 0$ if $C<1/2$,
\item $\Pb{ \Ld^c } \lesssim n^C$ for a sufficiently small constant $C>0$.
\end{enumerate}
\end{lemma}
\begin{proof}
For assertion \textit{(i)} fix $\varepsilon > 0$ and observe
\begin{align}\label{eq:meanBSsep}
\begin{split}
&\Prob \big(  n^C \maxhd \vert \hatdmuh - \dmuh \vert > \varepsilon   \big)\\
&\leq \Prob \Big( n^C\maxhd \Big \vert \frac{1}{K\widehat{L}_h^-}\sum_{j=1}^{K\widehat{L}_h^-}Z_{j,h} - \mu_{h}^{(1)} \Big\vert > \varepsilon/2 \cap \Ld \Big )\\
&+ \Prob \Big ( n^C\maxhd \Big \vert \frac{1}{K(L-\widehat{L}_h^+)}\sum_{j=K\widehat{L}_h+1}^n Z_{j,h} - \mu_{h}^{(2)} \Big \vert > \varepsilon/2 \cap \Ld \Big )
+ o(1)~.
\end{split}
\end{align}
The first two summands of the right-hand side of \eqref{eq:meanBSsep} exhibit the same structure, so we only treat the first of them.
Note that on the event $\Ld$, it holds that
\begin{align*}
\left(Z_{j,h} - \mu_{h}^{(1)}\right)I\{j\leq K\widehat{L}_h^{-}\} = X_{j,h}I\{j\leq K\widehat{L}_h^{-}\}.
\end{align*} 
Further there exists a constant $0<C_0<1$, such that the inequalities $C_0n \leq n\bt-K \leq K\widehat{L}_h^{-} \leq n\hatth$ hold. This implies
\begin{align*}
\Prob \Big ( n^C\maxhd \Big  \vert \frac{1}{K\widehat{L}_h^-}&\sum_{j=1}^{K\widehat{L}_h^-}Z_{j,h} - \mu_{h,1} \Big  \vert > \varepsilon/2 \cap \Ld \Big  )\\
&\leq \sum_{h=1}^d \Prob \Big ( \Big  \vert \frac{1}{K\widehat{L}_h^-}\sum_{j=1}^{K\widehat{L}_h^-}X_{j,h} \Big \vert > n^{-C}\varepsilon/2 \Big )\\
&\leq \sum_{h=1}^d \Prob \Big (  \max_{k=1}^n \Big  \vert \sum_{j=1}^{k}X_{j,h} \Big \vert > C_0n^{1-C}\varepsilon/2 \Big ).
\end{align*}
Since $1-C>1/2$ we obtain by the Fuk-Nagaev inequality (see Lemma E.3 in \cite{jirak:2015a}) for sufficiently large $n$
\begin{align*}
\sum_{h=1}^d \Prob \Big (  \max_{k=1}^n \Big \vert \sum_{j=1}^{k}X_{j,h} \Big \vert > C_0n^{1-C}\varepsilon/2 \Big )
\lesssim n^D\frac{n}{ n^{(1-C)p}}
&= n^{D+(C-1)p+1} \\
&\leq n^{D - p/2 +1}
= o(1)~,
\end{align*}
where we also used $D \leq p/2 - 2$. Assertion \textit{(ii)} is shown in the proof of Theorem C.12 in \cite{jirak:2015a}. 
\end{proof}

\begin{proof}[\textbf{Proof of Lemma \ref{lem:BSstableset}}]
This is a consequence of Lemma \ref{lem:basicbootstrap} (i) and the inequality (for any $\varepsilon >0$)
\begin{align*}
\Prob \big ( { a_d \big ( \max_{h \in \Sd} B_{n,h} - b_d   \big ) > \varepsilon } \big )  
\leq \Prob \Big ( { \max_{h \in \Sd} \vert \hatdmuh \vert > n^{-1/4} } \Big)~. 
\end{align*}
\end{proof}
\subsection{Proof of Theorem \ref{thm:BSunstableset}}
For the proof we introduce the following more simple version of the bootstrap CUSUM-process $\{ \mathbb{U}_{n,h}^{(L)}(s)\}_{s \in [0,1]}$ introduced in \eqref{eq:BSCUSUM} 
\begin{align*}
\widetilde{\mathbb{U}}_{n,h}^{(L)}(s)&:= \dfrac{1}{n}\sum_{\ell=1}^{\floor{Ls}}\xi_{\ell}\widehat{V}_{\ell,h}-\dfrac{\floor{Ls}}{Ln}\sum_{\ell=1}^{L}\xi_{\ell}\widehat{V}_{\ell,h}~,
\end{align*}
where the truncation is not conducted within the blocks. We make also use of the following extra notation
\begin{align*}
\widetilde{B}_{n,h}^{\diamond} &=  \dfrac{6\sqrt{n}}{\sigmah\tau(t_h)(t_h(1-t_h))^2} \int_0^1 \widetilde{\mathbb{U}}_{n,h}^{(L)}(s)k(s,t_h)ds~,\\
\widetilde{B}_{n,h}^{\ast} &=  \dfrac{6\sqrt{n}}{\sigmah\tau(t_h)(t_h(1-t_h))^2} \int_0^1 \widetilde{\mathbb{U}}_{n,h}^{(L)}(s)k(s,\hatth)ds~,\\
B_{n,h}^{\ast} &= \dfrac{6\sqrt{n}}{\sigmah\tau(t_h)(t_h(1-t_h))^2} \int_0^1 \mathbb{U}_{n,h}^{(L)}(s)k(s,\hatth)ds~,\\
B_{n,h}^{(I)} &=  \frac{6\sqrt{n}}{\widehat{s}_h\tau(\hatth)(\hatth(1-\hatth))^2} \int_0^1 \mathbb{U}_{n,h}^{(L)}(s)k(s,\hatth)ds~.
\end{align*}
The theorem's claim is a direct consequence of the next five lemmas. We will use the notation $\PbZ{\cdot} = \Pb{ \cdot \vert \Zn }$. and frequently apply that the implication $\Pb{A_n} = o(1) \Rightarrow \PbZ{A_n} = o_{\mathbb{P}}(1)$ holds for all sequences of measurable sets $\{A_n\}_{n \in \mathbb{N}}$.

\begin{lemma}\label{lem:BSTheorem1}
The weak convergence
\begin{align*}
a_d \left( \maxhSdc \widetilde{B}_{n,h}^{\diamond} -b_d \right)
\convd
\left\{
\begin{array}{rl}
  G           &\text{if}\;\; \limd |\Sd^c|/d = 1 ~, \vspace{0.2cm} \\	
  G + \log c  &\text{if}\;\; \limd |\Sd^c|/d = c \;\text{ for } c \in (0,1)~,\vspace{0.2cm} \\
 -\infty      &\text{if}\;\; \limd |\Sd^c|/d = 0 \\
\end{array} 
\right.
\end{align*}
holds conditionally on $\Zn$.
\end{lemma}
\begin{proof}
Without loss of generality assume that the sets $\Sd$ and $\Sd^c$ are given by
\begin{align*}
\Sd^c = \{1,\dots,s\}
\;\;\;
\text{and}
\;\;\;
\Sd = \{s+1,\dots,d\}
\end{align*}
with $s= | \Sd^c |$.
Let $N=\left(N_1,\dots,N_s\right)^T$ denote a centered $s$-dimensional Gaussian vector with covariance matrix $\Sigma=(\Sigma_{ij})_{i,j=1}^d$ defined by 
\begin{align*}
\Sigma_{i,j}=\dfrac{\gamma_{i,j}\widetilde{\tau}(t_i,t_j)}{\sigma_i\sigma_j\tau(t_i)\tau(t_j)}~,
\end{align*}
where the function $\widetilde{\tau}$ is defined in \eqref{eq:widetildetau}.
Our aim is to control the (conditional) Kolmogorov-distance between $\maxhSdc \widetilde{B}_{n,h}^{\diamond}$ and $\maxhSdc N_h$. Since the random variables $\left\lbrace \xi_{\ell} \right\rbrace_{\ell \in \mathbb{N}}$ are independent, we can directly calculate the conditional covariance
\begin{align*}
\Cov_{\vert \mathcal{Z}_n}&\left(\widetilde{B}_{n,h}^{\diamond},\widetilde{B}_{n,i}^{\diamond}\right)\\
&=\dfrac{36}{\sigma_h\sigma_i\tau(t_h)\tau(t_i)(t_h(1-t_h))^2(t_i(1-t_i))^2n}\sum_{\ell=1}^{L}\widehat{V}_{\ell,h}\widehat{V}_{\ell,i}\beta_{\ell,h}\beta_{\ell,i}
\end{align*}
with the extra notation
\begin{align*}
\beta_{\ell,h} = \int_0^1 k(s,t_h)\left(I\{ \ell \leq \floor{Ls}  \} - \dfrac{\floor{Ls}}{L}\right)ds~.
\end{align*}
Let $\theta_d$ denote the distance
\begin{align*}
\theta_d 
= \max_{1\leq h,i\leq s} \left\vert \Cov_{\vert \mathcal{Z}_n}\left(\widetilde{B}_{n,h}^{\diamond},\widetilde{B}_{n,i}^{\diamond}\right) - \Sigma_{h,i} \right\vert.
\end{align*}
Using the fact that $\maxhSdc \max_{\ell=1}^L \vert \beta_{\ell,h} \vert \leq 1$ a straightforward adaption of Lemma E.8 in \cite{jirak:2015} gives
\begin{align*}
\Prob \Big  (\max_{1\leq h,i \leq s} \Big \vert \sigma_h\sigma_i\tau(t_h)\tau(t_i)\Cov_{\vert \mathcal{Z}_n}\left(\widetilde{B}_{n,h}^{\diamond},\widetilde{B}_{n,i}^{\diamond}\right) - \gamma_{h,i}\widetilde{\tau}(t_h,t_i) \Big \vert &I_{\mathcal{L}_d} > n^{-\delta} \Big )\\
&\lesssim L^{-2}n^{-C}
\end{align*}
for sufficiently small constants $C,\delta>0$, where the set $\Ld$ is defined in \eqref{eq:Ld}.
Using the lower bound $\sigma_h\sigma_i\tau(t_h)\tau(t_i) \geq \sigma_-^2\tau_-^2$ yields
\begin{align}\label{setC}
\Pb{ \mathcal{C}(\delta)^{C} } \lesssim n^{-C}
\end{align}
for the set
\begin{align*}
\mathcal{C}(\delta):= \left\{ \theta_d I_{\mathcal{L}_d} \leq n^{-\delta} \right\}.
\end{align*}
For a sufficiently small $C>0$ we have
\begin{align*}
\Pb{ \mathbb{P}_{\vert \mathcal{Z}_n}(\mathcal{L}_d^C\cup \mathcal{C}(\delta)^C)\geq n^{-C}} &\leq n^C\mean{\mathbb{P}_{\vert \mathcal{Z}_n}(\mathcal{L}_d^C\cup \mathcal{C}(\delta)^C) }\\
&= n^C \Pb{\mathcal{L}_d^C\cup \mathcal{C}(\delta)^C }=o(1)~.
\end{align*}
Now, we can derive the following upper bound
\begin{align}
\begin{split}
&\sup_{x \in \mathbb{R}} \left\vert \mathbb{P}_{\vert \mathcal{Z}_n}\left( \maxhd \widetilde{B}_{n,h}^{\diamond} \leq x \right) - \Pb{ \maxhd N_h \leq x } \right\vert \\
&\leq \sup_{x \in \mathbb{R}} \left\vert \mathbb{P}_{\vert \mathcal{Z}_n}\left( \maxhd \widetilde{B}_{n,h}^{\diamond} \leq x \right) - \Pb{ \maxhd N_h \leq x }\right\vert I_{\mathcal{L}_d\cap \mathcal{C}(\delta)} + \mathcal{O}_{P}(n^{-C})~.\label{eq:condkolm}
\end{split}
\end{align}
The identities $\gamma_{h,h}=\sigma_h^2$ and $\tilde{\tau}(t_h)=\tau^2(t_h)$ give 
$\Sigma_{h,h} = 1$ for all $h \in \setd$,
and by Lemma 3.1 in \cite{chernozhukov:2013} on the set $\mathcal{L}_d \cap \mathcal{C}(\delta)$ we have for the first summand in \eqref{eq:condkolm}
\begin{align}\label{eq:condkolm2}
\begin{split}
\sup_{x \in \mathbb{R}} \left\vert \mathbb{P}_{\vert \mathcal{Z}_n}\left( \maxhSdc \widetilde{B}_{n,h}^{\diamond} \leq x \right) - \Pb{ \maxhSdc N_h \leq x }\right\vert I_{\mathcal{L}_d\cap \mathcal{C}(\delta)}\\
\lesssim \theta_d^{1/3}\max\{1,\log(d/\theta_d \}^{2/3}I_{\mathcal{L}_d\cap \mathcal{C}(\delta)} \leq n^{-C}~.
\end{split}
\end{align}
Combining \eqref{eq:condkolm} and \eqref{eq:condkolm2} yields for the Kolmogorov distance
\begin{align}\label{kolmogorovnull}
\sup_{x \in \mathbb{R}} \left\vert \mathbb{P}_{\vert \mathcal{Z}_n}\left( \maxhSdc \widetilde{B}_{n,h}^{\diamond} \leq x \right) - \Pb{ \maxhSdc N_h \leq x } \right\vert
= o_{\mathbb{P}}(1)~.
\end{align}
The proof now follows observing the bound
$\max_{1\leq h,i \leq s} \left\vert \Sigma_{h,i} \right\vert 
\leq$ \\$\max_{1\leq h,i \leq s} \vert \rho_{h,i} \vert$
for the covariance matrix of $N$, which was derived (see the proof of Lemma \ref{lem:gaussappro})
and using adapted scaling sequences (see proof of Corollary \ref{cor:adjustedgumbel}), which yields
\begin{align*}
a_d \Big ( \maxhSdc N_h -b_d \Big )
\convd
\left\{
\begin{array}{rl}
  G           &\text{if}\;\; \limd |\Sd^c|/d = 1 ~, \vspace{0.2cm} \\	
  G + \log c  &\text{if}\;\; \limd |\Sd^c|/d = c \;\text{ for } c \in (0,1)~,\vspace{0.2cm} \\
  -\infty     &\text{if}\;\; \limd |\Sd^c|/d = 0 ~.
\end{array} 
\right.
\end{align*}
\end{proof}

\begin{lemma}\label{lem:BSTheorem2}
Conditionally on $\Zn$ it holds that
\begin{align*}
a_d \maxhSdc \widetilde{B}_{n,h}^{\diamond} - a_d \maxhSdc  \widetilde{B}_{n,h}^{\ast} \convp 0~.
\end{align*}
\end{lemma}
\begin{proof}
The covariance kernel $k$ satisfies for all $t$, $t'$, $s \in [0,1]$
$ |k(s,t) - k(s,t')| \leq 2|t-t'|~. $
Due to $(t_h(1-t_h))^2\tau(t_h) \geq \bt^4\tau_-$ we derive the bound
\begin{align*}
a_d \Big \vert  \maxhSdc \widetilde{B}_{n,h}^{\ast} - \maxhSdc \widetilde{B}_{n,h}^{\diamond} \Big \vert
&\lesssim a_d\sqrt{n} \maxhSdc \int_0^1 \dfrac{1}{\sigmah} \big \vert \widetilde{\mathbb{U}}_{n,h}^{(L)}(s)\big \vert \vert k(s,t_h)-k(s,\hatth)\vert ds\\
&\lesssim a_d\sqrt{n} \maxhSdc \max_{s \in [0,1]}\dfrac{1}{\sigmah}\big \vert \widetilde{\mathbb{U}}_{n,h}^{(L)}(s) \big \vert \maxhSdc \vert \hatth - t_h \vert~.
\end{align*}
Choosing $C$ sufficiently it follows from Corollary 3.3 in \cite{jirak:2015} that for all $\varepsilon>0$
\begin{align}\label{eq:thBS}
\mathbb{P}_{\vert \mathcal{Z}} \left( n^C\maxhSdc \vert t_h - \hatth \vert > \varepsilon \right) 
= o_{\mathbb{P}}(1)~.
\end{align}
Theorem 2.5 and 4.4 from the same reference imply the weak convergence 
\begin{align*}
e_d \left( \sqrt{n} \maxhSdc \max_{s \in [0,1]} \dfrac{1}{\sigmah}\left\vert \widetilde{\mathbb{U}}_{n,h}^{(L)}(s) \right\vert  - \dfrac{e_d}{4}  \right)
\convd G
\end{align*}
conditionally on $\Zn$ in probability with $e_d=\sqrt{2 \log (2d) }$. This yields
\begin{align*}
\PbZ{a_d\sqrt{n}n^{-C} \maxhSdc \max_{s \in [0,1]} \dfrac{1}{\sigmah}\left\vert \widetilde{\mathbb{U}}_{n,h}^{(L)}(s) \right\vert >\varepsilon} = o_{\mathbb{P}}(1)
\end{align*}
for all $\varepsilon>0$, which completes the proof of Lemma \ref{lem:BSTheorem2}.
\end{proof}

\begin{lemma}\label{lem:BSTheorem3}
Conditionally on $\Zn$ it holds that
\begin{align*}
a_d \maxhSdc \widetilde{B}_{n,h}^{\ast} - a_d \maxhSdc B_{n,h}^{\ast} \convp 0.
\end{align*}
\end{lemma}
\begin{proof}
Define $\PbZL{\cdot} = \PbZ{\cdot \cap \Ld }$ with $\Ld$ introduced in \eqref{eq:Ld}. By Lemma \ref{lem:basicbootstrap} the claim follows if we can verify
\begin{align} \label{trivbound}
\begin{split}
\mathbb{P}_{ \vert \mathcal{Z}_n}^{\mathcal{L}} \Big ( \Big \vert \maxhSdc B^{\ast}_{n,h} - \maxhSdc &\widetilde{B}_{n,h}^{\ast} \Big \vert > \dfrac{\varepsilon}{a_d} \Big )\\
&\leq ~\mathbb{P}_{ \vert \mathcal{Z}_n}^{\mathcal{L}} \Big ( \maxhSdc  \big \vert B^{\ast}_{n,h} - \widetilde{B}_{n,h}^{\ast}\big \vert  > \dfrac{\varepsilon}{a_d} \Big ) 
= o_{\mathbb{P}}(1)~.
\end{split}
\end{align}
We have the trivial bound
\begin{align*}
\Big \vert \maxhSdc B^{\ast}_{n,h} - \maxhSdc \widetilde{B}_{n,h}^{\ast}\Big \vert 
\leq \maxhSdc  \big \vert B^{\ast}_{n,h} - \widetilde{B}_{n,h}^{\ast}\big \vert~ .
\end{align*}
Assumption (T2) and (C1) imply $\sigma_h \geq \sigma_-$ and\\ $\tau(t_h)(t_h(1-t_h))^2 \geq \tau_-\bt^4$, which yields  
\begin{align*}
\left\vert B^{\ast}_{n,h} - \widetilde{B}_{n,h}^{\ast} \right\vert
&=\left\vert \dfrac{6\sqrt{n}}{\sigmah\tau(t_h)(t_h(1-t_h))^2} \int_0^1 \left(\mathbb{U}_{n,h}^{(L)}(s)-\widetilde{\mathbb{U}}_{n,h}^{(L)}(s)\right)k(s,\hatth)ds 
\right\vert \\
&\lesssim \sqrt{n}\int_0^1 \left\vert \mathbb{U}_{n,h}^{(L)}(s)-\widetilde{\mathbb{U}}_{n,h}^{(L)}(s)\right\vert k(s,\hatth)ds\\
&\lesssim \sqrt{n} \max_{s \in [0,1]} \left\vert \mathbb{U}_{n,h}^{(L)}(s)-\widetilde{\mathbb{U}}_{n,h}^{(L)}(s) \right\vert 
\leq S_{n,h}^{(1)} + S_{n,h}^{(2)}~,
\end{align*}
where we use the notation
\begin{align*}
S_{n,h}^{(1)} &= \max_{s \in [0,1]} \dfrac{1}{\sqrt{n}}\Big \vert\sum_{j=1}^{L}\xi_j\widehat{V}_{j,h}(\floor{ns}) - \sum_{j=1}^{\floor{Ls}}\xi_{j}\widehat{V}_{j,h}(n) \Big\vert,\\
S_{n,h}^{(2)} &= \max_{s \in [0,1]}  \Big\vert \dfrac{\floor{ns}}{n\sqrt{n}} - \dfrac{\floor{Ls}}{L\sqrt{n}}  \Big\vert \Big\vert \sum_{j=1}^{L}\xi_j\widehat{V}_{j,h}(n) \Big\vert
\end{align*}
In the proof of Theorem C.4 from \cite{jirak:2015a} it is shown, that
\begin{align*}
\mathbb{P}_{\vert \mathcal{Z}}^{\mathcal{L}}\Big( \maxhSdc S_{n,h}^{(1)} > \dfrac{\varepsilon}{a_d} \Big) = o_{\mathbb{P}}(1)
\end{align*}
for all $\varepsilon>0$. The second summand can be bounded by
\begin{align*}
S_{n,h}^{(2)} 
&= \max_{s \in [0,1]}  \Big\vert \dfrac{\floor{ns}}{K} - \dfrac{\floor{Ls}K}{K} \Big\vert \Big\vert \dfrac{1}{\sqrt{n}L}  \sum_{j=1}^{L}\xi_j\widehat{V}_{j,h}(n) \Big\vert\\
&\leq \Big\vert \dfrac{1}{\sqrt{n}L}  \sum_{j=1}^{L}\xi_j\widehat{V}_{j,h}(n) \Big\vert
\end{align*}
and it follows again from the above reference
\begin{align*}
\mathbb{P}_{ \vert \mathcal{Z}_n}^{\mathcal{L}}\bigg( \maxhSdc \bigg\vert \dfrac{1}{\sqrt{n}L}  \sum_{j=1}^{L}\xi_j\widehat{V}_{j,h}(n) \bigg\vert > \dfrac{\varepsilon}{a_d} \bigg) = o_{\mathbb{P}}(1)
\end{align*}
for all $\varepsilon >0$, which finishes the proof of Lemma \ref{lem:BSTheorem3}
\end{proof}

\begin{lemma}\label{lem:BSTheorem4}
The weak convergence
\begin{align}\label{eq:weakconvbnhstar}
a_d \Big( \maxhSdc B_{n,h}^{(I)} - b_d \Big)
\convd
\left\{
\begin{array}{rl}
  G           &\text{if}\;\; \limd |\Sd^c|/d = 1 ~, \vspace{0.2cm} \\	
  G + \log c  &\text{if}\;\; \limd |\Sd^c|/d = c \;\text{ for } c \in (0,1)~,\vspace{0.2cm} \\
  -\infty     &\text{if}\;\; \limd |\Sd^c|/d = 0 \\
\end{array} 
\right.
\end{align}
holds conditionally on $\Zn$.
\end{lemma}
\begin{proof}
Combining Lemmas \ref{lem:BSTheorem1}, \ref{lem:BSTheorem2} and \ref{lem:BSTheorem3} already gives assertion \eqref{eq:weakconvbnhstar} for the random variables $B_{n,h}^{\ast}$. We use the additional notation
\begin{align*}
Q^{-}_d = \min_{h \in \Sd^c} \dfrac{\widehat{s}_h\tauhatth(\hatth(1-\hatth))^2}{\sigma_h \tauth(t_h(1-t_h))^2}
\;\;\;\text{und}\;\;\;
Q^{+}_d = \max_{h \in \Sd^c} \dfrac{\widehat{s}_h\tauhatth(\hatth(1-\hatth))^2}{\sigma_h \tauth(t_h(1-t_h))^2}
\end{align*}
and the set
$\mathcal{Q}_d = \left\lbrace \left| Q^{-}_d - 1 \right| \vee \left| Q^{+}_d -1 \right| \leq \delta_d \right\rbrace $
with $\delta_d = (\log d)^{-2}$. Let $u_d(x) = x/a_d+b_d$, then we obtain for fixed $\varepsilon>0$ and $d$ sufficiently large
\begin{align*}
\PbZ{ \max_{h \in \Sd^c} B_{n,h}^{\ast} \leq u_d(x-\varepsilon)}&-\PbZ{\mathcal{Q}_d^c}
\leq \PbZ{ \max_{h \in \Sd^c} B_{n,h}^{(I)} \leq u_d(x) } \\
&\leq \PbZ{ \max_{h \in \Sd^c} B_{n,h}^{\ast} \leq u_d(x+\varepsilon)} + \PbZ{\mathcal{Q}_d^c}~.
\end{align*}
Combining Proposition 3.5 from \cite{jirak:2015} with assertion \eqref{eq:thBS} one can easily verify that $\PbZ{\mathcal{Q}_d^c} = o_{\mathbb{P}}(1)$ and the remainder of the proof can be done analogously to the proof of Theorem \ref{thm:gumbel}.
\end{proof}

\begin{lemma}
Conditionally on $\Zn$ it holds that
\begin{align*}
a_d \maxhSdc B_{n,h}^{(I)} - a_d \maxhSdc B_{n,h} \convp 0~.
\end{align*}
\end{lemma}
\begin{proof}
For fixed $\varepsilon>0$ we have
\begin{align*}
\Prob \Big( &a_d\left\vert \maxhSdc B_{n,h}^{(I)} - \maxhSdc B_{n,h} \right\vert > \varepsilon \Big)
\leq \Prob \Big( a_d\maxhSdc \left\vert  B_{n,h}^{(I)} - B_{n,h} \right\vert > \varepsilon\Big)\\
&\leq \Prob \Big( \maxhSdc \mathcal{I}\left\lbrace \vert \hatdmuh \vert \leq n^{-1/4} \right\rbrace > \varepsilon \Big)
= \Prob \Big( \minhSdc \vert \hatdmuh \vert \leq n^{-1/4} \Big)~.
\end{align*}
The proof now follows by
\begin{align*}
\Prob \Big( \minhSdc &\vert \hatdmuh \vert \leq n^{-1/4} \Big)\\
&\leq \Prob \Big( \minhSdc \vert \dmuh \vert \leq 2n^{-1/4} \Big)
+ \Prob \Big( \minhSdc \vert \dmuh \vert  - \minhSdc \vert \hatdmuh \vert > n^{-1/4}\Big)\\
&\leq \Prob \Big( K \minhSdc \vert \dmuh \vert \leq 2 \Big)
+ \Prob \Big( \maxhSdc \vert \dmuh - \hatdmuh \vert > n^{-1/4}\Big) = o(1)~,
\end{align*}
where we also used that $K=n^{1-\ell}$ and that Assumption \ref{assump:bootstrap} implies $1-\ell < 1/4$ together with
$$ \lim_{n,d\to \infty} K \minhSdc \vert \dmuh \vert = \infty~.$$
\end{proof}

\end{document}